\pdfoutput=1
\RequirePackage{ifpdf}
\ifpdf % We~are running pdfTeX in pdf mode
\documentclass[pdftex]{sigma}
\else
\documentclass{sigma}
\fi

\usepackage{mathrsfs}
\usepackage[all]{xy}
\usepackage{extarrows}

\numberwithin{equation}{section}

\def\id{{\rm id}}
\def\lr#1{\left[#1\right]}
\def\lrb#1{\left(#1\right)}
\def\lq#1{\left[#1\right]_q}

\def\K{\mathbb{K}}

\def\d{\delta}

\def\A{\mathscr{A}}
\def\aw{\mathfrak{aw}}
\def\lbb#1{\left[\kern-0.15em\left[ {#1} \right] \kern-0.15em\right]}

\def\lbc#1{\left\lceil\kern-0.3em\left\lceil {#1}
 \right\rfloor\kern-0.3em\right\rfloor}

\newtheorem{Theorem}{Theorem}[section]
\newtheorem*{Theorem*}{Theorem}
\newtheorem{Corollary}[Theorem]{Corollary}
\newtheorem{Lemma}[Theorem]{Lemma}
\newtheorem{Proposition}[Theorem]{Proposition}
 { \theoremstyle{definition}
\newtheorem{Definition}[Theorem]{Definition}

\newtheorem{Example}[Theorem]{Example}
\newtheorem{Remark}[Theorem]{Remark} }

\begin{document}
\allowdisplaybreaks

\newcommand{\arXivNumber}{2407.10404}

\renewcommand{\PaperNumber}{115}

\FirstPageHeading

\ShortArticleName{On the Higher-Rank Askey--Wilson Algebras}

\ArticleName{On the Higher-Rank Askey--Wilson Algebras}

\Author{Wanxia WANG and Shilin YANG}
\AuthorNameForHeading{W.~Wang and S.~Yang}
\Address{School of Mathematics, Statistics and Mechanics, Beijing University of Technology, P.R.~China}
\Email{\href{mailto:18839136189@163.com}{18839136189@163.com}, \href{mailto:slyang@bjut.edu.cn}{slyang@bjut.edu.cn}}

\ArticleDates{Received July 16, 2024, in final form December 15, 2024; Published online December 28, 2024}

\Abstract{In the paper, the algebra $\mathscr{A}(n)$, which is generated by an upper triangular generating matrix with triple relations, is introduced.
It is shown that there exists an isomorphism between the algebra $\mathscr{A}(n)$ and the higher-rank Askey--Wilson algebra $\mathfrak{aw}(n)$ introduced by Cramp\'{e} et al. Furthermore, we establish a series of automorphisms of $\mathscr{A}(n)$, which satisfy braid group relations and coincide with those in $\mathfrak{aw}(n)$.}


\Classification{20F36; 33D45; 81R12}

\section{Introduction}
The study of algebraic structures has been a cornerstone of mathematical research, providing profound insights and applications in various fields.
The Askey--Wilson algebras ${\rm AW}_q$, introduced by Zhedanov in \cite{Zhe(1991)}, have significant implications in mathematical physics, quantum groups and orthogonal polynomials.
For instance, Koelink and Stokman \cite{KS(2001)} studied the solutions of the Askey--Wilson $q$-difference equation, which is related to Askey--Wilson algebras, offering valuable insights into the associated functions and transformations.
Cramp\'{e} and Gaboriaud et al. \cite{CGVZ(2020)} investigated the connections between Askey--Wilson algebras and the universal $R$-matrix of $U_q(\mathfrak{sl}_2)$, providing a new perspective for understanding these algebras.
Terwilliger and Vidunas \cite{TV(2004)} studied the relationships between the Askey--Wilson algebras and the corresponding Leonard pairs, to highlight the importance of these algebras in representation theory.
Terwilliger in \cite{Ter(2011)} introduced the universal Askey--Wilson algebra $\Delta_q$, which are generalization of ${\rm AW}_q$.
We point out that Huang in \cite{Huang(2015),Huang(2021)} gave the classification of finite-dimensional irreducible modules of $\Delta_q$.
Lavrenov \cite{Lav(1997)} demonstrated their applications in integrable system, to show strongly their significance in theoretical physics.
More relevant references can be found in~\cite{CFGPRV(2021),CVZ(2021),GW(2023),Koor(2007),Koor(2008)}.

As a natural extension, it is interesting to address the definition of the higher-rank
Askey--Wilson algebras and explore their structures and properties, which can provide more complex descriptions with richer symmetries and representation theories.
These descriptions can demonstrate advantages in solving complex physical systems and high-dimensional spaces.
Baseilhac and Koizumi \cite{BK(2005)} have explored applications of a deformed analogue of Onsager's symmetry,
which can be viewed as an analog of higher-rank Askey--Wilson algebras, in physical systems such as the $XXZ$ open chain.
De Bie and van de Vijver \cite{DV(2020)} investigated the discrete realization of the higher rank Racah algebras.
The readers can be referred to \cite{CL(2022),DDV(2020),De(2019),PW(2017)} for more applications.

To the best of our knowledge, Cramp\'{e} et al.\ \cite{CFPR(2023)} recently introduced definition of the higher-rank Askey--Wilson algebras explicitly.
This definition is strongly general and provides a broader framework to understand these algebras.
In the same paper, the authors explored Casimir elements and a series of automorphisms, which enjoy the braid group relations.
We hope to find more intuitive definition of the higher-rank Askey--Wilson algebras by the generators and relations, and furthermore to recover the above braided group automorphisms.
Fortunately, these higher-rank Askey--Wilson algebras can be defined equivalently by the generators located in upper-triangular matrix with triple relations.
It is remarked that the proof presented in \cite{CFPR(2023)} is comprehensive, but relies heavily on computations of computer algebra. This sometimes may obscure those people who are not familiar with the computer algebra tools to understand their proofs.
In the present paper, our viewpoint provides more accessible methods for studying these algebras, which is not merely to replicate the known result.
It is also useful to understand the higher-rank Askey--Wilson algebras for a wider audience.

Now we prepare to give the outline of the paper.
In Section \ref{sect:notations}, the crucial notations and fundamental concepts are introduced, and we define the algebras $\A(n)$ by an upper triangular generating matrix $\A$ with several relations.
Several basic properties for these algebras are described.
In Section \ref{sect:relations}, it is shown that $\A(n)$ is isomorphic to the higher-rank Askey--Wilson algebra $\aw(n)$ given by Cramp\'{e} et al.\ in \cite{CFPR(2023)}.
In Section \ref{sect:automorphism}, we establish a series of automorphisms of $\A(n)$, which coincide with those automorphisms of $\aw(n)$ given in \cite{CFPR(2023)}.
In Section \ref{braid}, we prove that those automorphisms established in Section~\ref{sect:automorphism} satisfy the braid group relations independently in our approach.
We provide the detailed proofs of these results.
The exploration may help us to understand $\aw(n)$ in a more intuitive way, also to point out what we can do in the future such as constructing the PBW basis of the algebra $\A(n)$ or equivalently $\aw(n)$.

\section[The algebra A(n)]{The algebra $\boldsymbol{\A(n)}$}\label{sect:notations}
In this section, we list some notations to use in the sequel, then we give the definition of the algebra $\A(n)$.

Always assume that $\K$ is an algebraically closed field with $\mathrm{char} \K=0$,
and $q \in \K$ with $q^4 \neq 1$. Let $\mathbb{Z}$ be the ring of integers, $\mathbb{N}$ the set of the nonnegative integers and $\mathbb{N}^*=\mathbb{N} \setminus\{0\}.$
For $i, j \in \mathbb{Z}$, we denote by $\lbb{i,j}$ the set $\{i, i+1, \dots, j\}$, where $i<j$.

Suppose that $A$, $B$ are the operators and set
\begin{equation*}
[A, B] \overset{\bigtriangleup}{=} AB-BA, \qquad \lq{A,B}\overset{\bigtriangleup}{=} \frac{qAB-q^{-1}BA}{q-q^{-1}}.
\end{equation*}
Then, for given operators $A$, $B$, $C$, one easily sees that
\begin{gather} %\label{eq:qjacobi2}
[A,B]+[B,A]=0, \label{eq:jacobi}\\
 [A,B]=\frac{q-q^{-1}}{q+q^{-1}}([A,B]_q-[B,A]_q), \label{eq:qjacobi0}\\
[[A,B]_q,C]_q-[A,[B,C]_q]_q =\frac{1}{\bigl(q-q^{-1}\bigr)^2} [B,[C,A]], \label{eq:qjacobi1}\\
[A,[B,C]_q]_q-[B,[C,A]_q]_q =\frac{q+q^{-1}}{q-q^{-1}} [[A,B],C]_{q^2},\\
[A,[B,C]_q]+[C,[A,B]_q]+ [B,[C,A]_{q}]=0. \label{eq:qjacobi3}
\end{gather}
Also, if $[A,B]=0$, then
\begin{gather}
 [A,B]_q=AB, \qquad [A,CB]_q=[A,C]_qB, \qquad [A,BC]_q=B[A,C]_q, \label{equ:comm1}\\
 [A, [C,B]_q]_q=[[A,C]_q,B]_q, \qquad [A,[B,C]_q]_q=[B,[A,C]_q]_q. \label{equ:comm2}
\end{gather}
In particular, if $[A,B]=[A,C]=0$, then
\begin{eqnarray}\label{equ:comm3}
[A, [B,C]_q]_q=A[B,C]_q.
\end{eqnarray}
We denote by \smash{$\A \overset{\bigtriangleup}{=} (a_{i,j})_{n\times n}$} an upper triangular generating matrix: $a_{i,j}=0$ whenever $i>j$.
Now, for convenience we partition the matrix into blocks in the following way:
\begin{itemize}\itemsep=0pt
 \item the block $\A_{11}(i-1,i-1)$: a $(i-1)\times (i-1)$-submatrix consisting of the first $i-1$ rows and columns;
 \item the block $\A_{12}(i-1,j-i)$: a $(i-1)\times (j-i)$-submatrix consisting of the rows from the $1$-th to the $(i-1)$-th and columns from the $i$-th to the $(j-1)$-th;
 \item the block $\A_{14}(i-1,n-j)$: a $(i-1)\times (n-j)$-submatrix consisting of the rows from the $1$-th to the $(i-1)$-th and columns from the $j$-th to the $n$-th;
 \item the block $\A_{32}(j-i-1,j-i-1)$: a $(j-i-1)\times (j-i-1)$-submatrix consisting of the rows from the $(i+1)$-th to the $(j-1)$-th and columns from the $(i+1)$-th to the $(j-1)$-th;
 \item the block $\A_{34}(j-i,n-j)$: a $(j-i)\times (n-j)$-submatrix consisting of the rows from the~$(i+1)$-th to the $j$-th and columns from the $(j+1)$-th to the $n$-th;
 \item the block $ \A_{44}(n-j,n-j)$: a $(n-j)\times (n-j)$-submatrix consisting of the rows from the~$(j+1)$-th to the $n$-th and columns from the $(j+1)$-th to the $n$-th.
\end{itemize}
Obviously, the $i$-th row and $j$-th column do not belong to any of the blocks mentioned above, which are naturally partitioned.

According to the partition, the matrix $\A$ is now represented as
\begin{equation}\label{eqn-A}
 \left(
 \begin{array}{c|c|c|c}
 \A_{11}(i-1,i-1) & \A_{12}(i-1,j-i) &
 %\A_{13}(i-1,1)
 \begin{array}{c} a_{1,j}\\ \vdots \\ a_{i-1, j} \end{array}
& \A_{14}(i-1,n-j)\\
 \hline
 \begin{array}{ccc} & & \end{array} &
 \begin{array}{lcl} a_{i,i} \quad & \quad \cdots \quad & \quad a_{i,j-1} \end{array}
 & a_{i, j} &
 \begin{array}{ccc} a_{i, j+1} & \cdots & a_{i, n} \end{array}
 \\
 \hline
 \begin{array}{ccc} & & \\
 & & \\
 & & \end{array} & \begin{array}{cc} \begin{array}{c}
 0\\
 \vdots
 \end{array}&
 \A_{32}(j-i-1,j-i-1)\\ 0&\cdots\cdots \qquad 0\end{array}
 &
 \begin{array}{c} a_{i+1, j}\\ \vdots \\ a_{j, j}\end{array}
 & \A_{34}(j-i,n-j)\\
 \hline
 \begin{array}{ccc} & & \\
 & & \\
 & & \end{array} & \begin{array}{ccc} & & \\
 & & \\
 & & \end{array} & \begin{array}{ccc} & & \\
 & & \\
 & & \end{array} & \A_{44}(n-j,n-j)
 \end{array}
 \right),
\end{equation}
where
\begin{gather*}
\A_{11}(i-1,i-1)=
 \begin{pmatrix}
 a_{1, 1} & \cdots & a_{1, i-1}\\
 & \ddots & \vdots \\
 & & a_{i-1, i-1}
 \end{pmatrix}
,
 \\
 \A_{12}(i-1,j-i)=
 \begin{pmatrix}
 a_{1, i} & \cdots & a_{1, j-1}\\
 \vdots & \vdots & \vdots \\
 a_{i-1, i} & \cdots & a_{i-1, j-1}
 \end{pmatrix}
, \\
\A_{14}(i-1,n-j)=
 \begin{pmatrix}
 a_{1, j} & \cdots & a_{1, n}\\
 \vdots & \vdots & \vdots \\
 a_{i-1, j} & \cdots & a_{i-1, n}
 \end{pmatrix},
 \\
\A_{32}(j-i-1,j-i-1)=
 \begin{pmatrix}
 a_{i+1, i+1} & \cdots & a_{i+1, j-1}\\
 & \ddots & \vdots \\
 & & a_{j-1, j-1}
 \end{pmatrix}
,\\
 \A_{34}(j-i,n-j)=
 \begin{pmatrix}
 a_{i+1, j+1} & \cdots & a_{i+1, n}\\
 \vdots & \vdots & \vdots \\
 a_{j, j+1} & \cdots & a_{j, n}\\
 \end{pmatrix}
,
 \\
 \A_{44}(n-j,n-j)=
 \begin{pmatrix}
 a_{j+1, j+1} & \cdots & a_{j+1, n}\\
 & \ddots & \vdots \\
 & & a_{n,n}
 \end{pmatrix}.
 \end{gather*}

The following notations are assigned:
\begin{itemize}\itemsep=0pt
\item For generators $a_{i, j} \in \A$ and $a_{k, l} \in \A_{34}(j-i, n-j)$, we choose the entries in the $i$-th, $k$-th, and $(j+1)$-th rows, and the $(k-1)$-th, $j$-th, and $l$-th columns, where $i < k \leq j < l$.
 Hence, we yield the submatrix
\begin{equation}\label{eqn-A-1}
\A_{i,k,j+1}^{k-1,j,l}=
\begin{pmatrix}
 a_{i,k-1}& a_{i,j} & a_{i,l}\\
 & a_{k,j} & a_{k, l}\\
 & & a_{j+1,l}
 \end{pmatrix}
,
\end{equation}
and let
\begin{gather}
f\bigl(\A_{i,k,j+1}^{k-1,j,l}\bigr)\overset{\bigtriangleup}{=} a_{i, j}+ a_{k, l} \lrb{a_{i, k-1}a_{j+1, l} + a_{k, j }a_{i, l}}-\lrb{a_{i, k-1}a_{k, j}+a_{i, l}a_{j+1, l}}, \label{eqn:natural}\\
g\bigl(\A_{i,k,j+1}^{k-1,j,l}\bigr)\overset{\bigtriangleup}{=} a_{k, l}+ a_{i, j}\lrb{a_{i, k-1}a_{j+1, l} + a_{k, j }a_{i, l}}-\lrb{a_{i, k-1}a_{i, l}+a_{k, j}a_{j+1, l}}. \label{eqn:sharp}
\end{gather}
For example, if we choose
\begin{equation*}
\A_{1,2,4}^{1,3,4}=
 \begin{pmatrix}
 a_{1,1}&a_{1,3} & a_{1,4}\\
 & a_{2,3} & a_{2, 4}\\
 & & a_{4,4}
 \end{pmatrix}
,
\end{equation*}
then
\begin{gather*}
f\bigl(\A_{1,2,4}^{1,3,4}\bigr)= a_{1, 3}+ a_{2, 4} \lrb{a_{1, 1}a_{4, 4} + a_{2, 3 }a_{1, 4}}-\lrb{a_{1, 1}a_{2, 3}+a_{1, 4}a_{4, 4}}, \\
g\bigl(\A_{1,2,4}^{1,3,4}\bigr)= a_{2, 4}+ a_{1, 3}\lrb{a_{1, 1}a_{4, 4} + a_{2, 3 }a_{1, 4}}-\lrb{a_{1, 1}a_{1, 4}+a_{2, 3}a_{4, 4}}.
\end{gather*}
 \item For generators $a_{i, j} \in \A$ and $a_{k, l} \in \A_{34}(j-i, n-j)$, we choose the entries in the $i$-th, $k$-th, and $(j+1)$-th rows, and the $j$-th, $l$-th, and $m$-th columns, where $i<k\leq j<l<m\leq n$.
 Hence we yield the submatrix
\begin{equation}\label{eqn-A-2}
\A_{i,k,j+1}^{j,l,m}=
 \begin{pmatrix}
a_{i,j}& a_{i,l}& a_{i,m}\\
a_{k,j} & a_{k,l} & a_{k, m}\\
& a_{j+1,l} & a_{j+1,m }
 \end{pmatrix}
,
\end{equation}
and let
\begin{gather}
{\det}_q\bigl(\A_{i,k,j+1}^{j,l,m}\bigr)
 \overset{\bigtriangleup}{=} [ [a_{i, j}, a_{k, l} ]_q, a_{j+1, m} ]_q+ [a_{i,l},a_{k,m} ]_q+ [ [a_{i, m}, a_{k, j} ]_q, a_{j+1, l} ]_q \nonumber\\
\phantom{{\det}_q\bigl(\A_{i,k,j+1}^{j,l,m}\bigr)
 \overset{\bigtriangleup}{=}}{}
- [ [ a_{i, l }, a_{k,j} ]_q, a_{j+1, m} ]_q - [ [ a_{i, j}, a_{k, m} ]_q, a_{j+1, l} ]_q - [a_{i, m}, a_{k, l} ]_q, \label{eqn:D-1}\\
 {\det}^q\bigl(\A_{i,k,j+1}^{j,l,m}\bigr)
 \overset{\bigtriangleup}{=} [ [a_{j+1, m}, a_{k, l} ]_q, a_{i, j} ]_q+ [a_{k,m}, a_{i,l} ]_q + [ [a_{j+1, l}, a_{k, j} ]_q, a_{i, m} ]_q \nonumber\\
\phantom{ {\det}^q\bigl(\A_{i,k,j+1}^{j,l,m}\bigr)
 \overset{\bigtriangleup}{=}}{} - [ [ a_{j+1, m}, a_{k, j} ]_q, a_{i, l} ]_q - [ [a_{j+1, l}, a_{k, m} ]_q, a_{i, j} ]_q- [a_{k, l}, a_{i, m} ]_q. \label{eqn:D-2}
\end{gather}
For example, if we choose
\begin{equation*}
\A_{1,2,4}^{3,4,5}=
 \begin{pmatrix}
a_{1,3}& a_{1,4}& a_{1,5}\\
a_{2,3} & a_{2,4} & a_{2, 5}\\
& a_{4,4} & a_{4,5 }
 \end{pmatrix}
,
\end{equation*}
then
\begin{gather*}
 {\det}_q\bigl(\A_{1,2,4}^{3,4,5}\bigr)
= [ [a_{1, 3}, a_{2, 4} ]_q, a_{4, 5} ]_q+ [a_{1,4},a_{2,5} ]_q+ [ [a_{1, 5}, a_{2, 3} ]_q, a_{4, 4} ]_q \\
\phantom{ {\det}_q(\A_{1,2,4}^{3,4,5})
=}{}
- [ [ a_{1, 4 }, a_{2,3} ]_q, a_{4, 5} ]_q - [ [ a_{1, 3}, a_{2, 5} ]_q, a_{4, 4} ]_q - [a_{1, 5}, a_{2, 4} ]_q, \\
 {\det}^q\bigl(\A_{1,2,4}^{3,4,5}\bigr)
= [ [a_{4, 5}, a_{2, 4} ]_q, a_{1, 3} ]_q+ [a_{2,5}, a_{1,4} ]_q + [ [a_{4, 4}, a_{2, 3} ]_q, a_{1, 5} ]_q \\
 \phantom{ {\det}^q(\A_{1,2,4}^{3,4,5})
= }{}- [ [ a_{4, 5}, a_{2, 3} ]_q, a_{1, 4} ]_q - [ [a_{4, 4}, a_{2, 5} ]_q, a_{1, 3} ]_q- [a_{2, 4}, a_{1, 5} ]_q.
\end{gather*}
\end{itemize}

Now we can introduce the definition of the algebra $\A(n)$ explicitly.
\begin{Definition}\label{defn:A(n)}
The $ \mathbb{K}$-algebra $\A(n)$ is an associative algebra generated by the upper generating matrix $\A$ in $\eqref{eqn-A}$, subjecting to the following defining relations:
\begin{itemize}\itemsep=0pt
\item[{\rm R1:}] The entries $a_{i,j}$ commutates with all entries in $\A$ except in $\A_{12}\lrb{i-1,j-i}$ and $\A_{34}(j-i,\allowbreak n-j).$

\item[{\rm R2:}] The entries of any submatrix \smash{$\A_{i,k,j+1}^{k-1,j,l}$} as $\eqref{eqn-A-1}$ enjoy the relations
\begin{gather}\label{rela-2}
[a_{k, l}, [a_{i, j}, a_{k, l}]_q]_q=f\bigl(\A_{i,k,j+1}^{k-1,j,l}\bigr), \qquad [a_{i, j}, [a_{k, l}, a_{i, j}]_q ]_q=g\bigl(\A_{i,k,j+1}^{k-1,j,l}\bigr).
\end{gather}

\item[{\rm R3:}] The entries of any submatrix \smash{$\A_{i,k,j+1}^{j,l,m}$} as $\eqref{eqn-A-2}$ enjoy the relations
\begin{gather}\label{rela-3}
{\det}_q\bigl(\A_{i,k,j+1}^{j,l,m}\bigr)={\det}^q\bigl(\A_{i,k,j+1}^{j,l,m}\bigr)=0.
\end{gather}
\end{itemize}
\end{Definition}
In this situation, we say that the algebra $\A(n)$ is generated by the matrix $\A$ with the relations (R1)--(R3).

\begin{Remark}\label{remark}
For $n=2$, relations ${\rm (R2)}$ and ${\rm (R3)}$ are automatically disappeared.
Hence, the algebra $\A(2)$ is generated by $a_{1,1}$, $ a_{1,2}$, $a_{2,2}$ with the relations
\begin{equation*}
[a_{1,1},a_{1,2}]=[a_{1,1},a_{2,2}]=[a_{1,2},a_{2,2}]=0.
\end{equation*}
In other words, $\A(2)\cong \K[x_1, x_2, x_3].$

For $n=3$, the relation ${\rm (R3)}$ is disappeared.
Hence, the algebra $\A(3)$ is generated by the entries in the matrix
\begin{equation*}
\begin{pmatrix}
a_{1,1}& a_{1,2}&a_{1,3}\\
& a_{2,2}&a_{2,3}\\
& &a_{3,3}
\end{pmatrix}
,
\end{equation*}
with the relations {\rm (R1)} and {\rm (R2)}.
The relation {\rm (R1)} indicates that $a_{1,1}$, $a_{2,2}$, $a_{3,3}$, $a_{1,3}$ are central and the relation {\rm (R2)} says that
\begin{gather*}
\bigl(q^2+q^{-2}\bigr)a_{2,3}a_{1,2}a_{2,3}-a_{2,3}^2a_{1,2}-a_{1,2}a_{2,3}^2 \\
\qquad=\bigl(q-q^{-1}\bigr)^2 (a_{1, 2}+\lrb{a_{1, 1}a_{3, 3} + a_{2, 2 }a_{1, 3}} a_{2, 3} -\lrb{a_{1, 1}a_{2, 2}+a_{1, 3}a_{3, 3}}),\\
\bigl(q^2+q^{-2}\bigr)a_{1,2}a_{2,3}a_{1,2}-a_{1,2}^2a_{2,3}-a_{2,3}a_{1,2}^2 \\
\qquad=\bigl(q-q^{-1}\bigr)^2 (a_{2, 3}+ \lrb{a_{1, 1}a_{3, 3} + a_{2, 2 }a_{1, 3}}a_{1, 2}-\lrb{a_{1, 1}a_{1, 3}+a_{2, 2}a_{3, 3}}).
\end{gather*}
\end{Remark}
Recall that the Askey--Wilson algebra ${\rm AW}_q$, introduced by Zhedanov in \cite{Zhe(1991)}, is presented with generators $K_0$, $K_1$ and relations
\begin{gather*}
\bigl(q^2+q^{-2}\bigr)K_1K_0K_1-K_1^2K_0-K_0K_1^2=BK_1+C_0K_0+D_0,\\
\bigl(q^2+q^{-2}\bigr)K_0K_1K_0-K_0^2K_1-K_1K_0^2=BK_0+C_1K_1+D_1,
\end{gather*}
where $B$, $C_0$, $C_1$, $D_0$, $D_1$ are the structural constants of the algebra.
If we set central elements
\begin{alignat*}{3}
& B=\bigl(q-q^{-1}\bigr)^2\lrb{a_{1, 1}a_{3, 3} + a_{2, 2 }a_{1, 3}}, \qquad && C_0=C_1=\bigl(q-q^{-1}\bigr)^2, & \\
& D_0=-\bigl(q-q^{-1}\bigr)^2\lrb{a_{1, 1}a_{2, 2}+a_{1, 3}a_{3, 3}}, \qquad&&
D_1=-\bigl(q-q^{-1}\bigr)^2\lrb{a_{1, 1}a_{1, 3}+a_{2, 2}a_{3, 3}}&
\end{alignat*}
are in the field $\K$, then the algebra $\A(3)$ is just a particular Askey--Wilson algebra ${\rm AW}_q$.

\begin{Example}\label{Example}
The algebra $\A(4)$ is generated by $a_{i,j}$ $( 1 \leq i \leq j \leq 4)$ with the relations as follows$:$
\begin{itemize}\itemsep=0pt
\item[\rm{R1:}] $a_{1,1}$, $a_{2,2}$, $a_{3,3}$, $a_{4,4}$, $a_{1,4}$ are in the center of $\A\lrb{4}$ and
\[
\lr{a_{1, 2}, a_{1, 3}}=\lr{a_{1, 2}, a_{3, 4}} =\lr{a_{1, 3}, a_{2, 3}}=\lr{a_{2, 3}, a_{2, 4}}=\lr{a_{2,4}, a_{3, 4}}=0;
 \]
\item[\rm{R2:}]
\begin{gather*}
[a_{2, 3}, [a_{1, 2}, a_{2, 3}]_q]_q=f\bigl(\A^{1,2,3}_{1,2,3}\bigr), \qquad
 [a_{1, 2}, [a_{2, 3}, a_{1, 2}]_q]_q=g\bigl(\A^{1,2,3}_{1,2,3}\bigr),\\
 [a_{3, 4}, [a_{2, 3}, a_{3, 4}]_q]_q=f\bigl(\A^{2,3,4}_{2,3,4}\bigr), \qquad
 [a_{2, 3}, [a_{3, 4}, a_{2, 3}]_q]_q=g\bigl(\A^{2,3,4}_{2,3,4}\bigr),\\
[a_{2, 4}, [a_{1, 3}, a_{2, 4}]_q]_q=f\bigl(\A^{1,3,4}_{1,2,4}\bigr), \qquad
 [a_{1, 3}, [a_{2, 4}, a_{1, 3}]_q]_q=g \bigl(\A^{1,3,4}_{1,2,4}\bigr),\\
[a_{2, 4}, [a_{1, 2}, a_{2, 4}]_q]_q=f\bigl(\A^{1,2,4}_{1,2,3}\bigr), \qquad
[a_{1, 2}, [a_{2, 4}, a_{1, 2}]_q]_q=g \bigl(\A^{1,2,4}_{1,2,3}\bigr),\\
[a_{3, 4}, [a_{1, 3}, a_{3, 4}]_q]_q=f\bigl(\A^{2,3,4}_{1,3,4}\bigr), \qquad
 [a_{1, 3}, [a_{3, 4}, a_{1, 3}]_q]_q=g \bigl(\A^{2,3,4}_{1,3,4}\bigr);
\end{gather*}
\item[\rm{R3:}]
\smash{$
{\det}_q\bigl(\A_{1,2,3}^{2,3,4}\bigr) ={\det}^q\bigl(\A_{1,2,3}^{2,3,4}\bigr)=0$}.
\end{itemize}
\end{Example}

The following relations are used to prove the main results.
\begin{Lemma}\label{lem:eqn-1}
For the submatrix \smash{$\A_{j,k,i+1}^{i,l,m}$} as $\eqref{eqn-A-2}$, the following relations hold in $\A(n)$:
\begin{gather}% \tag{4.a} \label{eqn:4a}\tag{5.a} \label{eqn:5a}\tag{5.b} \label{eqn:5b}
 [\lr{a_{j,l}, a_{k,m}}_q, a_{j,i}]_q +a_{j,k-1}a_{i+1,l}a_{j,m} +\lr{a_{k,l}, a_{i+1,m}}_q \nonumber \\
 \qquad= a_{j,m}\lr{a_{k,l}, a_{j,i}}_q+a_{j,k-1}\lr{a_{j,l}, a_{i+1,m}}_q+a_{i+1,l}a_{k,m}, \tag{1.a} \label{eqn:1a}\\
 [\lr{a_{j,i}, a_{k,m}}_q, a_{j,l}]_q+a_{j,k-1}a_{i+1,l}a_{j,m}+\lr{a_{i+1,m}; a_{k,l}}_q \nonumber \\
 \qquad = a_{j,k-1}\lr{a_{i+1,m}, a_{j,l}}_q +a_{j,m}\lr{a_{j,i}, a_{k,l}}_q+a_{i+1,l}a_{k,m},\tag{1.b} \label{eqn:1b} \\
 [\lr{a_{k,m}, a_{j,l}}_q, a_{i+1,m}]_q+a_{k,i}a_{l+1,m}a_{j,m} +\lr{a_{k,l}, a_{j,i}}_q \nonumber \\
 \qquad = a_{j,m}\lr{a_{k,l}, a_{i+1,m}}_q+a_{l+1,m}\lr{a_{k,m}, a_{j,i}}_q +a_{k,i}a_{j,l}, \tag{2.a} \label{eqn:2a}\\
 [\lq{a_{i+1,m},a_{j,l}}, a_{k,m}]_q+a_{k,i}a_{l+1,m}a_{j,m} +\lr{a_{j,i}, a_{k,l}}_q\nonumber \\
 \qquad = \lr{a_{i+1,m}, a_{k,l}a_{j,m}}_q +a_{l+1,m}\lr{a_{j,i}, a_{k,m}}_q+a_{k,i}a_{j,l}, \tag{2.b} \label{eqn:2b}\\
 [\lr{a_{k,l}, a_{j,i}}_q, a_{k,m}]_q +a_{j,k-1}a_{k,i}a_{l+1,m}+\lr{a_{j,l}, a_{i+1,m}}_q \nonumber \\
 \qquad= a_{k,i}\lr{a_{j,l}, a_{k,m}}_q+a_{j,k-1}\lr{a_{k,l}, a_{i+1,m}}_q+a_{l+1,m}a_{j,i}, \tag{3.a} \label{eqn:3a}\\
 [\lr{a_{k,m}, a_{j,i}}_q, a_{k,l}]_q+a_{j,k-1}a_{k,i}a_{l+1,m}+\lr{a_{i+1,m}, a_{j,l}}_q \nonumber \\
 \qquad= a_{k,i}\lr{a_{k,m}, a_{j,l}}_q+a_{j,k-1}\lr{a_{i+1,m}, a_{k,l}}_q+a_{l+1,m}a_{j,i}, \tag{3.b} \label{eqn:3b}\\
 [\lq{a_{j,l},a_{i+1,m}}, a_{k,l}]_q+a_{j,k-1}a_{i+1,l}a_{l+1,m} +\lr{a_{j,i}, a_{k,m}}_q \nonumber \\
 \qquad= a_{i+1,l}\lr{a_{j,l}, a_{k,m}}_q +a_{l+1,m}\lr{a_{j,i}, a_{k,l}}_q +a_{j,k-1}a_{i+1,m},\tag{4.a} \\
 [\lq{a_{k,l},a_{i+1,m}}, a_{j,l}]_q +a_{j,k-1}a_{i+1,l}a_{l+1,m} +\lr{a_{k,m}, a_{j,i}}_q \nonumber \\
 \qquad= a_{i+1,l}\lr{a_{k,m}, a_{j,l}}_q +a_{l+1,m}\lr{a_{k,l}, a_{j,i}}_q +a_{j,k-1}a_{i+1,m}, \tag{4.b} \label{eqn:4b} \\
 [\lr{a_{i,l}, a_{j,i}}_q, a_{k,l}]_q+a_{k,i-1}a_{i+1,l}a_{j,l}+\lr{a_{j,i-1}, a_{k,i}}_q \nonumber \\
 \qquad = a_{i+1,l}\lr{a_{j,i-1}, a_{k,l}}_q +a_{j,l}\lr{a_{i,l}, a_{k,i}}_q+a_{k,i-1}a_{j,i}, \tag{5.a} \\
 [\lr{a_{k,l}, a_{j,i}}_q, a_{i,l}]_q+a_{k,i-1}a_{i+1,l}a_{j,l}+\lr{a_{k,i}, a_{j,i-1}}_q \nonumber \\
 \qquad = a_{i+1,l}\lr{a_{k,l}, a_{j,i-1}}_q+a_{j,l}\lr{a_{k,i}, a_{j,k}}_q+a_{k,i-1}a_{j,i}, \tag{5.b} \\
 [\lr{a_{i,l}, a_{j,i}}_q, a_{i,m}]_q+a_{j,i-1}a_{i,i}a_{l+1,m} +\lr{a_{j,l}, a_{i+1,m} }_q\nonumber \\
 \qquad = a_{i,i}\lr{a_{j,l}, a_{i,m}}_q +a_{j,i-1}\lr{a_{i,l}, a_{i+1,m}}_q+a_{j,i}a_{l+1,m}, \tag{6.a} \label{eqn:6a}\\
 [\lr{a_{i,m}, a_{j,i}}_q, a_{i,l}]_q+a_{j,i-1}a_{i,i}a_{l+1,m} +\lr{a_{i+1,m}, a_{j,l}}_q \nonumber \\
 \qquad = a_{i,i}\lr{a_{i,m}, a_{j,l}}_q+a_{j,i-1}\lr{a_{i+1,m}, a_{i,l}}_q+a_{j,i}a_{l+1,m}. \tag{6.b} \label{eqn:6b}
\end{gather}
\end{Lemma}
\begin{proof}
We only verify the relation \eqref{eqn:1a} here.
The proofs of the other relations are in a similar way.
We have
\begin{gather*}
 [\lr{a_{j,l},a_{k,m}}_q,a_{j,i}]_q +a_{j,k-1}a_{i+1,l}a_{j,m} +\lr{a_{k,l},a_{i+1,m}}_q \\
 \qquad {}-a_{j,m}\lr{a_{k,l},a_{j,i}}_q-a_{j,k-1}\lr{a_{j,l},a_{i+1,m}}_q-a_{i+1,l}a_{k,m} \\
\qquad\qquad \xlongequal{\quad} [\lr{a_{j,l},a_{k,m}}_q-a_{j,m}a_{k,l},a_{j,i}]_q +a_{j,k-1}a_{i+1,l}a_{j,m} -a_{j,k-1}\lr{a_{j,l},a_{i+1,m}}_q\\
 \phantom{\qquad\qquad\xlongequal{\quad}}{} +\lr{a_{k,l},a_{i+1,m}}_q -a_{i+1,l}a_{k,m} \\
\qquad\qquad \xlongequal{\eqref{rela-3}} {}-[[a_{j,i},\lr{a_{k,l},a_{i+1,m}}_q ]_q -a_{i+1,l}\lr{a_{j,i},a_{k,m} }_q -a_{k,i}\lr{a_{j,l},a_{i+1,m}}_q\\
 \phantom{\qquad\qquad\xlongequal{\qquad}}{}+a_{k,i}a_{i+1,l}a_{j,m},
 a_{j,i}]_q +a_{j,k-1}a_{i+1,l}a_{j,m} -a_{j,k-1}\lr{a_{j,l},a_{i+1,m}}_q\\
 \phantom{\qquad\qquad\xlongequal{\qquad}}{}
 +\lr{a_{k,l},a_{i+1,m}}_q -a_{i+1,l}a_{k,m} \\
\qquad\qquad \xlongequal{\quad} -[[\lr{a_{j,i},a_{k,l}}_q,a_{i+1,m} ]_q,a_{j,i}]_q +a_{i+1,l}[\lr{a_{j,i},a_{k,m}}_q, a_{j,i}]_q \\
 \phantom{\qquad\qquad\xlongequal{\quad}}{}+a_{k,i}[\lr{a_{j,l},a_{i+1,m}}_q, a_{j,i}]_q -a_{j,i}a_{k,i}a_{i+1,l}a_{j,m}+a_{j,k-1}a_{i+1,l}a_{j,m} \\
 \phantom{\qquad\qquad\xlongequal{\quad}}{}-a_{j,k-1}\lr{a_{j,l},a_{i+1,m}}_q+\lr{a_{k,l},a_{i+1,m}}_q -a_{i+1,l}a_{k,m} \\
 \qquad\qquad \xlongequal[\eqref{equ:comm2}]{\eqref{rela-2}} -\bigl[g\bigl(\A_{j,k,i+1}^{k-1,i,l}\bigr),a_{i+1,m}\bigr]_q +a_{i+1,l}g\bigl(\A_{j,k,i+1}^{k-1,i,m}\bigr) +a_{k,i}[\lr{a_{j,l},a_{i+1,m}}_q, a_{j,i}]_q \\
 \phantom{\qquad\qquad\xlongequal{\qquad}}{} -a_{j,i}a_{k,i}a_{i+1,l}a_{j,m} +a_{j,k-1}a_{i+1,l}a_{j,m} -a_{j,k-1}\lr{a_{j,l},a_{i+1,m}}_q\\
 \phantom{\qquad\qquad\xlongequal{\qquad}}{} +\lr{a_{k,l},a_{i+1,m}}_q -a_{i+1,l}a_{k,m} \\
 \qquad\qquad \xlongequal{\eqref{equ:comm2}} -\bigl[g\bigl(\A_{j,k,i+1}^{k-1,i,l}\bigr)+a_{j,k-1}a_{j,l}-a_{k,l},a_{i+1,m}\bigr]_q +a_{k,i}[\lr{a_{j,l},a_{i+1,m}}_q, a_{j,i}]_q \\
 \phantom{\qquad\qquad\xlongequal{\qquad}\!\! }{} +a_{i+1,l}\bigl(g\bigl(\A_{j,k,i+1}^{k-1,i,m}\bigr) +a_{j,k-1}a_{j,m}-a_{k,m}\bigr) -a_{j,i}a_{k,i}a_{i+1,l}a_{j,m} \\
 \qquad\qquad \xlongequal{\eqref{eqn:sharp}} \lrb{a_{k,i}a_{i+1,l} -a_{j,k-1}a_{j,i}a_{i+1,l}}a_{i+1,m} -a_{i+1,l}\lrb{a_{k,i}a_{i+1,m} -a_{j,i}a_{j,k-1}a_{i+1,m}}\\
 \qquad\qquad \xlongequal{\quad} 0.
\end{gather*}
This proof is finished.
\end{proof}

\begin{Lemma}\label{lem:eqn-2}
For the submatrix \smash{$\A_{j,k,i+1}^{i,l,m}$} as {\rm \eqref{eqn-A-2}}, the following relations hold in $\A(n)$:
\begin{gather}%\label{eqn:com2} \label{eqn:com5}
 [a_{k,i}a_{j,l}-\lr{a_{j,i},a_{k,l}}_q, a_{i+1,l}a_{k,m}-\lr{a_{i+1,m},a_{k,l}}_q ]=0, \label{eqn:com1}\\
 [a_{k,i}a_{j,l}-\lr{a_{j,i},a_{k,l}}_q, a_{j,k-1}a_{i+1,m}-\lr{a_{j,i},a_{k,m}}_q]=0 , \\
 [a_{j,k-1}a_{i+1,m}-\lr{a_{j,i},a_{k,m}}_q, a_{k,l}a_{j,m}-\lr{a_{j,l},a_{k,m}}_q]=0, \label{eqn:com3}\\
 [a_{j,i}a_{l+1,m}-\lr{a_{j,l},a_{i+1,m}}_q,
 a_{i+1,l}a_{k,m}-[a_{k,l},a_{i+1,m}]_q]=0, \label{eqn:com4}\\
 [a_{j,i}a_{l+1,m}-\lr{a_{j,l},a_{i+1,m}}_q,
 a_{k,l}a_{j,m}-\lr{a_{j,l},a_{k,m}}_q]=0.
\end{gather}
\end{Lemma}
\begin{proof}
Here we only verify the relation \eqref{eqn:com1}.
The proofs of the other relations are in a similar way.
We have
\begin{gather*}
 [- [a_{j,i}, a_{k,l} ]_q+a_{k,i}a_{j,l}, - [a_{i+1,m}, a_{k,l} ]_q+a_{i+1,l}a_{k,m} ]\\
\qquad \xlongequal{\quad} {}- [- [a_{j,i}, a_{k,l} ]_q+a_{k,i}a_{j,l}, [a_{i+1,m}, a_{k,l} ]_q ]+a_{i+1,l} [- [a_{j,i}, a_{k,l} ]_q+a_{k,i}a_{j,l}, a_{k,m} ]\\
\qquad \xlongequal{\eqref{eq:qjacobi3}} [a_{i+1,m}, [a_{k,l}, - [a_{j,i}, a_{k,l} ]_q+a_{k,i}a_{j,l} ]_q ] + [a_{k,l}, [- [a_{j,i}, a_{k,l} ]_q+a_{k,i}a_{j,l}, a_{i+1,m} ]_q ]\\
\phantom{\qquad\xlongequal{\qquad}}{} +a_{i+1,l} [- [a_{j,i}, a_{k,l} ]_q+a_{k,i}a_{j,l}, a_{k,m} ]\\
\qquad \xlongequal[\eqref{rela-3}]{\eqref{rela-2}} [a_{i+1,m}, -a_{j,i}+a_{j,k-1}a_{k,i}+a_{i+1,l}a_{j,l} ]+ [a_{k,l}, [a_{j,l}, a_{k,m} ]_q-a_{j,k-1}a_{l+1,m} \\
\phantom{\qquad\xlongequal{\qquad}}{} -a_{k,l}a_{j,m}+a_{j,k-1}a_{l+1,m} +a_{i+1,l} (- [a_{j,i}, a_{k,m} ]_q +a_{j,k-1}a_{i+1,m} +a_{k,i}a_{j,m} ) ]\\
\phantom{\qquad\xlongequal{\qquad}}{} +a_{i+1,l} [- [a_{j,i}, a_{k,l} ]_q+a_{k,i}a_{j,l}, a_{k,m} ]\\
\qquad \xlongequal{(\textrm R1)} a_{i+1,l} [a_{i+1,m}, a_{j,l} ]+a_{i+1,l} [a_{k,l}, - [a_{j,i}, a_{k,m} ]_q+a_{j,k-1}a_{i+1,m}+a_{k,i}a_{j,m} ]\\
\phantom{\qquad\xlongequal{\qquad}}{} +a_{i+1,l} [- [a_{j,i}, a_{k,l} ]_q+a_{k,i}a_{j,l}, a_{k,m} ]\\
\qquad \xlongequal{\quad} a_{i+1,l} ( [a_{i+1,m}, a_{j,l} ]- [a_{k,l}, [a_{j,i}, a_{k,m} ]_q-a_{j,k-1}a_{i+1,m}-a_{k,i}a_{j,m} ] \\
 \phantom{\qquad \xlongequal{\quad}}{} - [ [a_{j,i}, a_{k,l} ]_q-a_{k,i}a_{j,l}, a_{k,m} ] )\\
 \qquad\xlongequal{\eqref{eq:qjacobi0}} \frac{q-q^{-1}}{q+q^{-1}}a_{i+1,l} ( [a_{i+1,m}, a_{j,l} ]_q- [a_{j,l}, a_{i+1,m} ]_q \\
\phantom{\qquad\xlongequal{\qquad}}{} + [a_{k,l}, - [a_{j,i}, a_{k,m} ]_q+a_{j,k-1}a_{i+1,m}+a_{k,i}a_{j,m} ]_q \\
\phantom{\qquad\xlongequal{\qquad}}{} - [ - [a_{j,i}, a_{k,m} ]_q+a_{j,k-1}a_{i+1,m}+a_{k,i}a_{j,m}, a_{k,l} ]_q \\
\phantom{\qquad\xlongequal{\qquad}}{} + [- [a_{j,i}, a_{k,l} ]_q+a_{k,i}a_{j,l}, a_{k,m} ]_q- [ a_{k,m}, (- [a_{j,i}, a_{k,l} ]_q+a_{k,i}a_{j,l} ) ]_q )\\
\qquad \xlongequal{\quad} \frac{q-q^{-1}}{q+q^{-1}}a_{i+1,l} ( ( [a_{i+1,m}, a_{j,l} ]_q+ [ [a_{k,m}, a_{j,i} ]_q,a_{k,l} ]_q- a_{k,i} [a_{k,m}, a_{j,l} ]_q \\
 \phantom{\qquad \xlongequal{\quad}}{} -a_{j,k-1} [a_{i+1,m}, a_{k,l} ]_q )- ( [a_{j,l}, a_{i+1,m} ]_q+ [ [a_{k,l}, a_{j,i} ]_q,a_{k,m} ]_q
\\
 \phantom{\qquad \xlongequal{\quad}}{} -a_{k,i} [a_{j,l}, a_{k,m} ]_q-a_{j,k-1} [a_{k,l}, a_{i+1,m} ]_q ) )\\
\qquad \xlongequal[\eqref{eqn:3b}]{\eqref{eqn:3a}} 0.
\end{gather*}
This proof is finished.
\end{proof}
\begin{Proposition}\label{prop:3}
The sub-algebra $\mathcal{W}$ of $\A(n)$ generated by the entries of \smash{$\A_{i,k,j+1}^{k-1,j,l}$} as $\eqref{eqn-A-1}$ is isomorphic to $\A(3)$.
\end{Proposition}
\begin{proof}
The algebra $\A(3)$ is described as in Remark \ref{remark}.
Set a map $\psi\colon\mathcal{W}$ to $\A(3)$ such that
\begin{gather*}
\psi\bigl(\A_{i,k,j+1}^{k-1,j,l}\bigr)=\psi\begin{pmatrix}
a_{i,k-1}& a_{i,j}&a_{i,l}\\
& a_{k,j}&a_{k,l}\\
& &a_{j+1,l}
\end{pmatrix}
=\begin{pmatrix}
a_{1,1}& a_{1,2}&a_{1,3}\\
& a_{2,2}&a_{2,3}\\
& &a_{3,3}
\end{pmatrix}.
\end{gather*}
It is obvious that $\psi$ is an algebraic isomorphism.
Finally, we note that $a_{i,k-1}$, $a_{k,j}$, $a_{j+1,l}$, $a_{i,j}$ are in the center of the subalgebra $\mathcal{W}$.
\end{proof}
\begin{Corollary}\label{cor:comm}
Given two generators $a_{i,j}, a_{k,l} \in \A$, then $[a_{i,j}, a_{k,l}]=0$ if one of the following holds:
\begin{enumerate}\itemsep=0pt
 \item[$(1)$] $i = k$;
 \item[$(2)$] $j = l$;
 \item[$(3)$] $a_{i,j}$, $a_{k,l}$ locate on the anti-diagonal of any $3\times 3$ submatrix of the matrix $\A.$
\end{enumerate}
\end{Corollary}
\begin{proof}
It is obvious by Definition \ref{defn:A(n)} and Proposition \ref{prop:3}.
\end{proof}

We also have the following obvious results.
\begin{Corollary}
The generators $a_{i, i}$ $(i \in \lbb{1,n})$ and $ a_{1, n}$ are in the center of $\A(n)$.
\end{Corollary}

\begin{Proposition} We have the following filtration of algebras$:$
\begin{equation*}
\A(1) \subseteq \A(2) \subseteq \cdots \subseteq \A(n-1) \subseteq \A(n).
\end{equation*}
\end{Proposition}

\section[The relations between algebra A(n) and aw(n)]{The relations between algebra $\boldsymbol{\A(n)}$ and $\boldsymbol{\aw(n)}$}\label{sect:relations}
In this section, we first recall the concept of the higher-rank Askey--Wilson algebras $\aw(n)$ introduced in \cite{CFPR(2023)}.
Then we explore the relationship between algebras $\A(n)$ and $\aw(n)$.
In fact, we shall see that $\A(n)$ is just isomorphic to the algebra $\aw(n)$.
In other word, we get a~equivalent definition of the higher-rank Askey--Wilson algebras $\aw(n)$.

To see the fact, firstly, let us review the definition of $\aw(n)$ in \cite{CFPR(2023)}.

The following relevant notations and terminology on sets and subsets are used:
\begin{itemize}\itemsep=0pt
 \item For two subsets $I, J \subseteq \{1, \dots, n\}$, we say that $I < J
 $ if and only if
 \begin{equation*}
 i < j , \qquad \forall i \in I, \qquad j \in J.
 \end{equation*}

 \item A non-empty subset $I \subseteq \{1, \dots, n\}$ is said to be connected if it consists in a subset of consecutive integers.

 \item Two disjoint connected subsets $I, J \subseteq \{1, \dots, n\}$ are said to be adjacent if their union is connected.

 \item A hole $H$ between two disjoint connected subsets $I_1$ and $I_2$ means that $H$ consists in the connected subset between $I_1$ and $I_2$, as the picture:
 \begin{equation*}
 \begin{array}{ccccc}
 \dots, \bullet, & \underbrace{ \bullet, \dots, \bullet,} & \underbrace{\bullet, \dots, \bullet,} & \underbrace{\bullet, \dots, \bullet,} & \bullet, \dots\\
 & I_1 & H & I_2 &
 \end{array}
 \end{equation*}
 In this picture and in all the similar pictures below, the integers dotted by $\bullet$'s are ordered either from left to right or from right to left (depending on the respective positions of $I_1$ and $I_2$ in the natural order).
 So such a picture does not mean that $I_1 < I_2$, but $I_1$, $H$, $I_2$ are adjacent connected subsets.
 \item A sequence $(I_1, \dots, I_k)$ of non-empty connected subsets of $\{1, \dots, n\}$ is said monotonic if either $I_1<I_2<\cdots<I_k$ or $I_1> I_2 >\cdots>I_k$.
\end{itemize}

Let us consider two non-empty and disjoint connected subsets $I_1$, $I_2$ of $\{1, \dots, n\}$ with a~non-empty hole $H$ between them:
 
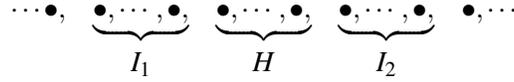
\begin{equation*}
 \begin{array}{ccccc}
 \dots, \bullet, & \underbrace{ \bullet, \dots, \bullet,} & \underbrace{\bullet, \dots, \bullet,} & \underbrace{\bullet, \dots, \bullet,} & \bullet, \dots\\
 & I_1 & H & I_2 &
 \end{array}
 \end{equation*}
The element $C_{I_1I_2}$ is defined by
\begin{eqnarray}\label{eqn:C{I_1I_2}}
C_{I_1I_2}:= -{[C_{I_1H}, C_{HI_2}]}_q+C_{I_1}C_{I_2}+C_{H}C_{I_1HI_2}.
\end{eqnarray}

Now we can describe the definition of higher-rank Askey--Wilson algebra $\aw(n)$.

\begin{Definition}[{\cite{CFPR(2023)}}] \label{defn:aw(n)}
The algebra $\aw(n)$ is the unital associative algebra generated by the elements $C_I$, where $I$ is any non-empty connected subset of $\{1, \ldots, n\}$, satisfying the following relations:
\begin{itemize}\itemsep=0pt
 \item for any two connected subsets $I$ and $J$,
$[C_I, C_J]=0$ if $I \cap J=\varnothing$ or $I \subset J$;
 \item for any monotonic sequence of three adjacent non-empty connected subsets $(I_1, I_2, I_3)$,
\begin{equation*}
C_{I_1 I_2}=-[C_{I_2 I_3}, C_{I_1 I_3}]_q+C_{I_1} C_{I_2}+C_{I_3} C_{I_1 I_2 I_3},
\end{equation*}
where $C_{I_1 I_3}$ is defined by $\eqref{eqn:C{I_1I_2}};$
 \item for any monotonic sequence of four adjacent non-empty connected subsets $(I_1, I_2, I_3, I_4)$,
\begin{equation*}
C_{I_1 I_4}=-\left[C_{I_1 I_3}, C_{I_3 I_4}\right]_q+C_{I_1} C_{I_4}+C_{I_3} C_{I_1 I_3 I_4},
\end{equation*}
where $C_{I_1 I_3}$, $C_{I_1 I_4}$ and $C_{I_1 I_3 I_4}$ are defined by $\eqref{eqn:C{I_1I_2}}.$
\end{itemize}
\end{Definition}

\begin{Theorem}
The algebra $\A(n)$ is isomorphic to $\aw(n)$.
\end{Theorem}
\begin{proof}
To see this, we define a map $\phi\colon \A(n) \rightarrow \aw(n)$, which assigns $a_{ij}$ to $C_{\{i, i+1, \dots, j\}}$, where~${\{i, i+1, \dots, j\}\subseteq\{1, 2, \dots, n\}}$.

It is evident that the mapping $\phi$ is well-defined.
In the following, we show that $\phi$ is indeed an algebra homomorphism.

(a) $\phi$ keeps the relations (R1) in Definition \ref{defn:A(n)}.
For a given generator $a_{i,j}\in \A$, if there exists $a_{k,l}\in \A$ such that $\lr{a_{i,j}, a_{k,l}}=0$, then $a_{k,l}\not\in \A_{12}(i-1,j-i)\cup \A_{34}(j-i,n-j)$.
Now, we set
\begin{equation*}
I_1=\{i, i+1, \dots, j\}, \qquad I_2=\{k, k+1, \dots, l\}.
\end{equation*}
If $a_{k,l}$ belongs to $\A_{11}(n-j,n-j) \hbox{ or } \A_{44}(i-1,i-1)$, then
$
I_1\cap I_2=\varnothing$.
If $a_{k,l} \in \A_{14}(i-1,n-j)$, we have
$
I_1 \subseteq I_2$.
If $a_{k,l} \in \A_{32}(j-i,j-i)$, we have
$
I_2 \subseteq I_1$.
In addition, if $a_{k,l}$ is located on the same row or column as $a_{i,j}$, then
$
I_1 \subseteq I_2 \hbox{ or } I_2 \subseteq I_1$.
Therefore, in all the aforementioned cases, we have
\begin{gather*}
 \phi(a_{i, j})\phi(a_{k, l})-\phi(a_{k, l})\phi(a_{i, j})
 =C_{I_1}C_{I_2}-C_{I_2}C_{I_1}
 =0.
\end{gather*}

(b) $\phi$ keeps the relations (R2) in Definition \ref{defn:A(n)}.
Recall that \smash{$\A_{i,k,j+1}^{k-1,j,l}$} as in $\eqref{eqn-A-1}$, we have $ i < k \leq j < l $ and define
 \begin{equation*}
 I_1=\{i, i+1, \dots, k-1\}, \qquad I_2=\{k, k+1, \dots, j\}, \qquad I_3=\{j+1, \dots, l\}.
 \end{equation*}
The sequences $(I_1, I_2, I_3)$ and $(I_3,I_2,I_1)$ are two monotonic sequences.

Accordingly, we have
 \begin{gather*}
 [\phi(a_{k, l}), {\lr{\phi(a_{i, j}), \phi(a_{k, l})}}_q]_q\\
 \qquad=[ C_{I_{2}I_3}, \lr{ C_{I_{1}I_2} , C_{I_{2}I_3}}_q ]_q
 =\lr{ C_{I_{2}I_3}, -C_{I_1I_3}+C_{I_1}C_{I_3}+C_{I_2}C_{I_1I_2I_3} }_q\\
\qquad =-\lr{ C_{I_{2}I_3}, C_{I_1I_3}}_q+C_{I_{2}I_3} C_{I_1}C_{I_3} + C_{I_{2}I_3}C_{I_2}C_{I_1I_2I_3} \\
 \qquad= C_{I_1I_2}-C_{I_1}C_{I_2}-C_{I_3}C_{I_1I_2I_3}+C_{I_{2}I_3} C_{I_1}C_{I_3} + C_{I_{2}I_3}C_{I_2}C_{I_1I_2I_3} \\
\qquad =\phi\lrb{a_{i, j}}-\phi\lrb{a_{i, k-1}}\phi\lrb{a_{k, j}}-\phi\lrb{a_{i, l}}\phi\lrb{a_{j+1, l}}\\
\phantom{ \qquad=}{} + \phi\lrb{a_{k, l}} \lrb{\phi\lrb{a_{i, k-1}}\phi\lrb{a_{j+1, l}} +\phi\lrb{ a_{k, j }} \phi\lrb{a_{i, l}}},
 \end{gather*}
and
\begin{gather*}
[\phi\lrb{a_{i, j}}, {\lr{\phi\lrb{a_{k, l}}, \phi\lrb{a_{i, j}}}}_q]_q\\
 \qquad=[ C_{I_{2}I_1}, \lr{ C_{I_{3}I_2} , C_{I_{2}I_1}}_q ]_q
 =\lr{ C_{I_{2}I_1}, -C_{I_3I_1}+C_{I_3}C_{I_1}+C_{I_2}C_{I_3I_2I_1} }_q\\
\qquad =-\lr{ C_{I_{2}I_1}, C_{I_3I_1}}_q+C_{I_{2}I_1} C_{I_3}C_{I_1} + C_{I_{2}I_1}C_{I_2}C_{I_3I_2I_1} \\
\qquad = C_{I_3I_2}-C_{I_3}C_{I_2}-C_{I_1}C_{I_3I_2I_1} +C_{I_{2}I_1} C_{I_3}C_{I_1} + C_{I_{2}I_1}C_{I_2}C_{I_3I_2I_1} \\
 \qquad=\phi\lrb{a_{k, l}}-\phi\lrb{a_{i, k-1}}\phi\lrb{a_{i, l}}-\phi\lrb{a_{k, j}}\phi\lrb{a_{j+1, l}}\\
\phantom{ \qquad=}{} + \phi\lrb{a_{i, j}}\lrb{\phi\lrb{a_{i, k-1}}\phi\lrb{a_{j+1, l}} + \phi\lrb{a_{k, j }}\phi\lrb{a_{i, l}}}.
 \end{gather*}
Hence,
\begin{gather*}
[\phi\lrb{a_{k, l}}, {\lr{\phi\lrb{a_{i, j}}, \phi\lrb{a_{k,l}}}}_q]_q=f\bigl[\phi\bigl(\A_{i,k,j+1}^{k-1,j,l}\bigr)\bigr],\\
[\phi\lrb{a_{i, j}}, {\lr{\phi\lrb{a_{k, l}}, \phi\lrb{a_{i, j}}}}_q]_q=g\bigl[\phi\bigl(\A_{i,k,j+1}^{k-1,j,l}\bigr)\bigr].
\end{gather*}

(c) $\phi$ keeps the relations (R3) in Definition \ref{defn:A(n)}:

Let us consider \smash{$\A_{i,k,j+1}^{j,l,m}$} as in $\eqref{eqn-A-2}$, we must have
$ i < k \leq j < l < m $. Set
\begin{gather*}
I_1=\{i, i+1, \dots, k-1\}, \qquad I_2=\{k, k+1, \dots, j\},\\
 I_3=\{ j+1, \dots, l\},\qquad I_4=\{l+1, \dots, m\}.
\end{gather*}
Then, $(I_1, I_2, I_3, I_4)$ and $(I_4, I_3, I_2, I_1)$ are two monotonic sequences and we have
 \begin{gather*}
 [ [\phi(a_{i, j}), \phi(a_{k, l}) ]_q, \phi(a_{j+1, m}) ]_q+ [\phi(a_{i,l}),\phi(a_{k,m}) ]_q+ [ [\phi(a_{i, m}), \phi(a_{k, j}) ]_q, \phi(a_{j+1, l}) ]_q \\
 \qquad - [ [ \phi(a_{i, l }), \phi(a_{k,j}) ]_q, \phi(a_{j+1, m}) ]_q - [ [ \phi(a_{i, j}), \phi(a_{k, m}) ]_q, \phi(a_{j+1, l}) ]_q - [\phi(a_{i, m}), \phi(a_{k, l}) ]_q\\
\phantom{\qquad\qquad}{} = [ [C_{I_1I_2}, C_{I_2I_3} ]_q, C_{I_3I_4} ]_q+ [C_{I_1I_2I_3},C_{I_2I_3I_4} ]_q+ [ [C_{I_1I_2I_3I_4}, C_{I_2} ]_q, C_{I_3} ]_q \\
 \phantom{\qquad\qquad=}{} - [ [ C_{I_1I_2I_3}, C_{I_2} ]_q, C_{I_3I_4} ]_q - [ [ C_{I_1I_2}, C_{I_2I_3I_4} ]_q, C_{I_3} ]_q - [C_{I_1I_2I_3I_4}, C_{I_2I_3} ]_q\\
\phantom{\qquad\qquad}{} = [ [C_{I_1I_2}, C_{I_2I_3} ]_q-C_{I_2}C_{I_1I_2I_3}, C_{I_3I_4} ]_q
+ [C_{I_1I_2I_3},C_{I_2I_3I_4} ]_q \\
 \phantom{\qquad\qquad=}{} + C_{I_3} \lrb{ C_{I_2} C_{I_1I_2I_3I_4}- [ C_{I_1I_2}, C_{I_2I_3I_4} ]_q }- C_{I_1I_2I_3I_4} C_{I_2I_3}\\
 \phantom{\qquad\qquad}{} = [-C_{I_1I_3}+C_{I_1}C_{I_3},C_{I_3I_4} ]_q + [C_{I_1I_2I_3},C_{I_2I_3I_4} ]_q + C_{I_3} \lrb{C_{I_1I_3I_4} -C_{I_1} C_{I_3I_4} }\\
 \phantom{\qquad\qquad=}{}- C_{I_1I_2I_3I_4} C_{I_2I_3}\\
 \phantom{\qquad\qquad}{} = - [C_{I_1I_3},C_{I_3I_4} ]_q + [C_{I_1I_2I_3},C_{I_2I_3I_4} ]_q- C_{I_1I_2I_3I_4} C_{I_2I_3} + C_{I_3} C_{I_1I_3I_4} \\
 \phantom{\qquad\qquad}{} = - [C_{I_1I_3},C_{I_3I_4} ]_q -C_{I_1I_4}+C_{I_1} C_{I_4} + C_{I_3} C_{I_1I_3I_4}\\
 \phantom{\qquad\qquad}{} = 0.
\end{gather*}
We also have
 \begin{gather*}
 [ [\phi(a_{j+1, m}), \phi(a_{k, l}) ]_q, \phi(a_{i, j}) ]_q+ [\phi(a_{k,m}), \phi(a_{i,l}) ]_q + [ [\phi(a_{j+1, l}), \phi(a_{k, j}) ]_q, \phi(a_{i, m}) ]_q \\
 \qquad {}- [ [ \phi(a_{j+1, m}), \phi(a_{k, j}) ]_q, \phi(a_{i, l}) ]_q - [ [\phi(a_{j+1, l}), \phi(a_{k, m}) ]_q, \phi(a_{i, j}) ]_q- [\phi(a_{k, l}), \phi(a_{i, m}) ]_q\\
 \phantom{\qquad\qquad}{} = [ [C_{I_3I_4}, C_{I_2I_3} ]_q, C_{I_1I_2} ]_q+ [C_{I_2I_3I_4}, C_{I_1I_2I_3} ]_q + [ [C_{I_3}, C_{I_2} ]_q, C_{I_1I_2I_3I_4} ]_q \\
 \phantom{\qquad\qquad=}{} - [ [ C_{I_3I_4}, C_{I_2} ]_q, C_{I_1I_2I_3} ]_q - [ [C_{I_3}, C_{I_2I_3I_4} ]_q, C_{I_1I_2} ]_q- [C_{I_2I_3}, C_{I_1I_2I_3I_4} ]_q\\
 \phantom{\qquad\qquad}{} = [ [C_{I_3I_4}, C_{I_2I_3} ]_q-C_{I_3} C_{I_2I_3I_4}, C_{I_1I_2} ]_q+ [C_{I_2I_3I_4}, C_{I_1I_2I_3} ]_q \\
 \phantom{\qquad\qquad=}{} +C_{I_2}\lrb{C_{I_3}C_{I_1I_2I_3I_4}- [ C_{I_3I_4}, C_{I_1I_2I_3} ]_q }-C_{I_2I_3} C_{I_1I_2I_3I_4}\\
 \phantom{\qquad\qquad}{} = [-C_{I_4I_2}+C_{I_2} C_{I_4}, C_{I_1I_2} ]_q+ [C_{I_2I_3I_4}, C_{I_1I_2I_3} ]_q +C_{I_2}\lrb{C_{I_4I_2I_1}-C_{I_4}C_{I_1I_2}}\\
 \phantom{\qquad\qquad=}{} -C_{I_2I_3} C_{I_1I_2I_3I_4}\\
 \phantom{\qquad\qquad}{} = - [C_{I_4I_2}, C_{I_1I_2} ]_q +C_{I_2}C_{I_4I_2I_1}+ \lrb{ [C_{I_2I_3I_4}, C_{I_1I_2I_3} ]_q -C_{I_2I_3} C_{I_1I_2I_3I_4}}\\
 \phantom{\qquad\qquad}{} = - [C_{I_4I_2}, C_{I_1I_2} ]_q +C_{I_2}C_{I_4I_2I_1}- C_{I_4I_1}+C_{I_1}C_{I_4}\\
 \phantom{\qquad\qquad}{} = 0.
 \end{gather*}
Hence,
\begin{gather*}
{\det}_q\bigl(\phi\bigl(\A_{i,k,j+1}^{j,l,m}\bigr)\bigr)={\det}^q\bigl(\phi\bigl(\A_{i,k,j+1}^{j,l,m}\bigr)\bigr)=0.
\end{gather*}

The above statements mean that $\phi$ is indeed an algebra homomorphism.

Furthermore, we define a map $\varphi\colon \aw(n) \rightarrow \A(n)$, which sends generator $C_I$ of $\aw(n)$ to~$a_{i_I, j_I}$, where $i_I=\min I$ and $j_I=\max I$.

Now, we have $i \leq j$ and $I=\{i, i+1, \dots, j\}$ be a non-empty connected subset of $\{1, 2, \dots, n\}$.
Similarly, it is straightforward to see that $\varphi$ is an algebra homomorphism.
Also, $\varphi \phi\lrb{a_{i, j}}=\varphi\lrb{C_I}=a_{i, j}$ and $\phi \varphi \lrb{C_I}=\phi\lrb{a_{i, j}}=C_I$.
It follows that $\phi$ is an isomorphism of algebras.

Thus $\A(n) \cong \aw(n)$
and the proof is finished.
\end{proof}

\section[Automorphisms of A(n)]{Automorphisms of $\boldsymbol{\A(n)}$}\label{sect:automorphism}

In this section, we construct a series of automorphisms of $\A(n)$,
which coincide with those of~$\aw(n)$ in \cite{CFPR(2023)}.
The proofs of the results are sketch in \cite{CFPR(2023)} and slightly difficult to handle.
Here we give the proofs in detail.

For the upper triangular generating matrix $\A=(a_{ij})_{n\times n}$ of the algebra $\A(n)$ and the map $f\colon \A(n)\to \A(n)$, the notation $f\lrb{\A}$ is defined by the upper triangular matrix $(f\lrb{\A}_{i,j})_{n\times n}$, where $f\lrb{\A}_{i,j}:=f\lrb{a_{i,j}}$.
For convenience, we denote
\begin{gather}%\label{eqn:r0'}\label{eqn:ri'}\label{eqn:ri+1'}
 {a_{1,j}}^{\d_0}\overset{\bigtriangleup}{=} -\lq{a_{2,n}, a_{1,j}}+a_{j+1,n}a_{1,1}+a_{2,j}a_{1,n}, \label{eqn:r0} \\
 {a_{1,j}}^{\d_0'}\overset{\bigtriangleup}{=} -\lq{a_{1,j}, a_{2,n}}+a_{j+1,n}a_{1,1}+a_{2,j}a_{1,n}, \\
 {a_{k,i}}^{\d_i}\overset{\bigtriangleup}{=} -\lq{a_{i,i+1}, a_{k,i}}+a_{i+1,i+1}a_{k,i-1}+a_{i,i}a_{k,i+1}, \label{eqn:ri} \\
 {a_{i+1, l}}^{\d_i}\overset{\bigtriangleup}{=} -\lq{a_{i,i+1}, a_{i+1, l}}+a_{i,i}a_{i+2,l}+a_{i+1,i+1}a_{i,l}, \label{eqn:ri+1}\\
 {a_{k,i}}^{\d_i'}\overset{\bigtriangleup}{=} -\lq{a_{k,i}, a_{i,i+1}}+a_{i+1,i+1}a_{k,i-1}+a_{i,i}a_{k,i+1}, \\
 {a_{i+1, l}}^{\d_i'}\overset{\bigtriangleup}{=} -\lq{a_{i,i+1}, a_{i+1, l}} +a_{i,i}a_{i+2,l} +a_{i+1,i+1}a_{i,l},
\end{gather}
where $1\leq i\leq n-1$, $2\leq j\leq n-1$, $1 \leq k \leq i-1$ and $i+1<l\leq n-1.$

Now, let $\d_0, \d_0'\colon \A(n) \rightarrow \A(n)$ the maps be given by
\begin{gather*}%\label{eqn-r0(A)}\label{eqn-r0'(A)}
\d_0(\A)=
 \begin{pmatrix}
 a_{1,n} &{a_{1,2}}^{\d_0} & \cdots & {a_{1,j}}^{\d_0}& \cdots &{a_{1,n-1}}^{\d_0}& a_{1,1} \\[1pt]
 & a_{2,2} & \cdots & a_{2,j} & \cdots &a_{2,n-1} &a_{2,n}\\[1pt]
 & & \ddots & \vdots & \vdots & \vdots & \vdots \\[1pt]
 & & & a_{j,j} & \cdots &a_{j,n-1} &a_{j,n}\\[1pt]
 & & & & \ddots & \vdots & \vdots \\[1pt]
 & & & & & a_{n-1,n-1}& a_{n-1,n}\\[1pt]
 & & & & & & a_{n,n}
 \end{pmatrix}
, \\
\d_0'(\A)=
 \begin{pmatrix}
 a_{1,n} & {a_{1,2}}^{\d_0'} & \cdots & {a_{1,j}}^{\d_0'}& \cdots &{a_{1,n-1}}^{\d_0'} & a_{1,1} \\[1pt]
 & a_{2,2} & \cdots & a_{2,j} & \cdots &a_{2,n-1} &a_{2,n}\\[1pt]
 & & \ddots & \vdots & \vdots & \vdots & \vdots \\[1pt]
 & & & a_{j,j} & \cdots &a_{j,n-1} &a_{j,n}\\[1pt]
 & & & & \ddots & \vdots & \vdots \\[1pt]
 & & & & & a_{n-1,n-1}& a_{n-1,n}\\[1pt]
 & & & & & & a_{n,n}
 \end{pmatrix},
\end{gather*}
the map $\d_i, \d_i'\colon \A(n) \rightarrow \A(n)$ ($i \in \lbb{1,n-1}$) be given by
\begin{gather*}% \label{eqn-ri(A)}\label{eqn-ri'(A)}
\d_i(\A)=
\begin{pmatrix}
 a_{1,1} &a_{1,2} & \cdots & {a_{1,i}}^{\d_i} & a_{1,i+1} & a_{1,i+2}& \cdots & a_{1,n} \\
 &\ddots & &\vdots & \vdots & \vdots & \vdots & \vdots\\
 & & & {a_{i-1,i}}^{\d_i} & a_{i-1,i+1} & a_{1-1,i+2}& \cdots & a_{i-1,n} \\
 & & &a_{i+1,i+1}& a_{i,i+1} & a_{i,i+2} &\cdots & a_{i,n} \\
 & & & & a_{i,i} & {a_{i+1,i+2}}^{\d_i}& \cdots & {a_{i+1,n}}^{\d_i} \\
 & & & & & a_{i+2,i+2}& \cdots & a_{i+2,n}\\
 & & & & & &\ddots & \vdots\\
 & & & & & & & a_{n,n}
\end{pmatrix}
,\\
\d_i'(\A)=
\begin{pmatrix}
 a_{1,1} &a_{1,2} & \cdots & {a_{1,i}}^{\d'_i} & a_{1,i+1} & a_{1,i+2}& \cdots & a_{1,n} \\
 &\ddots & &\vdots & \vdots & \vdots & \vdots & \vdots\\
 & & & {a_{i-1,i}}^{\d'_i} & a_{i-1,i+1} & a_{1-1,i+2}& \cdots & a_{i-1,n} \\
 & & &a_{i+1,i+1}& a_{i,i+1} & a_{i,i+2} &\cdots & a_{i,n} \\
 & & & & a_{i,i} & {a_{i+1,i+2}}^{\d'_i}& \cdots & {a_{i+1,n}}^{\d'_i} \\
 & & & & & a_{i+2,i+2}& \cdots & a_{i+2,n}\\
 & & & & & &\ddots & \vdots\\
 & & & & & & & a_{n,n}
\end{pmatrix}
.
\end{gather*}

For the maps $\delta_0$, $ \delta_0'$, the first row $a_{1,j}$ $(2\leq j\leq n-1)$ of $\A$ are mapping to ${a_{1,j}}^{\delta_0}$ \big(resp.~${a_{1,j}}^{\delta_0'}$\big), $\delta_0(a_{1,1})=a_{1,n}$, $\delta_0(a_{1,n})=a_{1, 1}$, and the other generators are fixed.
Similarly, for the maps~$\d_i$,~$\d_i'$ $(i \in \lbb{1,n-1})$, the generators of $\A$ are fixed except those of the $(i+1)$-th row and the $i$-th column.

Subsequently, we aim to demonstrate that the maps~$\d_i$,~$\d_i'$ $(i \in \lbb{0,n-1}$) defined in this way are automorphisms and satisfy the braid group relations.
We begin by introducing two lemmas.

\begin{Lemma}\quad
\begin{enumerate}\itemsep=0pt
\item[$(1)$] For \smash{$\A_{1,k,i+1}^{i,j,m} $} as $\eqref{eqn-A-2}$, we have
\begin{gather}%\label{r0-4}\label{r0-5}
\bigl[{a_{1,i}}^{\d_0},{a_{1,k-1}}^{\d_0}]_q =a_{1,i}a_{1,k-1}-a_{1,1}\lr{a_{2,i},a_{1,k-1}}_q -a_{i+1,n}a_{1,n}a_{1,k-1} \nonumber\\
\phantom{\bigl[{a_{1,i}}^{\d_0},{a_{1,k-1}}^{\d_0}\bigr]_q =}{}+a_{1,1}\bigl[{a_{1,i}}^{\d_0},a_{k,n}\bigr]_q
+a_{2,k-1}{a_{1,i}}^{\d_0}a_{1,n}\nonumber\\
\phantom{\bigl[{a_{1,i}}^{\d_0},{a_{1,k-1}}^{\d_0}\bigr]_q =}{}+\frac{1}{\bigl(q-q^{-1}\bigr)^2}\bigl[a_{2,n}, \bigl[a_{1,k-1}, {a_{1,i}}^{\d_0}\bigr]\bigr] , \label{r0-1}\\
\bigl[{a_{1,i}}^{\d_0},{a_{1,j}}^{\d_0} \bigr]_q =a_{1,i}a_{1,j} -a_{1,1}a_{2,i}a_{1,j} -a_{1,n}\lr{a_{i+1,n},a_{1,j}}_q +a_{1,1}{a_{1,i}}^{\d_0}a_{j+1,n} \nonumber\\
\phantom{\bigl[{a_{1,i}}^{\d_0},{a_{1,j}}^{\d_0} \bigr]_q =}{}
+a_{1,n}\bigl[{a_{1,i}}^{\d_0},a_{2,j}\bigr]_q+\frac{1}{\bigl(q-q^{-1}\bigr)^2}\bigl[a_{2,n}, \bigl[a_{1,j}, {a_{1,i}}^{\d_0}\bigr]\bigr],\label{r0-2}\\
\bigl[{a_{1,k-1}}^{\d_0},{a_{1,j}}^{\d_0}\bigr]_q =a_{1,k-1}a_{1,j}-a_{1,1}a_{2,k-1}a_{1,j} -a_{1,n}\lr{a_{k,n},a_{1,j}}_q +a_{1,1}{a_{1,k-1}}^{\d_0}a_{j+1,n} \nonumber\\
\phantom{\bigl[{a_{1,k-1}}^{\d_0},{a_{1,j}}^{\d_0}\bigr]_q =}{}
+a_{1,n}\bigl[{a_{1,k-1}}^{\d_0},a_{2,j}\bigr]_q+\frac{1}{\bigl(q-q^{-1}\bigr)^2}\bigl[a_{2,n}, \bigl[a_{1,j}, {a_{1,k-1}}^{\d_0}\bigr]\bigr], \label{r0-3}\\
\bigl[{a_{1,m}}^{\d_0},{a_{1,j}}^{\d_0}\bigr]_q
=a_{1,m}a_{1,j}- a_{1,1}\lr{a_{2,m},a_{1,j}}_q -a_{1,n}a_{m+1,n}a_{1,j} +a_{1,1}\bigl[{a_{1,m}}^{\d_0},a_{j+1,n}\bigr]_q \nonumber\\
\phantom{\bigl[{a_{1,m}}^{\d_0},{a_{1,j}}^{\d_0}\bigr]_q
=a}{}
+a_{1,n}{a_{1,m}}^{\d_0}a_{2,j}+\frac{1}{\bigl(q-q^{-1}\bigr)^2}\bigl[a_{2,n}, \bigl[a_{1,j}, {a_{1,m}}^{\d_0}\bigr]\bigr],\nonumber\\
\bigl[{a_{1,m}}^{\d_0},{a_{1,i}}^{\d_0}\bigr]_q
=a_{1,m}a_{1,i} -a_{1,1}\lr{a_{2,m},a_{1,i}}_q -a_{1,n}a_{m+1,n}a_{1,i} +a_{1,1}\bigl[{a_{1,m}}^{\d_0},a_{i+1,n}\bigr]_q \nonumber\\
\phantom{\bigl[{a_{1,m}}^{\d_0},{a_{1,i}}^{\d_0}\bigr]_q
=}{}
+a_{1,n}{a_{1,m}}^{\d_0}a_{2,i}+\frac{1}{\bigl(q-q^{-1}\bigr)^2}\bigl[a_{2,n}, \bigl[a_{1,i}, {a_{1,m}}^{\d_0}\bigr]\bigr].\nonumber
\end{gather}
\item[$(2)$] For \smash{$\A_{j,k,i+1}^{k-1,i,l}$} as $\eqref{eqn-A-1}$, we have
\begin{gather}
\bigl[{a_{k,i}}^{\d_i},{a_{i+1,l}}^{\d_i}\bigr]_q
=a_{k,i}a_{i+1,l}-a_{k,i-1}a_{i,i}a_{i+1,l} +{a_{k,i}}^{\d_i}a_{i,i}a_{i+2,l}\nonumber \\
\phantom{\bigl[{a_{k,i}}^{\d_i},{a_{i+1,l}}^{\d_i}\bigr]_q
=}{}-a_{i+1,i+1}\lr{a_{k,i+1},a_{i+1,l}}_q
+a_{i+1,i+1}\bigl[{a_{k,i}}^{\d_i},a_{i,l}\bigr]_q\nonumber \\
\phantom{\bigl[{a_{k,i}}^{\d_i},{a_{i+1,l}}^{\d_i}\bigr]_q
=}{} +\frac{1}{\bigl(q-q^{-1}\bigr)^2}\bigl[a_{i,i+1},\bigl[a_{i+1,l},{a_{k,i}}^{\d_i}\bigr]\bigr], \label{ri-1} \\
\bigl[{a_{j,i}}^{\d_i},{a_{k,i}}^{\d_i}\bigr]_q=a_{j,i}a_{k,i} -a_{i+1,i+1}a_{j,i+1}a_{k,i}-a_{i,i}\lr{a_{j,i-1},a_{k,i}}_q +{a_{j,i}}^{\d_i}a_{k,i-1}a_{i+1,i+1}\nonumber \\
\phantom{\bigl[{a_{j,i}}^{\d_i},{a_{k,i}}^{\d_i}\bigr]_q=}{}
+a_{i,i}\bigl[{a_{j,i}}^{\d_i}, a_{k,i+1}]_q +\frac{1}{\bigl(q-q^{-1}\bigr)^2}\bigl[a_{i,i+1},\bigl[a_{k,i},{a_{j,i}}^{\d_i}\bigr]\bigr], \label{ri-2} \\
\bigl[{a_{j,i}}^{\d_i}, {a_{i+1,l}}^{\d_i}\bigr]_q = a_{j,i}a_{i+1,l} -a_{j,i-1}a_{i,i}a_{i+1,l}+{a_{j,i}}^{\d_i}a_{i,i}a_{i+2,l} -a_{i+1,i+1}[a_{j,i+1},a_{i+1,l}]_q\nonumber \\
\phantom{\bigl[{a_{j,i}}^{\d_i}, {a_{i+1,l}}^{\d_i}\bigr]_q = }{}
+a_{i+1,i+1}\bigl[{a_{j,i}}^{\d_i},a_{i,l}\bigr]_q +\frac{1}{\bigl(q-q^{-1}\bigr)^2}\bigl[a_{i,i+1},\bigl[a_{i+1,l},{a_{j,i}}^{\d_i}\bigr]\bigr]. \label{ri-3}
\end{gather}
\item[$(3)$] For \smash{$\A_{j,k,i+1}^{i,l,m}$} as $\eqref{eqn-A-2}$, we have
\begin{gather}%\label{ri-5}
\bigl[{a_{k,i}}^{\d_i},{a_{i+1,m}}^{\d_i} \bigr]_q=a_{k,i}a_{i+1,m} -a_{k,i-1}a_{i,i}a_{i+1,m}+{a_{k,i}}^{\d_i}a_{i,i}a_{i+2,m}\nonumber \\
\phantom{\bigl[{a_{k,i}}^{\d_i},{a_{i+1,m}}^{\d_i} \bigr]_q=}{}
 -a_{i+1,i+1}\lr{a_{k,i+1},a_{i+1,m}}_q
+a_{i+1,i+1}\bigl[{a_{k,i}}^{\d_i},a_{i,m}\bigr]_q
\nonumber \\
\phantom{\bigl[{a_{k,i}}^{\d_i},{a_{i+1,m}}^{\d_i} \bigr]_q=}{}
+\frac{1}{\bigl(q-q^{-1}\bigr)^2}\bigl[a_{i,i+1},\bigl[a_{i+1,m},{a_{k,i}}^{\d_i}\bigr]\bigr], \label{ri-4} \\
\bigl[{a_{j,i}}^{\d_i},{a_{i+1,m}}^{\d_i}\bigr]_q=a_{j,i}a_{i+1,m} -a_{j,i-1}a_{i,i}a_{i+1,m}+{a_{j,i}}^{\d_i}a_{i,i}a_{i+2,m}\nonumber \\
\phantom{\bigl[{a_{j,i}}^{\d_i},{a_{i+1,m}}^{\d_i}\bigr]_q=}{}
 -a_{i+1,i+1}\lr{a_{j,i+1},a_{i+1,m}}_q+a_{i+1,i+1}\bigl[{a_{j,i}}^{\d_i},a_{i,m}\bigr]_q
 \nonumber \\
\phantom{\bigl[{a_{j,i}}^{\d_i},{a_{i+1,m}}^{\d_i}\bigr]_q=}{}
+\frac{1}{\bigl(q-q^{-1}\bigr)^2}\bigl[a_{i,i+1},\bigl[a_{i+1,m},{a_{j,i}}^{\d_i}\bigr]\bigr].\nonumber
\end{gather}
\end{enumerate}
\end{Lemma}
\begin{proof}
We only focus on proving only one of them, the others are similar and straightforward.
For example,
\begin{align*}
\bigl[{a_{1,i}}^{\d_0},{a_{1,k-1}}^{\d_0}\bigr]_q&
\xlongequal{\eqref{eqn:r0}}\bigl[{a_{1,i}}^{\d_0},-\lr{a_{2,n},a_{1,k-1}}_q+a_{1,1}a_{k,n}+a_{2,k-1}a_{1,n}\bigr]_q\\
&\xlongequal{\quad}-\bigl[{a_{1,i}}^{\d_0},\lr{a_{2,n},a_{1,k-1}}_q\bigr]_q +a_{1,1}\bigl[{a_{1,i}}^{\d_0},a_{k,n}\bigr]_q +a_{2,k-1}{a_{1,i}}^{\d_0}a_{1,n}\\
&\xlongequal{\eqref{eq:qjacobi1}}-\bigl[\bigl[{a_{1,i}}^{\d_0},a_{2,n}\bigr]_q,a_{1,k-1}\bigr]_q+a_{1,1}\bigl[{a_{1,i}}^{\d_0},a_{k,n}\bigr]_q
+a_{2,k-1}{a_{1,i}}^{\d_0}a_{1,n}\\
&\phantom{\xlongequal{\qquad}}{}+\frac{1}{\bigl(q-q^{-1}\bigr)^2}\bigl[a_{2,n}, \bigl[a_{1,k-1}, {a_{1,i}}^{\d_0}\bigr]\bigr]\\
&\xlongequal[\eqref{rela-2}]{\eqref{eqn:r0}}\lr{a_{1,i}-a_{1,1}a_{2,i}-a_{i+1,n}a_{1,n},a_{1,k-1}}_q+a_{1,1}\bigl[{a_{1,i}}^{\d_0},a_{k,n}\bigr]_q \\
&\phantom{\xlongequal{\qquad}}{}+a_{2,k-1}{a_{1,i}}^{\d_0}a_{1,n}+\frac{1}{\bigl(q-q^{-1}\bigr)^2}\bigl[a_{2,n}, \bigl[a_{1,k-1}, {a_{1,i}}^{\d_0}\bigr] \bigr]\\
&\xlongequal{\quad }a_{1,i}a_{1,k-1}-a_{1,1}\lr{a_{2,i},a_{1,k-1}}_q-a_{i+1,n}a_{1,n}a_{1,k-1} +a_{1,1}\bigl[{a_{1,i}}^{\d_0},a_{k,n}\bigr]_q\\
&\phantom{\xlongequal{\quad }}{}+a_{2,k-1}{a_{1,i}}^{\d_0}a_{1,n}+\frac{1}{\bigl(q-q^{-1}\bigr)^2}\bigl[a_{2,n}, \bigl[a_{1,k-1}, {a_{1,i}}^{\d_0}\bigr]\bigr].
\end{align*}

This proof is finished.
\end{proof}

\begin{Lemma}\qquad\samepage
\begin{enumerate}\itemsep=0pt
 \item[$(1)$] For \smash{$\A_{1,k,i+1}^{k-1,i,j}$} as $\eqref{eqn-A-1}$, we have
\begin{gather}
[a_{2,i}, \lr{a_{k,j},a_{1,i}}_q -a_{i+1,j}a_{1,k-1}-a_{k,i}a_{1,j}]_q+a_{2,k-1}a_{1,j} \nonumber \\
\qquad=\bigl[{a_{1,i}}^{\d_0},\lr{a_{k,j},a_{i+1,n}}_q -a_{i+1,j}a_{k,n} -a_{k,i}a_{j+1,n}\bigr]_q +\bigl[{a_{1,k-1}}^{\d_0},a_{j+1,n}\bigr]_q, \label{eqn:01}\\
[a_{i+1,n}, \lr{a_{k,j},a_{1,i}}_q -a_{i+1,j}a_{1,k-1}-a_{k,i}a_{1,j}]_q+\lr{a_{k,n},a_{1,j}}_q \nonumber \\
\qquad=\bigl[{a_{1,i}}^{\d_0},\lr{a_{k,j},a_{2,i}}_q -a_{k,i}a_{2,j}-a_{2,k-1}a_{i+1,j}\bigr]_q +\bigl[{a_{1,k-1}}^{\d_0},a_{2,j}\bigr]_q. \label{eqn:02}
\end{gather}
 \item[$(2)$] For \smash{$\A_{j,k,i+1}^{k-1,i,l}$} as $\eqref{eqn-A-1}$, we have
\begin{gather}
[a_{i+2,l},\lr{a_{k,l},a_{j,i-1}}_q-a_{j,k-1}a_{i,l}-a_{k,i-1}a_{j,l}]_q+a_{k,i-1}a_{j,i+1} \nonumber \\
\qquad=\bigl[\lr{a_{k,l},a_{j,i}}_q-a_{j,k-1}a_{i+1,l}-a_{k,i}a_{j,l}, {a_{i+1,l}}^{\d_i}\bigr]_q +\bigl[a_{k,i}, {a_{j,i}}^{\d_i}\bigr]_q, \label{eqn:i1} \\
[\lr{a_{k,l},a_{j,i-1}}_q-a_{j,k-1}a_{i,l}-a_{k,i-1}a_{j,l},a_{k,i+1}]_q+\lr{a_{i,l},a_{j,i+1}}_q \nonumber \\
\qquad=\bigl[\lr{a_{k,l},a_{j,i}}_q-a_{j,k-1}a_{i+1,l} -a_{k,i}a_{j,l}, {a_{k,i}}^{\d_i}\bigr]_q+\bigl[a_{i+1,l}, {a_{j,i}}^{\d_i}\bigr]_q. \label{eqn:i2}
\end{gather}
 \item[$(3)$] For \smash{$\A_{j,k,i+1}^{i,l,m}$} as $\eqref{eqn-A-2}$, we have
\begin{gather}
a_{k,i-1}(\lr{a_{j,l}, a_{i+1,m}}_q-a_{k,i-1}a_{i+1,l}a_{j,m})+[a_{j,i-1}, \lr{a_{k,l},a_{i+1,m}}_q-a_{i+1,l}a_{k,m}]_q \nonumber \\
\qquad
={a_{k,i}}^{\d_i}(\lr{a_{j,l},a_{i+2,m}}_q-a_{i+2,l}a_{j,m})+\bigl[{a_{j,i}}^{\d_i}, a_{i+2,l}a_{k,m}-\lr{a_{k,l},a_{i+2,m}}_q\bigr]_q, \label{eqn:j1}\\
[a_{j,i+1}, a_{i+1,l}a_{k,m} -\lr{a_{k,l},a_{i+1,m}}_q]_q +[a_{k,i+1},\lr{a_{j,l},a_{i+1,m}}_q-a_{i+1,l}a_{j,m}]_q
 \nonumber \\
\qquad
=\bigl[{a_{j,i}}^{\d_i},a_{i,l}a_{k,m}-\lr{a_{k,l},a_{i,m}}_q\bigr]_q +\bigl[{a_{k,i}}^{\d_i},\lr{a_{j,l},a_{i,m}}_q-a_{i,l}a_{j,m}\bigr]_q, \label{eqn:j2}\\
{a_{i+1,l}}^{\d_i}(\lr{a_{k,m},a_{j,i-1}}_q-a_{k,i-1}a_{j,m})+\bigl[{a_{i+1,m}}^{\d_i}, a_{k,i-1}a_{j,l}-\lr{a_{k,l},a_{j,i-1}}_q\bigr]_q \nonumber \\
\qquad
=a_{i+2,l}(\lr{a_{k,m}, a_{j,i}}_q -a_{i+2,l}a_{k,i}a_{j,m})+[a_{i+2,m}, \lr{a_{k,l},a_{j,i}}_q-a_{k,i}a_{j,l}]_q, \label{eqn:j3}\\
\bigl[{a_{i+1,m}}^{\d_i},a_{k,i+1}a_{j,l}-\lr{a_{k,l},a_{j,i+1}}_q\bigr]_q +\bigl[{a_{i+1,l}}^{\d_i},\lr{a_{k,m},a_{j,i+1}}_q-a_{k,i+1}a_{j,m}\bigr]_q \nonumber \\
\qquad
=[a_{i,m},a_{k,i}a_{j,l} -\lr{a_{k,l},a_{j,i}}_q]_q +[a_{i,l},\lr{a_{k,m},a_{j,i}}_q-a_{k,i}a_{j,m}]_q. \label{eqn:j4}
\end{gather}
\end{enumerate}
\end{Lemma}

\begin{proof}
We only focus on proving one of them, the others are similar and straightforward. For example,
\begin{gather*}
 [a_{2,i}, [a_{k,j},a_{1,i} ]_q -a_{i+1,j}a_{1,k-1} ]_q-a_{k,i}a_{2,i}a_{1,j}+a_{2,k-1}a_{1,j}\\
\qquad \xlongequal{\eqref{eqn:4b}} (a_{k,i} [a_{2,j},a_{1,i} ]_q-a_{1,1}a_{k,i}a_{i+1,j} - [a_{2,j},a_{1,k-1} ]_q+a_{1,1}a_{k,j} ) \\
\phantom{\qquad \xlongequal{\qquad} }{}
+a_{k,i}a_{2,i}a_{1,j}+a_{2,k-1}a_{1,j}\\
\qquad \xlongequal{\quad} (- [a_{2,j},a_{1,k-1} ]_q+a_{2,k-1}a_{1,j} )
-a_{k,i} (- [a_{2,j},a_{1,i} ]_q+a_{2,i}a_{1,j} )\\
\qquad\phantom{ \xlongequal{\quad} }{}
 +a_{1,1}a_{k,j}-a_{1,1}a_{k,i}a_{i+1,j}\\
\qquad \xlongequal{\eqref{eqn:2a}} ( [a_{2,n}, [a_{1,j},a_{k,n} ]_q ]_q-a_{1,n} [a_{2,j},a_{k,n} ]_q +a_{2,k-1}a_{j+1,n}a_{1,n}-a_{j+1,n} [a_{2,n},a_{1,k-1} ]_q )\\
\phantom{\qquad \xlongequal{\qquad} }{} -a_{k,i} ( [a_{2,n}, [a_{1,j},a_{i+1,n} ]_q -a_{j+1,n} [a_{2,n},a_{1,i} ]_q +a_{j+1,n}a_{2,i}a_{1,n}\\
\phantom{\qquad \xlongequal{\qquad} }{} -a_{1,n} [a_{2,j},a_{i+1,n} ]_q ]_q )+a_{1,1}a_{k,j}-a_{1,1}a_{k,i}a_{i+1,j}\\
\qquad \xlongequal{\quad} [a_{2,n}, [a_{1,j},a_{k,n} ]_q -a_{k,i} ( [a_{1,j},a_{i+1,n} ]_q +a_{j+1,n}a_{1,i} ) ]_q +a_{1,1} (a_{k,j}-a_{k,i}a_{i+1,j} ) \\
\phantom{\qquad \xlongequal{\quad} }{} -a_{1,n} ( [a_{2,j},a_{k,n} ]_q-a_{2,k-1}a_{j+1,n} +a_{k,i} ( [a_{2,j},a_{i+1,n} ]_q-a_{j+1,n}a_{2,i} ) )\\
 \phantom{\qquad \xlongequal{\quad} }{} -a_{j+1,n} [a_{2,n},a_{1,k-1} ]_q\\
 \qquad\xlongequal[\eqref{rela-3}]{\eqref{rela-2}} [ [a_{2,n},a_{1,i} ]_q , - [a_{k,j},a_{i+1,n} ]_q+a_{k,i}a_{j+1,n}+a_{i+1,j}a_{k,n} ]_q -a_{j+1,n} [a_{2,n},a_{1,k-1} ]_q\\
\phantom{\qquad \xlongequal{\qquad} }{} -a_{1,1} ( [a_{i+1,n},- [a_{k,j},a_{i+1,n} ]_q+a_{k,i}a_{j+1,n}+a_{i+1,j}a_{k,n} ]_q -a_{j+1,n}a_{k,n} ) \\
\phantom{\qquad \xlongequal{\qquad} }{} -a_{1,n} ( [a_{2,i},- [a_{k,j},a_{i+1,n} ]_q+a_{k,i}a_{j+1,n}+a_{i+1,j}a_{k,n} ]_q +a_{2,k-1}a_{j+1,n} )\\
\qquad \xlongequal{\quad} \bigl[{a_{1,i}}^{\d_0}, [a_{k,j},a_{i+1,n} ]_q -a_{i+1,j}a_{k,n}-a_{k,i}a_{j+1,n} \bigr]_q+ \bigl[{a_{1,k-1}}^{\d_0},a_{j+1,n} \bigr]_q.
\end{gather*}
This proof is finished.
\end{proof}

\begin{Proposition}
$\d_0$, $\d_0' $ are automorphisms of algebra $\A(n)$ and $\d_0\d_0'=\d_0'\d_0=\id$.
\end{Proposition}
\begin{proof}
To prove that $\d_0$ is an algebra homomorphism, it is sufficient to demonstrate that $\d_0$ keeps the relations of $\A(n)$ associating with $a_{1,i}$, $i \in \lbb{1, n}.$

The relations (R1):
If $i=1$ or $i=n$, $a_{1,1}$ and $a_{1,n}$ are in the centers of $\A(n)$.
By the definition of $\d_0$, $\d_0(a_{1,1})=a_{1,n}$, $\d_0(a_{1,n})=a_{1,1}$, which are also in the center of $\A(n)$.
We suppose that $i \in \lbb{2,n-1}$.
If there exists $a_{j,k}$ such that $[a_{1,i}, a_{k,l}]=0$, we consider two cases:

When $j=1$ and $ k\neq n$, let us assume that $i<k$, we have
\begin{gather*}
 \lr{\d_0\lrb{a_{1,i}},\d_0\lrb{a_{1,k}}}\\
\qquad \xlongequal[(\d_0)]{\lrb{\ref{eqn:r0}}} \bigl[{a_{1,i}}^{\d_0}, -\lr{a_{2,n}, a_{1,k}}_q+a_{k+1,n}a_{1,1}+a_{2,k}a_{1,n}\bigr]\\
\qquad\xlongequal{\lrb{\textrm{R1}}} \bigl[{a_{1,i}}^{\d_0}, -\lr{a_{2,n}, a_{1,k}}_q+a_{2,k}a_{1,n}\bigr]\\
\qquad \xlongequal{\lrb{\ref{rela-3}}} \bigl[{a_{1,i}}^{\d_0}, [\lr{a_{i+1, n}, a_{2, k}}_q, a_{1, i}]_q -a_{2, i}\lr{a_{i+1, n}, a_{1, k}}_q +a_{i+1, k}(a_{2, i} a_{1, n}-\lr{ a_{2, n}, a_{1, i}}_q)\bigr]\\
\qquad \xlongequal{\lrb{\ref{eqn:r0}}} \bigl[{a_{1,i}}^{\d_0}, [\lr{a_{i+1, n}, a_{2, k}}_q, a_{1, i}]_q -a_{2, i}\lr{a_{i+1, n}, a_{1, k}}_q +a_{i+1, k}({a_{1,i}}^{\d_0}-a_{i+1,n}a_{1,1})\bigr]\\
\qquad \xlongequal{\lrb{\textrm{R1}}} \bigl[{a_{1,i}}^{\d_0}, [\lr{a_{i+1, n}, a_{2, k}}_q, a_{1, i}]_q -a_{2, i}\lr{a_{i+1, n}, a_{1, k}}_q \bigr]\\
\qquad \xlongequal{\lrb{\ref{equ:comm2}}} \bigl[{a_{1,i}}^{\d_0}, [a_{i+1, n}, \lr{a_{2, k}, a_{1, i}}_q]_q -a_{2, i}\lr{a_{i+1, n}, a_{1, k}}_q \bigr]\\
 \qquad\xlongequal{\quad} \bigl[{a_{1,i}}^{\d_0}, \lr{a_{i+1, n}, \lr{a_{2, k}, a_{1, i}}_q-a_{2, i}a_{1, k}}_q \bigr]\\
\qquad \xlongequal{\lrb{\ref{eq:qjacobi3}}} -\bigl[a_{i+1, n}, \bigl[\lr{a_{2, k}, a_{1, i}}_q -a_{2, i}a_{1, k}, {a_{1,i}}^{\d_0} \bigr]_q \bigr] \\
\phantom{\qquad \xlongequal{\qquad}\!\! }{} - \bigl[\lr{a_{2, k}, a_{1, i}}_q-a_{2, i}a_{1, k},
 \bigl[{a_{1,i}}^{\d_0}, a_{i+1,n}\bigr]_q \bigr]\\
\qquad \xlongequal[\eqref{eqn:com4}]{\eqref{equ:comm3}} -{a_{1,i}}^{\d_0}[_{i+1, n}, \lr{a_{2, k}, a_{1, i}}_q-a_{2, i}a_{1, k} ]-{a_{1,i}}^{\d_0}[ \lr{a_{2, k}, a_{1, i}}_q-a_{2, i}a_{1, k} , a_{i+1,n}]\\
 \qquad\xlongequal{\eqref{eq:jacobi}} 0.
\end{gather*}
When $\d_0(a_{k,l})=a_{k,l}$, we have
\begin{align*}
[\d_0(a_{1,i}),\d_0(a_{j,k})]&\xlongequal{(\d_0)}\bigl[a_{j,i}^{\d_0}, a_{j,k}\bigr]\xlongequal{\eqref{eqn:r0}}[a_{j,k},-\lr{a_{2,n},a_{1,i}}_q+a_{i+1,n}a_{1,1}+a_{2,i}a_{1,n}]\\
&\xlongequal{\quad}-[\lr{a_{2,n}, a_{1,i}}_q, a_{j,k}]+\lr{a_{i+1,n}a_{1,1}, a_{j,k}}+\lr{a_{2,i}a_{1,n}, a_{j,k}}\\
 &\xlongequal{(\textrm{R1})}0.
 \end{align*}
Consequently, if $\lr{a_{i,j}, a_{k,l}}=0$, then $\bigl[{a_{i,j}}^{\d_0},{a_{k,l}}^{\d_0}\bigr]=0$, which implies that
\begin{equation*}
\bigl[{a_{i,j}}^{\d_0},{a_{k,l}}^{\d_0}\bigr]_q\xlongequal{\eqref{equ:comm1}} {a_{i,j}}^{\d_0}{a_{k,l}}^{\d_0}.
\end{equation*}
In the subsequent, we directly use this fact without further explanation.

The relations (R2): If we have chosen the submatrix \smash{$\A_{1,k,i+1}^{k-1,i,j}$} as $\eqref{eqn-A-1}$, the proof deduces to the simpler case when $k=2$ or $j=n$ and straightforward.
In more tedious calculations, we concentrate on exploring the general cases where $k\neq 2$ and $j \neq n$ and have
\begin{gather*}
 [ [\d_0 (a_{k,j} ),\d_0 (a_{1,i} ) ]_q, \d_0 (a_{k,j} ) ]_q\\
\qquad \xlongequal[\eqref{equ:comm2}]{\eqref{eqn:r0}} - [a_{2,n}, [ [a_{k,j}, a_{1,i} ]_q, a_{k,j} ]_q ]_q +a_{1,1} [ [a_{k,j}, a_{i+1,n} ]_q, a_{k,j} ]_q +a_{1,n} [ [a_{k,j}, a_{2,i} ]_q, a_{k,j} ]_q\\
\qquad \xlongequal{\eqref{rela-2}} - \bigl[a_{2,n},f\bigl(\A_{1,k,i+1}^{k-1,i,j}\bigr) \bigr]_q
+a_{1,1}g\bigl(\A_{k,i+1,n}^{i,j,n}\bigr) +a_{1,n}f\bigl(\A_{2,k,i+1}^{k-1,i,j}\bigr)\\
 \qquad\xlongequal[\eqref{eqn:sharp}]{\eqref{eqn:natural}} - [a_{2,n}, a_{1,i} ]_q+a_{1,1}a_{i+1,n}+a_{2,i}a_{1,n} \\
\phantom{\qquad \xlongequal{\qquad}\!\! }{} +a_{i+1,j}a_{k,j} (- [a_{2,n}, a_{1,k-1} ]_q+a_{1,1}a_{k,n} +a_{2,k-1}a_{1,n} ) \\
\phantom{\qquad \xlongequal{\qquad}\!\! }{} +a_{k,i}a_{k,j} (- [a_{2,n}, a_{1,j} ]_q+a_{1,1}a_{n,n}+a_{2,j}a_{1,n} ) \\
\phantom{\qquad \xlongequal{\qquad}\!\! }{} -a_{k,i} (- [a_{2,n},a_{1,k-1} ]_q+a_{1,1}a_{k,n}+a_{2,k-1}a_{1,n} ) \\
\phantom{\qquad \xlongequal{\qquad}\!\! }{} -a_{i+1,j} (- [a_{2,n}, a_{1,j} ]_q+a_{1,1}a_{n,n}+a_{2,j}a_{1,n} ) \\
 \qquad\xlongequal{\eqref{eqn:r0}} a_{1,i}^{\d_0}+ a_{i+1,j}a_{k,j} a_{1,k-1}^{\d_0} +a_{k,i}a_{k,j}a_{1,j}^{\d_0}-a_{k,i}a_{1,k-1}^{\d_0}-a_{i+1,j}a_{1,j}^{\d_0}\\
 \qquad\xlongequal{(\d_0)} \d_0\lrb{a_{1,i}}+\d_0\lrb{a_{1,k-1}}\d_0\lrb{a_{i+1,j}}\d_0\lrb{a_{k,j}} +\d_0\lrb{a_{k,i}}\d_0\lrb{a_{1,j}}\d_0\lrb{a_{k,j}}\\
 \phantom{ \qquad\xlongequal{\qquad}\!\! }{} -\d_0\lrb{a_{1,k-1}}\d_0\lrb{a_{k,i}} -\d_0\lrb{a_{i+1,j}}\d_0\lrb{a_{1,j}}
\end{gather*}
and
\begin{gather*}
 \lr{\d_0\lrb{a_{1,i}},\lr{\d_0\lrb{a_{k,j}},\d_0\lrb{a_{1,i}}}_q}_q \\
\qquad \xlongequal[(\d_0)]{\eqref{eqn:r0}} - \bigl[{a_{1,i}}^{\d_0}, [a_{2,n}, [a_{k,j},a_{1,i} ]_q ]_q +a_{1,1} [a_{k,j},a_{i+1,n} ]_q +a_{1,n} [a_{k,j},a_{2,i} ]_q \bigr]_q\\
\qquad \xlongequal{\eqref{eq:qjacobi1}} \left(- \bigl[ \bigl[{a_{1,i}}^{\d_0},a_{2,n} \bigr]_q, [a_{k,j},a_{1,i} ]_q \bigr]_q+\frac{1}{ \bigl(q-q^{-1}\bigr)^2} \bigl[a_{2,n}, \bigl[ [a_{k,j},a_{1,i} ]_q, {a_{1,i}}^{\d_0} \bigr] \bigr] \right)\\
 \phantom{\qquad \xlongequal{\qquad}}{} +a_{1,1} \bigl[{a_{1,i}}^{\d_0}, [a_{k,j},a_{i+1,n} ]_q \bigr]_q +a_{1,n} \bigl[{a_{1,i}}^{\d_0}, [a_{k,j},a_{2,i} ]_q \bigr]_q\\
\qquad \xlongequal[\eqref{rela-2}]{\eqref{eqn:r0}} [a_{1,i}, [a_{k,j},a_{1,i} ]_q ]_q-a_{1,1} [a_{2,i}, [a_{k,j},a_{1,i} ]_q ]_q -a_{1,n} [a_{i+1,n}, [a_{k,j},a_{1,i} ]_q ]_q \\
 \phantom{\qquad \xlongequal{\qquad}}{} +a_{1,1} \bigl[{a_{1,i}}^{\d_0}, [a_{k,j},a_{i+1,n} ]_q \bigr]_q+a_{1,n} \bigl[{a_{1,i}}^{\d_0}, [a_{k,j},a_{2,i} ]_q \bigr]_q \\
 \phantom{\qquad \xlongequal{\qquad}}{} +\frac{1}{ \bigl(q-q^{-1}\bigr)^2} \bigl[a_{2,n}, [ [a_{k,j},a_{1,i} ]_q, {a_{1,i}}^{\d_0} ] \bigr]\\
 \qquad\xlongequal{\eqref{rela-2}} g\bigl(\A_{1,k,i+1}^{k-1,i,j}\bigr)-a_{1,1} [a_{2,i}, [a_{k,j},a_{1,i} ]_q ]_q -a_{1,n} [a_{i+1,n}, [a_{k,j},a_{1,i} ]_q ]_q \\
 \phantom{\qquad \xlongequal{\qquad}}{} +a_{1,1} \bigl[{a_{1,i}}^{\d_0}, [a_{k,j},a_{i+1,n} ]_q \bigr]_q +a_{1,n} \bigl[{a_{1,i}}^{\d_0}, [a_{k,j},a_{2,i} ]_q \bigr]_q\\
 \phantom{\qquad \xlongequal{\qquad}}{} +\frac{1}{ \bigl(q-q^{-1}\bigr)^2} \bigl[a_{2,n}, \bigl[ [a_{k,j},a_{1,i} ]_q, {a_{1,i}}^{\d_0} \bigr] \bigr]\\
\qquad \xlongequal{\eqref{eqn:sharp}} a_{k,j}-a_{2,i}a_{i+1,j} -a_{1,1} [a_{2,i}, [a_{k,j},a_{1,i} ]_q ]_q -a_{1,n} [a_{i+1,n}, [a_{k,j},a_{1,i} ]_q ]_q \\
 \phantom{\qquad \xlongequal{\qquad}}{} +a_{1,1} \bigl[{a_{1,i}}^{\d_0}, [a_{k,j},a_{i+1,n} ]_q \bigr]_q +a_{1,n} \bigl[{a_{1,i}}^{\d_0}, [a_{k,j},a_{2,i} ]_q \bigr]_q+a_{1,i}a_{1,k-1}a_{i+1,j}\\
 \phantom{\qquad \xlongequal{\qquad}}{} +a_{1,i}a_{k,i}a_{1,j} -a_{1,k-1}a_{1,j}+\frac{1}{ \bigl(q-q^{-1}\bigr)^2} \bigl[a_{2,n}, \bigl[ [a_{k,j},a_{1,i} ]_q, {a_{1,i}}^{\d_0} \bigr] \bigr]\\
 \qquad\xlongequal[\eqref{r0-3}]{\eqref{r0-1}\eqref{r0-2} } a_{k,j}-a_{2,i}a_{i+1,j} -a_{1,1} [a_{2,i}, [a_{k,j},a_{1,i} ]_q ]_q -a_{1,n} [a_{i+1,n}, [a_{k,j},a_{1,i} ]_q ]_q \\
 \phantom{\qquad \xlongequal{\qquad}\quad\ }{} +a_{1,1} \bigl[{a_{1,i}}^{\d_0}, [a_{k,j},a_{i+1,n} ]_q \bigr]_q +a_{1,n} \bigl[{a_{1,i}}^{\d_0}, [a_{k,j},a_{2,i} ]_q \bigr]_q \\
 \phantom{\qquad \xlongequal{\qquad}\quad\ }{} +a_{i+1,j} \bigl({a_{1,i}}^{\d_0}{a_{1,k-1}}^{\d_0}+a_{1,1} [a_{2,i},a_{1,k-1} ]_q\\
 \phantom{\qquad \xlongequal{\qquad}\quad\ }{}
+a_{i+1,n}a_{1,n}a_{1,k-1} -a_{1,1} \bigl[{a_{1,i}}^{\d_0},a_{k,n} \bigr]_q \\
 \phantom{\qquad \xlongequal{\qquad}\quad\ }{} -a_{2,k-1}{a_{1,i}}^{\d_0}a_{1,n} \bigr)+a_{k,i} \bigl({a_{1,i}}^{\d_0}{a_{1,j}}^{\d_0}+a_{1,1}a_{2,i}a_{1,j} +a_{1,n} [a_{i+1,n},a_{1,j} ]_q \\
 \phantom{\qquad \xlongequal{\qquad}\quad\ }{} -a_{1,1}{a_{1,i}}^{\d_0}a_{j+1,n} -a_{1,n} \bigl[{a_{1,i}}^{\d_0},a_{2,j} \bigr]_q \bigr)- \bigl({a_{1,k-1}}^{\d_0}{a_{1,j}}^{\d_0} +a_{1,1}a_{2,k-1}a_{1,j}\\
 \phantom{\qquad \xlongequal{\qquad}\quad\ }{}
+a_{1,n} [a_{k,n},a_{1,j} ]_q -a_{1,1}{a_{1,k-1}}^{\d_0}a_{j+1,n} -a_{1,n} \bigl[{a_{1,k-1}}^{\d_0},a_{2,j} \bigr]_q \bigr)\\
 \phantom{\qquad \xlongequal{\qquad}\quad\ }{} +\frac{1}{ \bigl(q-q^{-1}\bigr)^2} \bigl( \bigl[a_{2,n}, \bigl[ [a_{k,j},a_{1,i} ]_q-a_{i+1,j}a_{1,k-1} -a_{k,i}a_{1,j}, {a_{1,i}}^{\d_0} \bigr]\\
 \phantom{\qquad \xlongequal{\qquad}\quad\ }{} + \bigl[a_{1,j}, {a_{1,k-1}}^{\d_0} \bigr] \bigr] \bigr)\\
 \qquad\xlongequal{\eqref{eq:qjacobi0}} a_{k,j}-a_{2,i}a_{i+1,j} +a_{i+1,j}{a_{1,i}}^{\d_0} {a_{1,k-1}}^{\d_0} +a_{k,i}{a_{1,i}}^{\d_0}{a_{1,j}}^{\d_0}-{a_{1,k-1}}^{\d_0}{a_{1,j}}^{\d_0}\\
 \phantom{\qquad \xlongequal{\qquad}}{} -a_{1,1} ( [a_{2,i}, [a_{k,j},a_{1,i} ]_q -a_{i+1,j}a_{1,k-1}-a_{k,i}a_{1,j} ]_q+{a_{1,k-1}}^{\d_0}a_{j+1,n} )\\
 \phantom{\qquad \xlongequal{\qquad}}{} +a_{1,1} \bigl( \bigl[{a_{1,i}}^{\d_0}, [a_{k,j},a_{i+1,n} ]_q -a_{i+1,j}a_{k,n} -a_{k,i}a_{j+1,n} \bigr]_q+a_{2,k-1}a_{1,j} \bigr) \\
 \phantom{\qquad \xlongequal{\qquad}}{} -a_{1,n} ( [a_{i+1,n}, [a_{k,j},a_{1,i} ]_q -a_{i+1,j}a_{1,k-1}-a_{k,i}a_{1,j} ]_q + [a_{k,n},a_{1,j} ]_q )\\
 \phantom{\qquad \xlongequal{\qquad}}{} +a_{1,n} \bigl( [{a_{1,i}}^{\d_0}, [a_{k,j},a_{2,i} ]_q -a_{k,i}a_{2,j} -a_{2,k-1}a_{i+1,j} ]_q + \bigl[{a_{1,k-1}}^{\d_0},a_{2,j} \bigr]_q \bigr)\\
 \qquad\xlongequal[\eqref{eqn:02}]{\eqref{eqn:01}}
 a_{k,j}-a_{2,i}a_{i+1,j}+a_{i+1,j}a_{1,i}^{\d_0}a_{1,k-1}^{\d_0} +a_{k,i} a_{1,i}^{\d_0}a_{1,j}^{\d_0} - a_{1,k-1}^{\d_0}a_{1,j}^{\d_0}\\
 \qquad\xlongequal{\quad} \d_0\lrb{a_{k,j}}-\d_0\lrb{a_{2,i}}\d_0\lrb{a_{i+1,j}} +\d_0\lrb{a_{i+1,j}}\d_0\lrb{a_{1,i}}\d_0\lrb{a_{1,k-1}}\\
 \phantom{ \qquad\xlongequal{\quad}}{} +\d_0\lrb{a_{k,i}}\d_0\lrb{a_{1,i}}\d_0\lrb{a_{1,j}} -\d_0\lrb{a_{1,k-1}}\d_0\lrb{a_{1,j}}.
\end{gather*}
Hence,
\begin{gather*}
f\bigl(\d_0\bigl(\A_{1,k,i+1}^{k-1,i,j}\bigr)\bigr)= [\d_0\lrb{a_{k, j}}, {\lr{\d_0\lrb{a_{1, i}}, \d_0\lrb{a_{k, j}}}}_q]_q,\\
 g\bigl(\d_0\bigl(\A_{1,k,i+1}^{k-1,i,j}\bigr)\bigr)=[\d_0\lrb{a_{1, i}}, {\lr{\d_0\lrb{a_{k, j}}, \d_0\lrb{ a_{1,i}}}}_q]_q.
\end{gather*}

The relations (R3):
Choosing the submatrix \smash{$\A_{1,k,i+1}^{i,j,l}(l\neq n)$}, we have
 \begin{gather*}
 [\lr{\d_0\lrb{a_{i+1,l}}, \d_0\lrb{a_{k,j}}}_q, \d_0\lrb{a_{1,i}}]_q+ \lr{\d_0\lrb{a_{k,l}}, \d_0\lrb{a_{1,j}}}_q +\d_0\lrb{a_{i+1,j}}\d_0\lrb{a_{k,i}} \d_0\lrb{a_{1,l}} \\
\qquad -\d_0\lrb{a_{k,i}}\lr{\d_0\lrb{a_{i+1,l}}, \d_0\lrb{a_{1,j}}}_q -\d_0\lrb{a_{i+1,j}}\lr{\d_0\lrb{a_{k,l}}, \d_0\lrb{a_{1,i}}}_q -\d_0\lrb{a_{k,j}} \d_0\lrb{a_{1,l}}\\
\qquad\qquad\xlongequal[(\d_0)]{\eqref{eqn:r0}} [ [a_{i+1,l}, a_{k,j} ]_q, - [a_{2,n},a_{1,i} ]_q+a_{1,1}a_{i+1,n}+a_{2,i}a_{1,n} ]_q\\
\phantom{\qquad\qquad\qquad\ }{} + [a_{k,l}, - [a_{2,n},a_{1,j} ]_q+a_{1,1}a_{j+1,n}+a_{2,j}a_{1,n} ]_q\\
\phantom{\qquad\qquad\qquad\ }{} +a_{i+1,j}a_{k,i} (- [a_{2,n},a_{1,l} ]_q+a_{1,1}a_{l+1,n}+a_{2,l}a_{1,n} ) \\
\phantom{\qquad\qquad\qquad\ }{} -a_{k,i} [a_{i+1,l}, - [a_{2,n},a_{1,j} ]_q+a_{1,1}a_{j+1,n}+a_{2,j}a_{1,n} ]_q \\
\phantom{\qquad\qquad\qquad\ }{}-a_{i+1,j} [a_{k,l},- [a_{2,n},a_{1,i} ]_q+a_{1,1}a_{i+1,n}+a_{2,i}a_{1,n} ]_q \\
\phantom{\qquad\qquad\qquad\ }{} -a_{k,j} (- [a_{2,n},a_{1,l} ]_q+a_{1,1}a_{l+1,n}+a_{2,l}a_{1,n} )\\
\qquad\qquad \xlongequal{\eqref{equ:comm2}} - [a_{2,n}, [ [a_{i+1,l}, a_{k,j} ]_q,a_{1,i} ]_q ]_q +a_{1,1} [ [a_{i+1,l}, a_{k,j} ]_q, a_{i+1,n} ]_q \\
 \phantom{\qquad\qquad \qquad\ }{} +a_{1,n} [ [a_{i+1,l}, a_{k,j} ]_q, a_{2,i} ]_q - [a_{2,n}, [a_{k,l},a_{1,j} ]_q ]_q +a_{1,1} [a_{k,l},a_{j+1,n} ]_q\\
 \phantom{\qquad\qquad \qquad\ }{} +a_{1,n} [a_{k,l}, a_{2,j} ]_q - [a_{2,n},a_{i+1,j}a_{k,i}a_{1,l} ]_q +a_{1,1}a_{i+1,j}a_{k,i}a_{l+1,n} \\
 \phantom{\qquad\qquad \qquad\ }{}
 +a_{i+1,j}a_{k,i}a_{2,l}a_{1,n} \\
 \phantom{\qquad\qquad \qquad\ }{} + [a_{2,n},a_{k,i} [a_{i+1,l},a_{1,j} ]_q ]_q-a_{1,1}a_{k,i} [a_{i+1,l}, a_{j+1,n} ]_q -a_{k,i}a_{1,n} [a_{i+1,l},a_{2,j} ]_q\\
 \phantom{\qquad\qquad \xlongequal{\eqref{equ:comm2}} }{} + [a_{2,n},a_{i+1,j} [a_{k,l},a_{1,i} ]_q ]_q -a_{1,1}a_{i+1,j} [a_{k,l},a_{i+1,n} ]_q -a_{i+1,j}a_{1,n} [a_{k,l},a_{2,i} ]_q \\
 \phantom{\qquad\qquad \xlongequal{\eqref{equ:comm2}} }{}+ [a_{2,n},a_{k,j}a_{1,l} ]_q-a_{1,1}a_{l+1,n}a_{k,j}-a_{k,j}a_{2,l}a_{1,n}\\
 \qquad\qquad \xlongequal{\eqref{eqn:D-2}} - \bigl[a_{2,n},{\det}^q \bigl(\A_{1,k,i+1}^{i,j,l} \bigr) \bigr]_q +a_{1,1} ( [a_{i+1,l}, [a_{k,j}, a_{i+1,n} ]_q ]_q +a_{i+1,j}a_{k,i}a_{l+1,n} \\
 \phantom{\qquad\qquad \xlongequal{\eqref{eqn:D-2}}}{} + [a_{k,l},a_{j+1,n} ]_q-a_{k,i} [a_{i+1,l}, a_{j+1,n} ]_q -a_{i+1,j} [a_{k,l},a_{i+1,n} ]_q -a_{l+1,n}a_{k,j} )\\
 \phantom{\qquad\qquad \xlongequal{\eqref{eqn:D-2}}}{}+a_{1,n}{\det}^q \bigl(\A_{2,k,i+1}^{i,j,l} \bigr)\\
\qquad\qquad \xlongequal[\eqref{eqn:6a}]{\eqref{rela-3}} 0
\end{gather*}
and
\begin{gather*}
 [\lr{\d_0\lrb{a_{1,i}}, \d_0\lrb{a_{k,j}}}_q, \d_0\lrb{a_{i+1,l}}]_q+ \lr{\d_0\lrb{a_{1,j}}, \d_0\lrb{a_{k,l}}}_q +\d_0\lrb{a_{k,i}}\lr{\d_0\lrb{a_{i+1,j}}, \d_0\lrb{a_{1,l}}}_q \\
 \qquad -\d_0\lrb{a_{i+1,j}}\lr{\d_0\lrb{a_{1,i}}, \d_0\lrb{a_{k,l}}}_q -\d_0\lrb{a_{k,i}}\lr{\d_0\lrb{a_{1,j}}, \d_0\lrb{a_{i+1,l}}}_q-\d_0\lrb{a_{k,j}} \d_0\lrb{a_{1,l}}\\
 \qquad\qquad\xlongequal[(\d_0)]{\eqref{eqn:r0}} [ - [a_{2,n},a_{1,i} ]_q +a_{1,1}a_{i+1,n}+a_{2,i}a_{1,n}, [a_{k,j}, a_{i+1,l} ]_q ]_q\\
\phantom{\qquad\qquad\qquad\ \,}{} + [ - [a_{2,n},a_{1,j} ]_q+a_{1,1}a_{j+1,n}+a_{2,j}a_{1,n}, a_{k,l} ]_q \\
\phantom{\qquad\qquad\qquad\ \,}{} +a_{k,i}a_{i+1,j} (- [a_{2,n},a_{1,l} ]_q+a_{1,1}a_{l+1,n}+a_{2,l}a_{1,n} ) \\
 \phantom{\qquad\qquad\qquad\ \,}{} -a_{i+1,j} [ - [a_{2,n},a_{1,i} ]_q+a_{1,1}a_{i+1,n}+a_{2,i}a_{1,n}, a_{k,l} ]_q \\
 \phantom{\qquad\qquad\qquad\ \, }{} -a_{k,i} [- [a_{2,n},a_{1,j} ]_q+a_{1,1}a_{j+1,n}+a_{2,j}a_{1,n}, a_{i+1,l} ]_q\\
 \phantom{\qquad\qquad\qquad\ \, }{} -a_{k,j} (- [a_{2,n},a_{1,l} ]_q+a_{1,1}a_{l+1,n}+a_{2,l}a_{1,n} ) \\
\qquad\qquad \xlongequal{\eqref{equ:comm2}} - [a_{2,n}, [a_{1,i}, [a_{k,j}, a_{i+1,l} ]_q ]_q ]_q +a_{1,1} [a_{i+1,n}, [a_{k,j}, a_{i+1,l} ]_q ]_q \\
 \phantom{\qquad\qquad\qquad\ \,}{} +a_{1,n} [a_{2,i}, [a_{k,j}, a_{i+1,l} ]_q ]_q - [a_{2,n}, [a_{1,j}, a_{k,l} ]_q ]_q +a_{1,1} [a_{j+1,n},a_{k,l} ]_q\\
 \phantom{\qquad\qquad\qquad\ \,}{} +a_{1,n} [a_{2,j}, a_{k,l} ]_q - [a_{2,n},a_{k,i}a_{i+1,j}a_{1,l} ]_q +a_{1,1}a_{k,i}a_{i+1,j}a_{l+1,n} \\
 \phantom{\qquad\qquad\qquad\ \,}{} +a_{k,i}a_{i+1,j}a_{2,l}a_{1,n} + [a_{2,n},a_{i+1,j} [a_{1,i}, a_{k,l} ]_q ]_q -a_{i+1,j}a_{1,1} [a_{i+1,n}, a_{k,l} ]_q\\
 \phantom{\qquad\qquad\qquad\ \,}{} -a_{i+1,j}a_{1,n} [a_{2,i}, a_{k,l} ]_q + [a_{2,n},a_{k,i} [a_{1,j}, a_{i+1,l} ]_q ]_q-a_{k,i}a_{1,1} [a_{j+1,n}, a_{i+1,l} ]_q \\
 \phantom{\qquad\qquad\qquad\ \,}{} -a_{k,i}a_{1,n} [a_{2,j}, a_{i+1,l} ]_q + [a_{2,n},a_{k,j}a_{1,l} ]_q-a_{1,1}a_{l+1,n}a_{k,j}-a_{k,j}a_{2,l}a_{1,n}\\
\qquad\qquad \xlongequal{\eqref{eqn:D-1}} - \bigl[a_{2,n},{\det}_q \bigl(\A_{1,k,i+1}^{i,j,l} \bigr) \bigr]_q +a_{1,1} ( [a_{i+1,n}, [a_{k,j}, a_{i+1,l} ]_q ]_q + [a_{j+1,n},a_{k,l} ]_q \\
 \phantom{\qquad\qquad\qquad\ \, \ }{} +a_{k,i}a_{i+1,j}a_{l+1,n} -a_{i+1,j} [a_{i+1,n}, a_{k,l} ]_q-a_{k,i} [a_{j+1,n}, a_{i+1,l} ]_q\\
 \phantom{\qquad\qquad\qquad\ \,\ }{} -a_{1,1}a_{l+1,n}a_{k,j} ) +a_{1,n}{\det}_q \bigl(\A_{2,k,i+1}^{i,j,l} \bigr)\\
\qquad\qquad \xlongequal[\eqref{eqn:6b}]{\eqref{rela-3}}0.
\end{gather*}
Hence,
\begin{gather*}
{\det}^q \bigl(\d_0 \bigl(\A_{1,k,i+1}^{i,j,l}\bigr) \bigr) ={\det}_q \bigl(\d_0 \bigl(\A_{1,k,i+1}^{i,j,l}\bigr) \bigr)=0.
\end{gather*}
In particular, when $l=n$, we have
\begin{gather*}
{\det}^q \bigl(\d_0 \bigl(\A_{1,k,i+1}^{i,j,n} \bigr) \bigr) ={\det}_q \bigl(\d_0 \bigl(\A_{1,k,i+1}^{i,j,n} \bigr) \bigr)=0.
\end{gather*}
Therefore, $\d_0$ is an algebra homomorphism of $\A(n)$.
Similarly, $\d_0'$ is also an algebra homomorphism.
In fact, we also have
\begin{gather*}
\d_0\d_0'\lrb{a_{1,1}} = \d_0\lrb{a_{1,n}}=a_{1,1},\qquad
\d_0\d_0'\lrb{a_{1,n}} = \d_0\lrb{a_{1,1}}=a_{1,n},\\
\d_0\d_0'\lrb{a_{1,i}},
 = \d_0\bigl({a_{1,i}}^{\d_0'}\bigr)=a_{1,i}+[[a_{2,n}, a_{1,i}]_q, a_{2,n}]_q -f\bigl(\A_{1,2,i+1}^{1,i,n}\bigr)
=a_{1,i},\\
\d_0\d_0'\lrb{a_{i,l}} = \d_0\lrb{a_{i,l}}=a_{i,l},
\end{gather*}
where $i \neq 1$ and $l \in \lbb{1,n}$.
This means that $\d_0\d_0'={\rm id}$.
Similarly, $\d_0'\d_0={\rm id}$.

 The proof is finished.
\end{proof}

\begin{Proposition}
$\d_i$, $\d_i'$, $i \in \lbb{1,n-1}$ are the automorphisms of algebra $\A(n)$ and $\d_i\d_i'=\d_i'\d_i\allowbreak=\id$.
\end{Proposition}
\begin{proof}
To prove that $\d_i$ is an algebra homomorphism, it is sufficient to demonstrate that $\d_i$ keeps the relations of $\A(n).$
It is similar for $\d_i'$.
By definition of $\d_i$, we only need to show that~$\d_i$ keeps those relations that concern with the $(i+1)$-th row's generators $a_{i+1, k}$ and $i$-column's generators $a_{j,i}$.

The relations (R1):
Noting that $a_{i,i}$ and $a_{i+1,i+1}$ belong to the center of $\A(n)$, we easily see that
$\d_i\lrb{a_{i,i}}=a_{i+1,i+1}, \d_i\lrb{a_{i+1,i+1}}=a_{i,i}.$
 Consequently, $\d_i$ must keep the relations associated with these elements.

If $\lr{a_{j,i}, a_{k,l}}=0$ and $k>j$, we have for $l=i$ that
 \begin{align*}
 [\d_i (a_{j,i}), \d_i(a_{k,i} ) ] &
 \xlongequal[\quad]{\eqref{eqn:ri}} \bigl[- [a_{i,i+1}, a_{j,i} ]_q+a_{i+1,i+1}a_{j,i-1}, {a_{k,i}}^{\d_i} \bigr] \\
 &\xlongequal{\eqref{eqn:3b}} \bigl[ [a_{k,i+1}, [a_{j,i-1}, a_{k,i} ]_q ]_q -a_{k,i-1} [a_{k,i+1}, a_{j,i} ]_q \\
 &\qquad \quad -a_{j,k-1} ( [a_{i,i+1}, a_{k,i} ]_q-a_{i+1,i+1}a_{k,i-1} ) , {a_{k,i}}^{\d_i} \bigr]\\
& \xlongequal{\eqref{eqn:ri}} \bigl[ [a_{k,i+1}, [a_{j,i-1}, a_{k,i} ]_q ]_q-a_{k,i-1} [a_{k,i+1}, a_{j,i} ]_q \\
 &\qquad \quad +a_{j,k-1} \bigl({a_{k,i}}^{\d_i}-a_{i,i}a_{k,i+1} \bigr) , {a_{k,i}}^{\d_i} \bigr] \\
 &\xlongequal{(\textrm{R1})} \bigl[ [a_{k,i+1}, [a_{j,i-1}, a_{k,i} ]_q ]_q-a_{k,i-1} [a_{k,i+1}, a_{j,i} ]_q, {a_{k,i}}^{\d_i} \bigr] \\
 & \xlongequal{\quad} \bigl[ [a_{k,i+1}, [a_{j,i-1}, a_{k,i} ]_q-a_{k,i-1}a_{j,i} ]_q, {a_{k,i}}^{\d_i} \bigr] \\
 &\xlongequal{\eqref{eq:qjacobi3}} - \bigl[ \bigl[ [a_{j,i-1}, a_{k,i} ]_q- a_{k,i-1}a_{j,i}, {a_{k,i}}^{\d_i} \bigr]_q, a_{k,i+1} \bigr] \\
 &\qquad\ \ - \bigl[ \bigl[{a_{k,i}}^{\d_i}, a_{k,i+1} \bigr]_q, [a_{j,i-1}, a_{k,i} ]_q-a_{k,i-1}a_{j,i} \bigr] \\
 &\xlongequal[\eqref{eqn:com1}]{\eqref{equ:comm3}} -{a_{k,i}}^{\d_i} [ [a_{j,i-1}, a_{k,i} ]_q-a_{k,i-1}a_{j,i}, a_{k,i+1} ] \\
 &\qquad\quad \, -{a_{k,i}}^{\d_i} [a_{k,i+1}, [a_{j,i-1}, a_{k,i} ]_q-a_{k,i-1}a_{j,i} ] \\
 &\xlongequal{\eqref{eq:jacobi}} 0.
 \end{align*}
Similarly, we can prove that
$[\d_i(a_{i+1,j}), \d_i(a_{i+1,k}) ]=0$.
If $k=i+1$, we have
 \begin{align*}
 [\d_i(a_{j,i}), \d_i(a_{i+1, l}) ]&
 \xlongequal[\quad]{\eqref{eqn:ri}} \bigl[- [a_{i,i+1}, a_{j,i} ]_q+a_{i,i}a_{j,i+1}, {a_{i+1,l}}^{\d_i} \bigr]\\
 &\xlongequal{\eqref{eqn:2a}} \bigl[ [a_{i,l}, [a_{j,i+1}, a_{i+1,l} ]_q ]_q-a_{i+2,l} [a_{i,l}, a_{j,i} ]_q \\
 &\qquad\ \ -a_{j,l} ( [a_{i,i+1}, a_{i+1,l} ]_q+a_{i,i}a_{i+2,l} ), {a_{i+1,l}}^{\d_i} \bigr] \\
 &\xlongequal{\eqref{eqn:ri}} \bigl[ [a_{i,l}, [a_{j,i+1}, a_{i+1,l} ]_q ]_q-a_{i+2,l} [a_{i,l}, a_{j,i} ]_q\\
&\qquad \ \ +a_{j,l} ({a_{i+1,l}}^{\d_i}-a_{i+1,i+1}a_{i,l} ), {a_{i+1,l}}^{\d_i} \bigr] \\
& \xlongequal{(R1)} \bigl[ [a_{i,l}, [a_{j,i+1}, a_{i+1,l} ]_q ]_q-a_{i+2,l} [a_{i,l}, a_{j,i} ]_q, {a_{i+1,l}}^{\d_i} \bigr] \\
 & \xlongequal{\quad} \bigl[ [a_{i,l}, [a_{j,i+1}, a_{i+1,l} ]_q -a_{i+2,l}a_{j,i} ]_q, {a_{i+1,l}}^{\d_i} \bigr] \\
& \xlongequal{\eqref{eq:qjacobi3}} - \bigl[ \bigl[ [a_{j,i+1}, a_{i+1,l} ]_q -a_{i+2,l}a_{j,i}, {a_{i+1,l}}^{\d_i} \bigr]_q, a_{i,l} \bigr] \\
& \qquad\ \ - \bigl[ \bigl[{a_{i+1,l}}^{\d_i} , a_{i,l} \bigr]_q, [a_{j,i+1}, a_{i+1,l} ]_q -a_{i+2,l}a_{j,i} \bigr] \\
 &\xlongequal[\eqref{eqn:com3}]{\eqref{equ:comm3}} -{a_{i+1,l}}^{\d_i} [ [a_{j,i+1}, a_{i+1,l} ]_q -a_{i+2,l}a_{j,i}, a_{i,l} ]\\
 &\qquad\quad\, -{a_{i+1,l}}^{\d_i} [a_{i,l}, [a_{j,i+1}, a_{i+1,l} ]_q -a_{i+2,l}a_{j,i} ] \\
& \xlongequal{\eqref{eq:jacobi}} 0.
 \end{align*}
 If $\d_i(a_{k,l})=a_{k,l}$, we have
\begin{eqnarray*}
\lr{\d_i\lrb{a_{j,i}}, \d_i\lrb{a_{k,l}}}
=-[\lr{a_{i,i+1}, a_{j,i}}_q, a_{k,l}]+\lr{a_{i+1,i+1}a_{j,i-1}, a_{k,l}}+\lr{a_{i,i}a_{j,i+1}, a_{k,l}}=0.
 \end{eqnarray*}
Similarly, we can prove that
$
\lr{\d_i\lrb{a_{i+1,j}}, \d_i\lrb{a_{k,l}}}=0$.

The relations (R2):
If we have chosen the submatrix \smash{$\A_{j,k,i+1}^{k-1,i,l}$} as $\eqref{eqn-A-1}$, the proof deduce to the simpler case when $k=i$ or $l=i+1$ and straightforward.
In more tedious calculations, we concentrate on exploring the general cases, where $k \neq i$ and $l \neq i+1$,
\begin{gather*}
 [\d_i (a_{k, l} ), [\d_i (a_{j, i} ),\d_i (a_{k, l} ) ]_q ]_q\\
\qquad \xlongequal{\eqref{eqn:ri}} - [ a_{k, l} , [ [a_{i, i+1}, a_{j, i} ]_q , a_{k, l} ]_q ]_q +a_{i+1, i+1} [ a_{k, l} , [a_{j, i-1} , a_{k, l} ]_q ]_q\\
 \qquad\qquad \quad{} +a_{i,i} [ a_{k, l} , [a_{j, i+1} , a_{k, l} ]_q ]_q\\
\qquad \xlongequal{\eqref{rela-2}} {}- [ a_{k, l} , [ [a_{i, i+1}, a_{j, i} ]_q , a_{k, l} ]_q ]_q +a_{i+1, i+1}f\bigl(\A_{j,k,i}^{k-1,i-1,l}\bigr)+a_{i,i}f\bigl(\A_{j,k,i+2}^{k-1,i+1,l}\bigr)\\
\qquad \xlongequal{\eqref{equ:comm2}} - [a_{i, i+1}, [a_{k, l}, [a_{j, i}, a_{k, l} ]_q ]_q ]_q +a_{i+1, i+1}f\bigl(\A_{j,k,i}^{k-1,i-1,l}\bigr) +a_{i,i}f\bigl(\A_{j,k,i+2}^{k-1,i+1,l}\bigr)\\
\qquad \xlongequal{\eqref{rela-2}} - \bigl[a_{i, i+1},f \bigl(\A_{j,k,i+1}^{k-1,i,l} \bigr) \bigr]_q +a_{i+1, i+1}f\bigl(\A_{j,k,i}^{k-1,i-1,l}\bigr) +a_{i,i}f \bigl(\A_{j,k,i+2}^{k-1,i+1,l} \bigr)\\
 \qquad\xlongequal{\eqref{eqn:natural}} - [a_{i, i+1}, a_{j, i} ]_q+a_{j, i-1}a_{i+1, i+1}+a_{i,i}a_{j, i+1} \\
 \qquad\qquad \quad {} +a_{j,k-1} a_{k, l} (- [a_{i, i+1}, a_{i+1, l} ]_q+a_{i,i}a_{i+2,l}+a_{i+1, i+1}a_{i, l} ) \\
 \qquad\qquad \quad{} + a_{j, l}a_{k, l} (- [a_{i, i+1}, a_{k,i} ]_q+a_{k, i-1}a_{i+1, i+1}+a_{i,i}a_{k, i+1} ) \\
 \qquad\qquad\quad{} -a_{j,k-1} (- [a_{i, i+1}, a_{k,i} ]_q+a_{k, i-1}a_{i+1, i+1}+a_{i,i}a_{k, i+1} )\\
 \qquad\qquad \quad{} -a_{j, l} (- [a_{i, i+1}, a_{i+1, l} ]_q+a_{i,i}a_{i+2,l}+a_{i+1, i+1}a_{i, l} ) \\
\qquad \xlongequal{\eqref{eqn:ri}} {a_{j, i}}^{\d_i} +a_{j,k-1}a_{k, l}{a_{i+1, l}}^{\d_i}+a_{j, l}a_{k,l} {a_{k,i}}^{\d_i}-a_{j,k-1}{a_{k,i}}^{\d_i} -a_{j, l}{a_{i+1, l}}^{\d_i}\\
 \qquad\xlongequal{\quad} \d_i (a_{j, i} ) +\d_i (a_{j,k-1} )\d_i (a_{i+1,l} ) \d_i (a_{k, l} )+\d_i (a_{k,i} )\d_i (a_{j, l} )\d_i (a_{k, l} ) \\
\qquad\qquad {}-\d_i (a_{j,k-1} )\d_i (a_{k,i} ) -\d_i (a_{j, l} ) \d_i (a_{i+1, l} )
\end{gather*}
and
\begin{gather*}
 [\d_i(a_{j,i}), [\d_i (a_{k,l} ),\d_i(a_{j,i}) ]_q ]_q \\
 \qquad\xlongequal[\quad]{\eqref{eqn:ri}} - \bigl[{a_{j,i}}^{\d_i}, [a_{i,i+1}, [a_{k,l},a_{j,i} ]_q ]_q \bigr]_q + a_{i+1,i+1} \bigl[{a_{j,i}}^{\d_i}, [a_{k,l},a_{j,i-1} ]_q ]_q \\
\qquad\qquad\ \ {}+ a_{i,i} \bigl[{a_{j,i}}^{\d_i}, [a_{k,l},a_{j,i+1} ]_q \bigr]_q \\
 \qquad\xlongequal{\eqref{eq:qjacobi1}} - \bigl[ \bigl[{a_{j,i}}^{\d_i},a_{i,i+1} \bigr]_q, [a_{k,l},a_{j,i} ]_q \bigr]_q +\frac{1}{\bigl(q-q^{-1}\bigr)^2} \bigl[a_{i,i+1}, \bigl[ [a_{k,l},a_{j,i} ]_q, {a_{j,i}}^{\d_i} \bigr] \bigr]\\
\qquad\qquad \ \ {}+ a_{i+1,i+1} \bigl[{a_{j,i}}^{\d_i}, [a_{k,l},a_{j,i-1} ]_q \bigr]_q + a_{i,i} \bigl[{a_{j,i}}^{\d_i}, [a_{k,l},a_{j,i+1} ]_q \bigr]_q \\
\qquad \xlongequal[\eqref{rela-2}]{\eqref{eqn:ri}} [a_{j,i}-a_{j,i-1}a_{i,i}-a_{i+1,i+1}a_{j,i+1}, [a_{k,l},a_{j,i} ]_q ]_q + a_{i+1,i+1} \bigl[{a_{j,i}}^{\d_i}, [a_{k,l},a_{j,i-1} \bigr]_q ]_q \\
\qquad\qquad \quad {}+ a_{i,i} \bigl[{a_{j,i}}^{\d_i}, [a_{k,l},a_{j,i+1} ]_q \bigr]_q +\frac{1}{\bigl(q-q^{-1}\bigr)^2} \bigl[a_{i,i+1}, \bigl[ [a_{k,l},a_{j,i} ]_q, {a_{j,i}}^{\d_i} \bigr] \bigr]\\
\qquad \xlongequal{\quad} [a_{j,i}, [a_{k,l},a_{j,i} ]_q ]_q -a_{i,i} [a_{j,i-1}, [a_{k,l},a_{j,i} ]_q ]_q-a_{i+1,i+1} [a_{j,i+1}, [a_{k,l},a_{j,i} ]_q ]_q\\
 \qquad\qquad {}+ a_{i+1,i+1} \bigl[{a_{j,i}}^{\d_i}, [a_{k,l},a_{j,i-1} ]_q \bigr]_q + a_{i,i} \bigl[{a_{j,i}}^{\d_i}, [a_{k,l},a_{j,i+1} ]_q \bigr]_q\\
 \qquad\qquad {}+\frac{1}{\bigl(q-q^{-1}\bigr)^2} \bigl[a_{i,i+1}, \bigl[ [a_{k,l},a_{j,i} ]_q,{a_{j,i}}^{\d_i} \bigr] \bigr] \\
\qquad \xlongequal[\eqref{rela-2}]{\eqref{eqn:sharp}} a_{k,l}-a_{j,k-1}a_{j,l}-a_{i,i} [a_{j,i-1}, [a_{k,l},a_{j,i} ]_q ]_q+ a_{i,i} \bigl[{a_{j,i}}^{\d_i}, [a_{k,l},a_{j,i+1} ]_q \bigr]_q \\
 \qquad\qquad\quad {}-a_{i+1,i+1} [a_{j,i+1}, [a_{k,l},a_{j,i} ]_q ]_q+ a_{i+1,i+1} \bigl[{a_{j,i}}^{\d_i}, [a_{k,l},a_{j,i-1} ]_q \bigr]_q-a_{k,i}a_{i+1,l}\\
 \qquad\qquad\quad {}+a_{j,i}a_{k,i}a_{j,l}+a_{j,i}a_{j,k-1}a_{i+1,l} +\frac{1}{\bigl(q-q^{-1}\bigr)^2} \bigl[a_{i,i+1}, \bigl[ [a_{k,l},a_{j,i} ]_q, {a_{j,i}}^{\d_i} \bigr] \bigr]\\
\qquad \xlongequal[\eqref{ri-3}]{\substack{{\eqref{ri-1}\eqref{ri-2}}}} a_{k,l} -a_{j,k-1}a_{j,l}-a_{i,i} [a_{j,i-1}, [a_{k,l},a_{j,i} ]_q ]_q+ a_{i,i} \bigl[{a_{j,i}}^{\d_i}, [a_{k,l},a_{j,i+1} ]_q \bigr]_q \\
 \qquad\qquad\qquad {}-a_{i+1,i+1} [a_{j,i+1}, [a_{k,l},a_{j,i} ]_q ]_q +a_{i+1,i+1} \bigl[{a_{j,i}}^{\d_i}, [a_{k,l},a_{j,i-1} ]_q \bigr]_q \\
 \qquad\qquad \qquad{}- \bigl({a_{k,i}}^{\d_i}{a_{i+1,l}}^{\d_i} +a_{k,i-1}a_{i,i}a_{i+1,l}-{a_{k,i}}^{\d_i}a_{i,i}a_{i+2,l} +a_{i+1,i+1} [a_{k,i+1},a_{i+1,l} ]_q \\
 \qquad\qquad\qquad {}-a_{i+1,i+1} \bigl[{a_{k,i}}^{\d_i},a_{i,l} \bigr]_q \bigr) + \bigl({a_{j,i}}^{\d_i}{a_{k,i}}^{\d_i}+a_{i+1,i+1}a_{j,i+1}a_{k,i} +a_{i,i} [a_{j,i-1},a_{k,i} ]_q \\
 \qquad\qquad \qquad {}-{a_{j,i}}^{\d_i}a_{k,i-1}a_{i+1,i+1} -a_{i,i} \bigl[{a_{j,i}}^{\d_i}, a_{k,i+1} \bigr]_q \bigr)a_{j,l} +a_{j,k-1} \bigl({a_{j,i}}^{\d_i}{a_{i+1,l}}^{\d_i} \\
 \qquad\qquad \qquad {}+a_{j,i-1}a_{i,i}a_{i+1,l}+{a_{j,i}}^{\d_i}a_{i,i}a_{i+2,l} +a_{i+1,i+1} [a_{j,i+1},a_{i+1,l} ]_q\\
 \qquad\qquad \qquad{}-a_{i+1,i+1} \bigl[{a_{j,i}}^{\d_i},a_{i,l} \bigr]_q \bigr) +\frac{1}{\bigl(q-q^{-1}\bigr)^2} \bigl( \bigl[a_{i,i+1}, [a_{k,l},a_{j,i} ]_q -a_{j,l}a_{k,i}\\
 \qquad\qquad \qquad{}-a_{j,k-1}a_{i+1,l} + \bigl[a_{i+1,l},{a_{k,i}}^{\d_i} \bigr] \bigr] \bigr)\\
\qquad \xlongequal{\eqref{eq:qjacobi0} } a_{k,l}-a_{j,k-1}a_{j,l}-{a_{k,i}}^{\d_i}{a_{i+1,l}}^{\d_i} +{a_{j,i}}^{\d_i}{a_{k,i}}^{\d_i}a_{j,l}+{a_{j,i}}^{\d_i}a_{j,k-1}{a_{i+1,l}}^{\d_i}\\
 \qquad\qquad\ \ {}+a_{i,i} \bigl( [{a_{j,i}}^{\d_i}, [a_{k,l},a_{j,i+1} ]_q -a_{j,l}a_{k,i+1} -a_{j,k-1}a_{i+2,l} ]_q+{a_{k,i}}^{\d_i}a_{i+2,l} \bigr)\\
 \qquad\qquad\ \ {}-a_{i,i} ( [a_{j,i-1}, [a_{k,l},a_{j,i} ]_q -a_{j,k-1}a_{i+1,l} -a_{j,l}a_{k,i} ]_q+a_{k,i-1}a_{i+1,l} )\\
 \qquad\qquad\ \ {}+a_{i+1,i+1} \bigl( \bigl[{a_{j,i}}^{\d_i}, [a_{k,l},a_{j,i-1} ]_q -a_{k,i-1}a_{j,l}-a_{j,k-1}a_{i,l} \bigr]_q+ \bigl[{a_{k,i}}^{\d_i},a_{i,l} \bigr]_q \bigr)\\
\qquad\qquad \ \ {}-a_{i+1,i+1} \bigl( [a_{j,i+1}, [a_{k,l},a_{j,i} ]_q -a_{j,k-1}a_{i+1,l}-a_{k,i}a_{j,l} ]_q + [a_{k,i+1},a_{i+1,l} ]_q \bigr)\\
 \qquad\xlongequal[\eqref{eqn:i2}]{\eqref{eqn:i1}} a_{k,l}-a_{j,k-1}a_{j,l} -{a_{k,i}}^{\d_i}{a_{i+1,l}}^{\d_i} +{a_{j,i}}^{\d_i}{a_{k,i}}^{\d_i}a_{j,l} +{a_{j,i}}^{\d_i}a_{j,k-1}{a_{i+1,l}}^{\d_i}\\
\qquad \xlongequal{\quad} \d_i \lrb{a_{k,l} } -\d_i \lrb{a_{j,k-1}}\d_i\lrb{a_{j,l}} -\d_i\lrb{a_{k,i}}\d_i\lrb{a_{i+1,l}} +\d_i\lrb{a_{j,i}}\d_i\lrb{a_{k,i}}\d_i\lrb{a_{j,l}}\\
\qquad\qquad{}+\d_i\lrb{a_{j,i}}\d_i\lrb{a_{j,k-1}}\d_i\lrb{a_{i+1,l}},
\end{gather*}
we have
\begin{gather*}
[\d_i\lrb{a_{k, l}},\lr{\d_i\lrb{a_{j, i}}, \d_i\lrb{a_{k, l}}}_q]_q
=f\bigl(\d_i\bigl(\A_{j,k,i+1}^{k-1,i,l}\bigr)\bigr), \\
[\d_i{a_{j,i}},\lr{\d_i\lrb{a_{k,l}}, \d_i\lrb{a_{j,i}}}_q]_q
= g\bigl(\d_i\bigl(\A_{j,k,i+1}^{k-1,i,l}\bigr)\bigr).
\end{gather*}

Similarly, we can get that
\begin{gather*}
[\d_i\lrb{a_{i+1, l}},\lr{\d_i\lrb{a_{j, k}},\d_i\lrb{a_{i+1, l}}}_q]_q
=f\bigl(\d_i\bigl(\A_{j,i+1,k+1}^{i,k,l}\bigr)\bigr), \\
[\d_i\lrb{a_{j,k}},\lr{\d_i\lrb{a_{i+1,l}},\d_i\lrb{a_{j,k}}}_q]_q
= g\bigl(\d_i\bigl(\A_{j,i+1,k+1}^{i,k,l}\bigr)\bigr),\\
[\d_i\lrb{a_{k,i}},\lr{\d_i\lrb{a_{j,l}},\d_i\lrb{a_{k,i}}}_q]_q
=f\bigl(\d_i\bigl(\A_{j,k,l+1}^{k-1,l,i}\bigr)\bigr), \\
[\d_i\lrb{a_{j, l}},\lr{\d_i\lrb{a_{k, i}},\d_i\lrb{a_{j, l}}}_q]_q
= g\bigl(\d_i\bigl(\A_{j,k,l+1}^{k-1,l,i}\bigr)\bigr),\\
[\d_i\lrb{a_{k, l}},\lr{\d_i\lrb{a_{i+1, j}},\d_i\lrb{a_{k, l}}}_q]_q
=f\bigl(\d_i\bigl(\A_{i+1,k,j+1}^{k-1,j,l}\bigr)\bigr), \\
[\d_i\lrb{a_{i+1,j}},\lr{\d_i\lrb{a_{k,l}},\d_i\lrb{a_{i+1,j}}}_q]_q
= g\bigl(\d_i\bigl(\A_{i+1,k,j+1}^{k-1,j,l}\bigr)\bigr).
\end{gather*}

The relations (R3): Choosing the submatrix \smash{$\A_{j,k,i+1}^{i,l,m}$} as $\eqref{eqn-A-2}$, we explore the cases when $k\neq i$ and $l \neq i+1$ and have that
\begin{gather*}
 \lr{\d_i\lrb{a_{j,l}},\d_i\lrb{a_{k,m}}}_q +[\lr{\d_i\lrb{a_{j,i}},\d_i\lrb{a_{k,l}}}_q, \d_i\lrb{a_{i+1,m}}]_q+\d_i\lrb{a_{k,i}}\d_i\lrb{a_{i+1,l}}\d_i\lrb{a_{j,m}} \\
 \quad {}-\d_i\lrb{a_{k,i}}\lr{\d_i\lrb{a_{j,l}}, \d_i\lrb{a_{i+1,m}}}_q -\d_i\lrb{a_{i+1,l}}\lr{\d_i\lrb{a_{j,i}},\d_i\lrb{a_{k,m}}}_q -\d_i\lrb{a_{k,l}}\d_i\lrb{a_{j,m}}\\
\quad\qquad \xlongequal[\eqref{eqn:ri+1}]{\eqref{equ:comm2}} [a_{j,l},a_{k,m} ]_q - \bigl[{a_{j,i}}^{\d_i}, [a_{i,i+1}, [a_{k,l}, a_{i+1, m} ]_q ]_q ]_q+ [ [{a_{j,i}}^{\d_i}, a_{k,l} ]_q,a_{i,i}a_{i+1,m} ]_q\\
 \qquad\qquad\qquad\ \ {}+ \bigl[ \bigl[{a_{j,i}}^{\d_i},a_{k,l} \bigr]_q, a_{i+2,i+1}a_{i,m} \bigr]_q+a_{j,m}{a_{k,i}}^{\d_i}{a_{i+1,l}}^{\d_i} -{a_{k,i}}^{\d_i} [a_{j,l},{a_{i+1,m}}^{\d_i} ]_q \\
\quad\qquad\qquad \ \ {}-{a_{i+1,l}}^{\d_i} \bigl[{a_{j,i}}^{\d_i},a_{k,m} \bigr]_q -a_{k,l}a_{j,m} \\
\quad\qquad \xlongequal{\eqref{eq:qjacobi1}} [a_{j,l},a_{k,m} ]_q-a_{k,l}a_{j,m} - \bigl[ \bigl[{a_{j,i}}^{\d_i},a_{i,i+1} \bigr]_q, [a_{k,l},a_{i+1,m} ]_q \bigr]_q\\
\quad\qquad\qquad \ \ {}-\frac{1}{\bigl(q-q^{-1}\bigr)^2} \bigl[a_{i,i+1}, \bigl[ [a_{k,l},a_{i+1,m} ]_q,{a_{j,i}}^{\d_i} \bigr] \bigr] +a_{i,i} \bigl[{a_{j,i}}^{\d_i}, [a_{k,l},a_{i+2,m} ]_q \bigr]_q\\
\quad\qquad\qquad \ \ {}+a_{i+1,i+1} \bigl[{a_{j,i}}^{\d_i}, [a_{k,l},a_{i,m} ]_q \bigr]_q - \bigl[a_{j,l},{a_{k,i}}^{\d_i}{a_{i+1,m}}^{\d_i} \bigr]_q\\
 \quad\qquad\qquad\ \ {}- \bigl[{a_{i+1,l}}^{\d_i} {a_{j,i}}^{\d_i},a_{k,m} \bigr]_q +a_{j,m}{a_{k,i}}^{\d_i}{a_{i+1,l}}^{\d_i} \\
\quad\qquad \xlongequal[\substack{{\eqref{ri-3}\eqref{ri-4}}}] {\substack{{\eqref{rela-2}\eqref{ri-1}}}} [a_{j,l},a_{k,m} ]_q-a_{k,l}a_{j,m} + [a_{j,i}, [a_{k,l},a_{i+1,m} ]_q ]_q \\
\quad\qquad\qquad \qquad\ \ {}-a_{i,i} [a_{j,i-1}, [a_{k,l},a_{i+1,m} ]_q ]_q -a_{i+1,i+1} [a_{j,i+1}, [a_{k,l},a_{i+1,m} ]_q ]_q\\
\qquad\qquad\qquad \qquad\ \ {}-\frac{1}{\bigl(q-q^{-1}\bigr)^2} \bigl[a_{i,i+1}, \bigl[ [a_{k,l},a_{i+1,m} ]_q, {a_{j,i}}^{\d_i} \bigr] \bigr]+a_{i,i} \bigl[{a_{j,i}}^{\d_i}, [a_{k,l},a_{i+2,m} ]_q \bigr]_q \!\\
 \quad\qquad\qquad\qquad\ \ {}+ a_{i+1,i+1} \bigl[{a_{j,i}}^{\d_i}, [a_{k,l},a_{i,m} ]_q \bigr]_q- \bigl[a_{j,l},a_{k,i}a_{i+1,m}-a_{k,i-1}a_{i,i}a_{i+1,m} \\
 \quad\qquad\qquad\qquad\ \ {}+ {a_{k,i}}^{\d_i}a_{i,i}a_{i+2,m} -a_{i+1,i+1} [a_{k,i+1},a_{i+1,m} ]_q + a_{i+1,i+1} \bigl[{a_{k,i}}^{\d_i},a_{i,m} \bigr]_q \\
\quad\qquad\qquad \qquad\ \ {}-\frac{1}{\bigl(q-q^{-1}\bigr)^2} \bigl[a_{i,i+1}, \bigl[a_{i+1,m}, {a_{k,i}}^{\d_i} \bigr] \bigr] \bigr]_q - \bigl[a_{j,i}a_{i+1,l} -a_{j,i-1}a_{i,i}a_{i+1,l} \\
 \quad\qquad\qquad\qquad\ \ {}+{a_{j,i}}^{\d_i}a_{i,i}a_{i+2,l} -a_{i+1,i+1} [a_{j,i+1},a_{i+1,l} ]_q +a_{i+1,i+1} \bigl[{a_{j,i}}^{\d_i},a_{i,l} \bigr]_q \\
\quad\qquad\qquad \qquad\ \ {}-\frac{1}{\bigl(q-q^{-1}\bigr)^2} \bigl[a_{i,i+1}, \bigl[a_{i+1,l},{a_{j,i}}^{\d_i} \bigr] \bigr],a_{k,m} \bigr]_q +a_{j,m} (a_{k,i}a_{i+1,l} \\
 \quad\qquad\qquad\qquad\ \ {}-a_{k,i-1}a_{i,i}a_{i+1,l}+{a_{k,i}}^{\d_i}a_{i,i}a_{i+2,l} -a_{i+1,i+1} [a_{k,i+1},a_{i+1,l} ]_q\\
\quad\qquad\qquad\qquad\ \ {}+a_{i+1,i+1} \bigl[{a_{k,i}}^{\d_i},a_{i,l} \bigr]_q +\bigl[a_{i,i+1}, \bigl[a_{i+1,l},{a_{k,i}}^{\d_i} \bigr] \bigr] )\\
\quad\qquad \xlongequal[\eqref{eqn:2b}\eqref{eq:qjacobi0}]{\eqref{rela-3}} {\det}_q\bigl(\A_{j,k,i+1}^{i,l,m}\bigr)-a_{i,i} \bigl( [a_{j,i-1}, [a_{k,l}, a_{i+1,m} ]_q -a_{i+1,l}a_{k,m} ]_q \\
 \quad\qquad\qquad\qquad\ \ {}-a_{k,i-1} ( [a_{j,l},a_{i+1,m} ]_q- a_{i+1,l}a_{j,m} )- \bigl[{a_{j,i}}^{\d_i} [a_{k,l},a_{i+2,m} ]_q +a_{i+2,l}a_{k,m} \bigr]_q\!\\
 \quad\qquad\qquad\qquad\ \ {}+{a_{k,i}}^{\d_i} ( [a_{j,l},a_{i+2,m} ]_q + a_{i+2,l}a_{j,m} ) \bigr)-a_{i+1,i+1} \bigl( [a_{j,i+1}, [a_{k,l},a_{i+1,m} ]_q\\
 \quad\qquad\qquad\qquad\ \ {}- [a_{i+1,l}, a_{k,m} ]_q ]_q- \bigl[{a_{j,i}}^{\d_i} [a_{k,l},a_{i,m} ]_q - [a_{i,l},a_{k,m} ]_q \bigr]_q
\\
 \quad\qquad\qquad\qquad\ \ {}- [a_{k,i+1}, [a_{j,l},a_{i+1,m} ]_q-a_{j,m}a_{i+1,l} ]_q + \bigl[{a_{k,i}}^{\d_i}a_{j,m}a_{i,l} - [a_{j,l},a_{i,m} ]_q \bigr]_q \bigr)\\
 \quad\qquad\qquad\qquad\ \ {}+a_{i,i}a_{l+1,m} \bigl(- \bigl[{a_{j,i}}^{\d_i}a_{k,i+1} \bigr]_q+{a_{k,i}}^{\d_i}a_{j,i+1} + [a_{j,i-1},a_{k,i} ]_q-a_{j,i}a_{k,i-1} \bigr)\\
 \quad\qquad\qquad\qquad\ \ {} -a_{i+1,i+1}a_{l+1,m} \bigl(a_{k,i}a_{j,i+1}-a_{k,i-1}{a_{j,i}}^{\d_i}- [a_{k,i+1},a_{j,i} ]_q+ \bigl[{a_{k,i}}^{\d_i} a_{j,i-1} \bigr]_q \bigr)\\
 \quad\qquad\xlongequal{\eqref{rela-3}} a_{i,i} \bigl( \bigl[{a_{j,i}}^{\d_i} [a_{k,l},a_{i+2,m} ]_q -a_{k,i+1}a_{l+1,m}-a_{i+2,l}a_{k,m} \bigr]_q \\[-0.5mm]
\quad\qquad\qquad\quad {} -{a_{k,i}}^{\d_i} ( [a_{j,l},a_{i+2,m} ]_q
-a_{j,i+1}a_{l+1,m} {}-a_{i+2,l}a_{j,m} ) \bigr)\\
 \quad\qquad\qquad\quad {}-a_{i,i} ( [a_{j,i-1}, [a_{k,l},a_{i+1,m} ]_q
-a_{k,i}a_{l+1,m} -a_{i+1,l}a_{k,m} ]_q \\
\quad\qquad\qquad\quad {}-a_{k,i-1} ( [a_{j,l},a_{i+1,m} ]_q-a_{j,i}a_{l+1,m} -a_{i+1,l}a_{j,m} ) )\\
 \quad\qquad\qquad\quad {}-a_{i+1,i+1} ( [a_{j,i+1}, [a_{k,l},a_{i+1,m} ]_q-a_{k,i} a_{l+1,m}-a_{i+1,l}a_{k,m} ]_q \\
\quad\qquad\qquad \quad {}- [a_{k,i+1}, [a_{j,l},a_{i+1,m} ]_q-a_{j,i}a_{l+1,m} -a_{i+1,l}a_{j,m} ]_q )\\
 \quad\qquad\qquad \quad {}+a_{i+1,i+1} \bigl( \bigl[{a_{k,i}}^{\d_i} [a_{j,l},a_{i,m} ]_q -a_{j,i-1}a_{l+1,m}-a_{i,l}a_{j,m} \bigr]_q \\
\quad\qquad\qquad\quad {}- \bigl[{a_{j,i}}^{\d_i} [a_{k,l},a_{i,m} ]_q-a_{k,i-1}a_{l+1,m} -a_{i,l}a_{k,m} \bigr]_q \bigr)\\[-0.5mm]
\quad\qquad \xlongequal[\eqref{rela-3}]{\eqref{eqn:2b}} - a_{i,i} (a_{k,i-1}(\lr{a_{j,l}, a_{i+1,m}}_q-a_{k,i-1}a_{i+1,l}a_{j,m})+[a_{j,i-1}, \lr{a_{k,l},a_{i+1,m}}_q\\[-0.5mm]
\quad\qquad\qquad \quad {}-a_{i+1,l}a_{k,m}]_q -{a_{k,i}}^{\d_i}(\lr{a_{j,l},a_{i+2,m}}_q-a_{i+2,l}a_{j,m})-\bigl[{a_{j,i}}^{\d_i}, a_{i+2,l}a_{k,m}\\
 \quad\qquad\qquad\quad {}-\lr{a_{k,l},a_{i+2,m}}_q\bigr]_q \bigr)+a_{i+1,i+1} \bigl([a_{j,i+1}, a_{i+1,l}a_{k,m} -\lr{a_{k,l},a_{i+1,m}}_q]_q \\
\quad\qquad\qquad\quad {}+[a_{k,i+1},\lr{a_{j,l},a_{i+1,m}}_q -a_{i+1,l}a_{j,m}]_q -\bigl[{a_{j,i}}^{\d_i},a_{i,l}a_{k,m}-\lr{a_{k,l},a_{i,m}}_q\bigr]_q \\
\quad\qquad\qquad\quad {}-\bigl[{a_{k,i}}^{\d_i},\lr{a_{j,l},a_{i,m}}_q-a_{i,l}a_{j,m}\bigr]_q \bigr)\\[-0.5mm]
\quad\qquad \xlongequal[\eqref{eqn:j2}]{\eqref{eqn:j1}} 0
\end{gather*}
and
\begin{gather*}
 \lr{\d_i\lrb{a_{k,m}},\d_i\lrb{a_{j,l}}}_q +[\lr{\d_i\lrb{a_{i+1,m}},\d_i\lrb{a_{k,l}}}_q,\d_i\lrb{a_{j,i}}]_q -\d_i\lrb{a_{i+1,l}}\lr{\d_i\lrb{a_{k,m}},\d_i\lrb{a_{j,i}}}_q\\
 \quad {}-\d_i\lrb{a_{k,i}}\lr{\d_i\lrb{a_{i+1,m}},\d_i\lrb{a_{j,l}}}_q +\d_i\lrb{a_{i+1,l}}\d_i\lrb{a_{k,i}}\d_i\lrb{a_{j,m}}-\d_i\lrb{a_{k,l}}\d_i\lrb{a_{j,m}} \\[-0.5mm]
\quad\qquad\xlongequal[\eqref{eqn:ri}]{\eqref{equ:comm2}} [a_{k,m},a_{j,l} ]_q
- \bigl[ \bigl[ \bigl[{a_{i+1,m}}^{\d_i} ,a_{i,i+1} \bigr]_q,a_{k,l} \bigr]_q ,a_{j,i} \bigr]_q
\\[-0.5mm]
 \quad\qquad\qquad\quad {}+ \bigl[ \bigl[{a_{i+1,m}}^{\d_i} ,a_{k,l} \bigr]_q, a_{i+1,i+1}a_{j,i-1} \bigr]_q+ \bigl[ \bigl[{a_{i+1,m}}^{\d_i} ,a_{k,l} \bigr]_q, a_{i,i}a_{j,i+1} \bigr]_q
\\
 \quad\qquad\qquad\quad {}-{a_{i+1,l}}^{\d_i} \bigl[a_{k,m},{a_{j,i}}^{\d_i} \bigr]_q -{a_{k,i}}^{\d_i} \bigl[{a_{i+1,m}}^{\d_i} ,a_{j,l} \bigr]_q +{a_{i+1,l}}^{\d_i} {a_{k,i}}^{\d_i} a_{j,m}-a_{k,l}a_{j,m}\\[-0.5mm]
\quad\qquad \xlongequal{\eqref{eq:qjacobi1}} [a_{k,m},a_{j,l} ]_q-a_{k,l}a_{j,m} - \bigl[ \bigl[{a_{i+1,m}}^{\d_i} ,a_{i,i+1} \bigr]_q, [a_{k,l},a_{j,i} ]_q \bigr]_q\\[-0.5mm]
\quad\qquad\qquad \quad {}-\frac{1}{\bigl(q-q^{-1}\bigr)^2} \bigl[a_{i,i+1}, \bigl[ [a_{k,l},a_{j,i} ]_q, {a_{i+1,m}}^{\d_i} \bigr] \bigr] \\
 \quad\qquad\qquad\quad {}+a_{i+1,i+1} \bigl[{a_{i+1,m}}^{\d_i}, [a_{k,l},a_{j,i-1} ]_q \bigr]_q +a_{i,i} \bigl[{a_{i+1,m}}^{\d_i} , [a_{k,l},a_{j,i+1} ]_q \bigr]_q \\
\quad\qquad\qquad \quad {}- \bigl[a_{k,m},{a_{i+1,l}}^{\d_i} {a_{j,i}}^{\d_i} \bigr]_q - \bigl[{a_{k,i}}^{\d_i} {a_{i+1,m}}^{\d_i} ,a_{j,l} \bigr]_q +{a_{i+1,l}}^{\d_i} {a_{k,i}}^{\d_i} a_{j,m}\\
\quad\qquad \xlongequal[\eqref{ri-3}\eqref{ri-4}]{\eqref{rela-2}\eqref{ri-1}} [a_{k,m},a_{j,l} ]_q -a_{k,l}a_{j,m} + [a_{i+1,m}, [a_{k,l},a_{j,i} ]_q ]_q \\
\quad\qquad\qquad \qquad\quad {}-a_{i+1,i+1} [a_{i+2,m}, [a_{k,l},a_{j,i} ]_q ]_q -a_{i,i} [a_{i,m}, [a_{k,l},a_{j,i} ]_q ]_q\\
 \quad\qquad\qquad\qquad\quad {}-\frac{1}{\bigl(q-q^{-1}\bigr)^2} \bigl[a_{i,i+1}, \bigl[ [a_{k,l},a_{j,i} ]_q, {a_{i+1,m}}^{\d_i} \bigr] \bigr]\\
 \quad\qquad\qquad\qquad\quad {}+a_{i+1,i+1} \bigl[{a_{i+1,m}}^{\d_i} , [a_{k,l},a_{j,i-1} ]_q \bigr]_q+a_{i,i} \bigl[{a_{i+1,m}}^{\d_i} , [a_{k,l},a_{j,i+1} ]_q \bigr]_q\\
 \quad\qquad\qquad\qquad\quad {}- \bigl[a_{k,m},a_{i+1,l}a_{j,i} -a_{i+2,l}a_{i+1,i+1}a_{j,i}+{a_{i+1,l}}^{\d_i} a_{i+1,i+1}a_{j,i-1}\\
 \quad\qquad\qquad\qquad\quad {}-a_{i,i} [a_{i,l},a_{j,i} ]_q +a_{i,i} \bigl[{a_{i+1,l}}^{\d_i} ,a_{j,i+1} \bigr]_q\\
 \quad\qquad\qquad\qquad\quad {}-\frac{1}{\bigl(q-q^{-1}\bigr)^2} \bigl[a_{i,i+1}, \bigl[a_{j,i}, {a_{i+1,l}}^{\d_i} \bigr] \bigr] \bigr]_q - \bigl[a_{i+1,m}a_{k,i} \\
 \quad\qquad\qquad\qquad\quad {}-a_{i+2,m}a_{i+1,i+1}a_{k,i} +{a_{i+1,m}}^{\d_i} a_{i+1,i+1}a_{k,i-1} -a_{i,i} [a_{i,m},a_{k,i} ]_q \\
\quad\qquad\qquad \qquad \quad {}+a_{i,i} \bigl[{a_{i+1,m}}^{\d_i} ,a_{k,i+1} \bigr]_q-\frac{1}{\bigl(q-q^{-1}\bigr)^2} \bigl[a_{i,i+1}, \bigl[a_{k,i}, {a_{i+1,m}}^{\d_i} \bigr] \bigr],a_{j,l} \bigr]_q \\
\quad\qquad\qquad \qquad \quad{}+ \bigl(a_{i+1,l}a_{k,i} -a_{i+2,l}a_{i+1,i+1}a_{k,i}+{a_{i+1,l}}^{\d_i} a_{i+1,i+1}a_{k,i-1} -a_{i,i} [a_{i,l},a_{k,i} ]_q \\
\quad\qquad\qquad\qquad \quad {}+a_{i,i} \bigl[{a_{i+1,l}}^{\d_i} ,a_{k,i+1} \bigr]_q + \bigl[a_{i,i+1}, \bigl[a_{k,i},{a_{i+1,l}}^{\d_i} \bigr] \bigr] \bigr)a_{j,m}\\
\quad\qquad \xlongequal[\eqref{eqn:2a}\eqref{eq:qjacobi0}]{\eqref{rela-3}} {\det}^q\bigl(\A_{j,k,i+1}^{i,l,m}\bigr)-a_{i+1,i+1} \bigl( [a_{i+2,m}, [a_{k,l}, a_{j,i} ]_q-a_{k,i}a_{j,l} ]_q\\
 \quad\qquad\qquad\qquad\ {}-a_{i+2,l} ( [a_{k,m},a_{j,i} ]_q -a_{k,i}a_{j,m} ) - \bigl[{a_{i+1,m}}^{\d_i}, [a_{k,l}, a_{j,i-1} ]_q -a_{k,i-1}a_{j,l} \bigr]_q \\
 \quad\qquad\qquad\qquad\ {}+{a_{i+1,l}}^{\d_i} ( [a_{k,m},a_{j,i-1} ]_q -a_{k,i-1}a_{j,m} ) \bigr) -a_{i,i} \bigl( [a_{i,m}, [a_{k,l},a_{j,i} ]_q\\
\quad\qquad\qquad\qquad\ {}- [a_{k,i}, a_{j,l} ]_q ]_q - \bigl[{a_{i+1,m}}^{\d_i}, [a_{k,l},a_{j,i+1} ]_q - [a_{k,i+1},a_{j,l} ]_q \bigr]_q\\
 \quad\qquad\qquad\qquad\ {}- [a_{i,l}, [a_{k,m},a_{j,i} ]_q -a_{j,m}a_{k,i} ]_q+ \bigl[{a_{i+1,l}}^{\d_i},a_{j,m}a_{k,i+1} - [a_{k,m},a_{j,i+1} ]_q \bigr]_q \bigr)\\
\quad\qquad\qquad\qquad\ {}-a_{i+1,i+1}a_{j,k-1} \bigl( \bigl[{a_{i+1,m}}^{\d_i},a_{i,l} \bigr]_q -{a_{i+1,l}}^{\d_i}a_{i,m} - [a_{i+2,m},a_{i+1,l} ]_q \bigr)\\
\quad\qquad\qquad\qquad\ {}+a_{i+1,m}a_{i+2,l} +a_{i,i}a_{j,k-1} \bigl(-a_{i+1,l}a_{i,m} +a_{i+2,l}{a_{i+1,m}}^{\d_i} + [a_{i,l},a_{i+1,m} ]_q \\
\quad\qquad\qquad\qquad\ {}- \bigl[{a_{i+1,l}}^{\d_i},a_{i+2,m} \bigr]_q \bigr)\\
\quad\qquad \xlongequal{\eqref{rela-3}} a_{i+1,i+1} \bigl( \bigl[{a_{i+1,m}}^{\d_i}, [a_{k,l},a_{j,i-1} ]_q -a_{i,l}a_{j,k-1}-a_{k,i-1}a_{j,l} \bigr]_q \\
\quad\qquad\qquad \quad {}-{a_{i+1,l}}^{\d_i} \bigl( [a_{k,m},a_{j,i-1} ]_q-a_{i,m}a_{j,k-1} -a_{k,i-1}a_{j,m} \bigr) \bigr)\\
 \quad\qquad\qquad \quad {}-a_{i+1,i+1} ( [a_{i+2,m}, [a_{k,l},a_{j,i} ]_q-a_{i+1,l}a_{j,k-1} -a_{k,i}a_{j,l} ]_q \\
 \quad\qquad\qquad\quad {}-a_{i+2,l} ( [a_{k,m},a_{j,i} ]_q-a_{i+1,m}a_{j,k-1} -a_{k,i}a_{j,m} ) )\\
 \quad\qquad\qquad\quad {}-a_{i,i} ( [a_{i,m}, [a_{k,l},a_{j,i} ]_q -a_{i+1,l}a_{j,k-1}-a_{k,i}a_{j,l} ]_q \\
\quad\qquad\qquad \quad {}- [a_{i,l}, [a_{k,m},a_{j,i} ]_q-a_{i+1,m}a_{j,k-1} -a_{k,i}a_{j,m} ]_q )\\
\quad\qquad\qquad\quad {}+a_{i,i} \bigl( \bigl[{a_{i+1,l}}^{\d_i}, [a_{k,m},a_{j,i+1} ]_q -a_{i+2,m}a_{j,k-1}-a_{k,i+1}a_{j,m} \bigr]_q \\
\quad\qquad\qquad \quad {}- \bigl[{a_{i+1,m}}^{\d_i}, [a_{k,l},a_{j,i+1} ]_q -a_{i+2,l}a_{j,k-1}-a_{k,i+1}a_{j,l} \bigr]_q \bigr)\\
\quad\qquad \xlongequal[\eqref{rela-3}]{\eqref{eqn:2a}} a_{i+1,i+1} \bigl({a_{i+1,l}}^{\d_i} (\lr{a_{k,m},a_{j,i-1}}_q-a_{k,i-1}a_{j,m}) \\
 \quad\qquad\qquad \quad {}+\bigl[{a_{i+1,m}}^{\d_i}, a_{k,i-1}a_{j,l}-\lr{a_{k,l},a_{j,i-1}}_q\bigr]_q \\
 \quad\qquad\qquad \quad {}-a_{i+2,l}(\lr{a_{k,m}, a_{j,i}}_q -a_{i+2,l}a_{k,i}a_{j,m})-[a_{i+2,m}, \lr{a_{k,l},a_{j,i}}_q-a_{k,i}a_{j,l}]_q \bigr)\\
 \quad\qquad\qquad \quad {}+a_{i,i} \bigl(\bigl[{a_{i+1,m}}^{\d_i},a_{k,i+1}a_{j,l}-\lr{a_{k,l},a_{j,i+1}}_q\bigr]_q \\
 \quad\qquad\qquad \quad {}+\bigl[{a_{i+1,l}}^{\d_i},\lr{a_{k,m},a_{j,i+1}}_q-a_{k,i+1}a_{j,m}\bigr]_q \\
 \quad\qquad\qquad \quad {}-\bigl[a_{i,m},a_{k,i}a_{j,l} -\lr{a_{k,l},a_{j,i}}_q\bigr]_q -\bigl[a_{i,l},\lr{a_{k,m},a_{j,i}}_q-a_{k,i}a_{j,m}\bigr]_q \bigr)\\
\quad\qquad \xlongequal[\eqref{eqn:j4}]{\eqref{eqn:j3}} 0,
\end{gather*}
Hence,
\begin{gather*}
{\det}_q \bigl(\d_i \bigl(\A_{j,k,i+1}^{i,l,m} \bigr) \bigr) ={\det}^q \bigl(\d_i \bigl(\A_{j,k,i+1}^{i,l,m} \bigr) \bigr)=0.
\end{gather*}
Similarly, we can get that
\begin{gather*}
{\det}_q \bigl(\d_i \bigl(\A_{j,k,l+1}^{l,i,m} \bigr) \bigr) ={\det}^q \bigl(\d_i \bigl(\A_{j,k,l+1}^{l,i,m} \bigr) \bigr)=0;\\
{\det}_q \bigl(\d_i \bigl(\A_{j,k,l+1}^{l, m, i} \bigr) \bigr) ={\det}^q \bigl(\d_i \bigl(\A_{j,k,l+1}^{l, m, i} \bigr) \bigr)=0;\\
{\det}_q \bigl(\d_i \bigl(\A_{i+1,j,k+1}^{k,l,m} \bigr) \bigr) ={\det}^q \bigl(\d_i \bigl(\A_{i+1,j,k+1}^{k,l,m} \bigr) \bigr)=0;\\
{\det}_q \bigl(\d_i \bigl(\A_{j,i+1,k+1}^{k,l,m} \bigr) \bigr) ={\det}^q \bigl(\d_i \bigl(\A_{j,i+1,k+1}^{k,l,m} \bigr) \bigr)=0.
\end{gather*}
Hence, the maps $\d_i$ is an algebra homomorphism of $\A(n)$.
Similarly, $\d_i'$ is an algebra homomorphism.

In fact, $\d_i$ (resp.\ $\d_i'$) are automorphisms of $\A(n)$.
Indeed,
\begin{gather*}
\d_i\d_i'\lrb{a_{i,i}} = \d_i\lrb{a_{i+1,i+1}}=a_{i,i},\qquad
\d_i\d_i'\lrb{a_{i+1,i+1}} = \d_i\lrb{a_{i,i}}=a_{i+1,i+1},\\
\d_i\d_i'\lrb{a_{j,i}} = \d_i\bigl({a_{j,i}}^{\d_i'}\bigr)
 = -\bigl[{a_{j,i}}^{\d_i}, a_{i,i+1}\bigr]_q+a_{j,i-1}a_{i,i}+a_{i+1,i+1}a_{j,i+1}\\
\phantom{\d_i\d_i'\lrb{a_{j,i}} = }{} = [\lr{a_{i,i+1}, a_{j,i}}_q, a_{i,i+1}]_q-f\bigl(\A_{j,i,i+1}^{i-1,i,i+1}\bigr)+a_{j,i}
 = a_{j,i},\\
\d_i\d_i'\lrb{a_{i+1,l}} = \d_i\bigl({a_{i+1,l}}^{\d_i'}\bigr)
 = -\bigl[{a_{i+1,l}}^{\d_i} , a_{i,i+1}\bigr]_q+a_{i+1,i+1}a_{i+2,l}+a_{i,i}a_{i,l}\\
\phantom{\d_i\d_i'\lrb{a_{i+1,l}} =}{} = [\lr{a_{i,i+1}, a_{i+1,l}}_q, a_{i,i+1}]_q -g\bigl(\A_{i,i+1,i+2}^{i,i+1,l}\bigr)+a_{i+1,l}
 = a_{i+1,l},\\
\d_i\d_i'\lrb{a_{j,l}} = \d_i\lrb{a_{j,l}}=a_{j,l},
\end{gather*}
where $j \neq i+1$ and $l \neq i$.
Hence $\d_i\d_i'=\id$ and similarly, we have $\d_i'\d_i=\id$.

This proof is finished.
\end{proof}

\begin{remark}
The automorphisms $\d_i$ of $\A(n)$ coincide with those given in
{\rm\cite{CFPR(2023)}}. Indeed, one can easily see that
\begin{equation*}
\xymatrix{
 \A(n) \ar[d]^{\phi} \ar[r]^{\d_i}
 & \A(n) \ar[d]^{\phi} \\
 \aw(n) \ar[r]_{r_i}
 & \aw(n).
}
\end{equation*}
\end{remark}

\section{The braid relations}\label{braid}
In this section, we prove that the automorphisms established in Section
\ref{sect:automorphism} satisfy the braid group relations.

A group $B_n$ is called a braid group if $B_n$ is generated by
$\sigma_1, \dots, \sigma_{n}$ with the following relations:
\begin{gather}
\sigma_i\sigma_{i+1}\sigma_i=\sigma_{i+1 }\sigma_i\sigma_{i+1}, \qquad
i \in \lbb{1, n-1}, \label{eqn4-1}\\
\sigma_i\sigma_j=\sigma_j\sigma_i,\qquad |i-j| \geq 2, \qquad i, j \in \lbb{1,n}.
 \label{eqn4-2}
\end{gather}
Accordingly, we have
\begin{Theorem} The algebra automorphisms $\d_i$ where $i \in \lbb{0,n-1}$
of $\A(n)$, satisfy the braid relations
\begin{gather*}
\d_i\d_{i+1}\d_i=\d_{i+1 }\d_i\d_{i+1},\qquad i \in \lbb{0,n-2};\\
\d_i\d_j=\d_j\d_i,\qquad |i-j| \geq 2, \qquad i,j \in \lbb{0,n-1}.
\end{gather*}
\end{Theorem}
\begin{proof}
The proof includes more tedious calculations, but to understand them
in a better way, we write them down in detail.

Firstly, we prove that $\d_i\d_{i+1}\d_i=\d_{i+1 }\d_i\d_{i+1}$, $i \in \lbb{0,n-2}$.

(a) When $i=0$.
If $k \in \lbb{3,n}$ and $j\in\lbb{2,n}$, obviously we have
$\d_0\d_1\d_0\lrb{a_{k,j}}=a_{k,j}=\d_1\d_0\d_1\lrb{a_{k,j}}$.
If $k=1$ or $k=2$ or $j=3$, then we have
\begin{gather*}
 \d_0\d_1\d_0\lrb{a_{1,1}} =a_{1,1}=\d_1\d_0\d_1\lrb{a_{1,1}},\qquad
 \d_0\d_1\d_0\lrb{a_{1,n}}=a_{2,2}=\d_1\d_0\d_1\lrb{a_{1,n}},\\
 \d_0\d_1\d_0\lrb{a_{2,2}}=a_{1,n}=\d_1\d_0\d_1\lrb{a_{2,2}},\\
 \d_0\d_1\d_0\lrb{a_{1,j}}
 \xlongequal[\eqref{eqn:ri}]{\eqref{eqn:r0}} [a_{1,2},\lr{a_{2,n}, a_{1,j}}_q]_q+a_{2,2}a_{j+1,n}-a_{1,n}\lr{a_{1,2}, a_{2,j}}_q +a_{1,1}a_{3,j}a_{1,n} \\
\phantom{ \d_0\d_1\d_0\lrb{a_{1,j}}\xlongequal{\qquad}}{} -a_{1,1}\bigl(a_{j+1,n}a_{1,2}+\bigl[{a_{1,2}}^{\d_0}, a_{2,j}\bigr]_q -a_{2,2}{a_{1,j}}^{\d_0}\bigr) \\
 \phantom{ \d_0\d_1\d_0\lrb{a_{1,j}}}{} \xlongequal{\lrb{\ref{eqn:3b}}} [a_{1,2},\lr{a_{2,n}, a_{1,j}}_q]_q+a_{2,2}a_{j+1,n}-a_{1,n}\lr{a_{1,2}, a_{2,j}}_q \\
 \phantom{ \d_0\d_1\d_0\lrb{a_{1,j}}\xlongequal{\qquad}}{} +a_{1,1}a_{3,j}a_{1,n}-a_{1,1}\lr{a_{3,n}, a_{1,j}}_q\\
 \phantom{ \d_0\d_1\d_0\lrb{a_{1,j}}}{} \xlongequal{\eqref{equ:comm2}} [\lr{a_{1,2}, a_{2,n}}_q, a_{1,j}]_q-a_{1,1}\lr{a_{3,n}, a_{1,j}}_q+a_{2,2}a_{j+1,n}\\
 \phantom{ \d_0\d_1\d_0\lrb{a_{1,j}}\xlongequal{\qquad}}{} -a_{1,n}(\lr{a_{1,2}, a_{2,j}}_q-a_{1,1}a_{3,j}) \\
 \phantom{ \d_0\d_1\d_0\lrb{a_{1,j}}}{} \xlongequal{\eqref{eqn:ri}} [\lr{a_{1,2}, a_{2,n}}_q -a_{1,1} a_{3,n}, a_{1,j}]_q +a_{2,2}a_{j+1,n}+a_{1,n}\bigl({a_{2,j}}^{\d_1}-a_{2,2} a_{1,j}\bigr) \\
 \phantom{ \d_0\d_1\d_0\lrb{a_{1,j}}}{} \xlongequal{\eqref{eqn:ri}} -\bigl[{a_{2,n}}^{\d_1}, a_{1,j}\bigr]_q+a_{2,2}a_{j+1,n}+{a_{2,j}}^{\d_1}a_{1,n}\\
 \phantom{ \d_0\d_1\d_0\lrb{a_{1,j}}}{} \xlongequal[\eqref{eqn:ri}]{\lrb{\ref{eqn:r0}}} \d_1\d_0\d_1\lrb{a_{1,j}},\\
 \d_0\d_1\d_0\lrb{a_{2,j}} \xlongequal[\eqref{eqn:ri}]{\lrb{\ref{eqn:r0}}} [\lr{a_{2,n}, a_{1,2}}_q, a_{2,j}]_q -a_{1,1}\lr{a_{3,n}, a_{2,j}}_q +a_{3,j} a_{1,n} \\
 \phantom{ \d_0\d_1\d_0\lrb{a_{2,j}} \xlongequal{\qquad}}{} -a_{2,2}\lr{a_{2,n}, a_{1,j}}_q+a_{1,1}a_{2,2}a_{j+1,n}\\
 \phantom{ \d_0\d_1\d_0\lrb{a_{2,j}}}{} \xlongequal[\eqref{eqn:ri}]{\eqref{equ:comm2}} [a_{2,n},\lr{a_{1,2}, a_{2,j}}_q]_q -a_{1,1}\bigl(a_{3,j}a_{2,n} - {a_{2,j}}^{\d_1}a_{1,n}+\bigl[ {a_{2,n}}^{\d_1}, a_{1,j}\bigr]_q\bigr) \\
 \phantom{ \d_0\d_1\d_0\lrb{a_{2,j}}\xlongequal{\qquad}}{} +a_{3,j}a_{1,n}-a_{2,2}\lr{a_{2,n}, a_{1,j}}_q +a_{1,1}a_{2,2}a_{j+1,n}\\
 \phantom{ \d_0\d_1\d_0\lrb{a_{2,j}}}{} \xlongequal{\quad} [a_{2,n},\lr{a_{1,2}, a_{2,j}}_q -a_{1,1}a_{3,j} -a_{2,2}a_{1,j}]_q+a_{1,n}a_{3,j} \\
 \phantom{ \d_0\d_1\d_0\lrb{a_{2,j}} \xlongequal{\quad} }{} +a_{1,1} \bigl({a_{2,j}}^{\d_1}a_{1,n} - \bigl[ {a_{2,n}}^{\d_1}, a_{1,j}\bigr]_q\bigr)+a_{1,1}a_{2,2}a_{j+1,n}\\
 \phantom{ \d_0\d_1\d_0\lrb{a_{2,j}}}{} \xlongequal[\eqref{eqn:ri}]{\eqref{eqn:1b}} -\bigl[a_{2,n} , {a_{2,j}}^{\d_1}\bigr]_q+a_{1,n}a_{3,j} +a_{1,1} \bigl({a_{2,j}}^{\d_1}a_{1,n}-\bigl[ {a_{2,n}}^{\d_1}, a_{1,j}\bigr]_q\bigr)\\
 \phantom{ \d_0\d_1\d_0\lrb{a_{2,j}} \xlongequal{\qquad} }{} +a_{1,1}a_{2,2}a_{j+1,n}\\
 \phantom{ \d_0\d_1\d_0\lrb{a_{2,j}}}{} \xlongequal[\eqref{eqn:ri}]{\lrb{\ref{eqn:r0}}} \d_1\d_0\d_1\lrb{a_{2,j}}.
\end{gather*}

(b) When $i \in \lbb{1,n-2}$.
If $k \neq i+1, i+2$ and $j\neq i, i+1$, obviously we have
$\d_i\d_{i+1}\d_i\lrb{a_{k,j}}=a_{k,j}=\d_{i+1}\d_i\d_{i+1}\lrb{a_{k,j}}$.
If $k=i+1$ or $k=i+2$ or $j=i$ or $j=i+1$, we have
\begin{gather*}
 \d_i\d_{i+1}\d_i\lrb{a_{i,i}}
=a_{i+2,i+2}
=\d_{i+1}\d_i\d_{i+1}\lrb{a_{i,i}},\\
 \d_i\d_{i+1}\d_i\lrb{a_{i+1,i+1}}
=a_{i+1,i+1}
=\d_{i+1}\d_i\d_{i+1}\lrb{a_{i+1,i+1}},\\
 \d_i\d_{i+1}\d_i\lrb{a_{i+2,i+2}}
=a_{i,i}
=\d_{i+1}\d_i\d_{i+1}\lrb{a_{i+2,i+2}},\\
 \d_i\d_{i+1}\d_i\lrb{a_{k,i}}
 \xlongequal[\eqref{eqn:ri+1}]{\lrb{\ref{eqn:ri}}} -\bigl[a_{i+1,i+2}, {a_{k,i}}^{\d_{i}}\bigr]_q+a_{i+1,i+1}a_{i+2,i+2} {a_{k,i}}^{\d_{i}}+a_{k,i-1}a_{i+2,i+2}\\
\phantom{ \d_i\d_{i+1}\d_i\lrb{a_{k,i}}\xlongequal{\qquad}\!\! }{} -a_{i+1,i+1}\bigl[ {a_{i+1,i+2}}^{\d_{i}}, a_{1,i+1}\bigr]_q+a_{i,i}a_{i+1,i+1}a_{1,i+2}\\
 \phantom{ \d_i\d_{i+1}\d_i\lrb{a_{k,i}} }{} \xlongequal{\lrb{\ref{eqn:ri}}} [a_{i+1,i+2},\lr{a_{i,i+1}, a_{k,i}}_q]_q -a_{i+1,i+2}a_{k,i-1}a_{i+1,i+1}\\
 \phantom{ \d_i\d_{i+1}\d_i\lrb{a_{k,i}}\xlongequal{\qquad}\!\! }{} -a_{i,i}\lr{a_{i+1,i+2}, a_{k,i+1}}_q +a_{i+1,i+1}a_{i+2,i+2} {a_{k,i}}^{\d_{i}} +a_{k,i-1}a_{i+2,i+2}\\
 \phantom{ \d_i\d_{i+1}\d_i\lrb{a_{k,i}}\xlongequal{\qquad}\!\! }{} -a_{i+1,i+1}\bigl[{a_{i+1,i+2}}^{\d_{i}}, a_{k,i+1}\bigr]_q +a_{i,i}a_{i+1,i+1}a_{k,i+2}\\
 \phantom{ \d_i\d_{i+1}\d_i\lrb{a_{k,i}} }{} \xlongequal{\lrb{\ref{equ:comm2}}} [\lr{a_{i+1,i+2}, a_{i,i+1}}_q -a_{i+1,i+1}a_{i,i+2}, a_{k,i}]_q+a_{k,i-1}a_{i+2,i+2}\\
 \phantom{ \d_i\d_{i+1}\d_i\lrb{a_{k,i}}\xlongequal{\qquad}\!\! }{} -a_{i,i}(\lr{a_{i+1,i+2}, a_{k,i+1}}_q-a_{i+1,i+1}a_{k,i+2})\\
 \phantom{ \d_i\d_{i+1}\d_i\lrb{a_{k,i}} }{} \xlongequal{\lrb{\ref{eqn:ri+1}}} -\bigl[{a_{i,i+1}}^{\d_{i+1}}, a_{k,i}\bigr]_q+a_{k,i-1}a_{i+2,i+2} +a_{i,i}{a_{k,i+1}}^{\d_{i+1}}\\
 \phantom{ \d_i\d_{i+1}\d_i\lrb{a_{k,i}} }{} \xlongequal[\eqref{eqn:ri+1}]{\lrb{\ref{eqn:ri}}} \d_{i+1}\d_i\d_{i+1}\lrb{a_{k,i}}, \\ \d_i\d_{i+1}\d_i\lrb{a_{k,i+1}}
 \xlongequal[\eqref{eqn:ri+1}]{\lrb{\ref{eqn:ri}}} [\lr{a_{i,i+1}, a_{i+1,i+2}}_q, a_{k,i+1}]_q -a_{i+1,i+1}\lr{a_{i,i+2}, a_{k,i+1}}_q \\
 \phantom{ \d_i\d_{i+1}\d_i\lrb{a_{k,i+1}}\xlongequal{\qquad} }{} -a_{i+2,i+2}\lr{a_{i,i+1}, a_{k,i}}_q +a_{k,i-1}a_{i+1,i+1}a_{i+2,i+2} +a_{i,i}a_{k,i+2}\\
 \phantom{ \d_i\d_{i+1}\d_i\lrb{a_{k,i+1}} }{} \xlongequal{\lrb{\ref{rela-3}}} [a_{i,i+1},\lr{a_{i+1,i+2}, a_{k,i+1}}_q]_q -a_{i+1,i+1} \bigl(\bigl[{a_{i,i+1}}^{\d_{i+1}}, a_{k,i}\bigr]_q \\
 \phantom{ \d_i\d_{i+1}\d_i\lrb{a_{k,i+1}}\xlongequal{\qquad} }{} -a_{i,i} {a_{k,i+1}}^{\d_{i+1}} +a_{i,i+1}a_{k,i+2} \bigr) -a_{i+2,i+2}(\lr{a_{i,i+1}, a_{k,i}}_q \\
 \phantom{ \d_i\d_{i+1}\d_i\lrb{a_{k,i+1}}\xlongequal{\qquad} }{}+a_{k,i-1}a_{i+1,i+1}a_{i+2,i+2}) +a_{i,i}a_{k,i+2}\\
 \phantom{ \d_i\d_{i+1}\d_i\lrb{a_{k,i+1}}}{} \xlongequal[\eqref{eqn:ri+1}]{\lrb{\ref{rela-2}}} \bigl[\bigl[{a_{i,i+1}}^{\d_{i+1}}, a_{i+1,i+2}\bigr]_q-a_{i,i}a_{i+1,i+1} -a_{i+2,i+2}a_{i,i+2}, {a_{k,i+1}}^{\d_{i+1}}\bigr]_q \\
 \phantom{ \d_i\d_{i+1}\d_i\lrb{a_{k,i+1}}\xlongequal{\qquad} }{} -a_{i+1,i+1}\bigl[{a_{i,i+1}}^{\d_{i+1}},a_{k,i}\bigr]_q+a_{k,i-1}a_{i+1,i+1}a_{i+2,i+2} \\
 \phantom{ \d_i\d_{i+1}\d_i\lrb{a_{k,i+1}}\xlongequal{\qquad} }{} +a_{i,i}{a_{k,i+1}}^{\d_{i+1}}a_{i+1,i+1}+a_{i,i}{a_{k,i+2}}^{\d_{i+1}}\\
 \phantom{ \d_i\d_{i+1}\d_i\lrb{a_{k,i+1}}}{} \xlongequal{\quad} \bigl[\bigl[{a_{i,i+1}}^{\d_{i+1}}, a_{i+1,i+2}\bigr]_q, {a_{k,i+1}}^{\d_{i+1}}\bigr]_q -a_{i,i}a_{i+1,i+1}{a_{k,i+1}}^{\d_{i+1}} \\
 \phantom{ \d_i\d_{i+1}\d_i\lrb{a_{k,i+1}}\xlongequal{\quad} }{} -a_{i+2,i+2} \bigl[a_{i,i+2}, {a_{k,i+1}}^{\d_{i+1}} \bigr]_q -a_{i+1,i+1} \bigl[{a_{i,i+1}}^{\d_{i+1}}, a_{k,i} \bigr]_q \\
 \phantom{ \d_i\d_{i+1}\d_i\lrb{a_{k,i+1}}\xlongequal{\quad} }{} +a_{k,i-1}a_{i+1,i+1}a_{i+2,i+2} +a_{i,i}{a_{k,i+1}}^{\d_{i+1}} a_{i+1,i+1}+a_{i,i}{a_{k,i+2}}^{\d_{i+1}}\\
 \phantom{ \d_i\d_{i+1}\d_i\lrb{a_{k,i+1}} }{} \xlongequal{\lrb{\ref{eqn:ri}}} \d_{i+1}\bigl(-\bigl[{a_{i+1,i+2}}^{\d_{i}}, a_{k,i+1}\bigr]_q+{a_{k,i}}^{\d_{i}}a_{i+2,i+2} +a_{i,i}a_{k,i+2}\bigr)\\
 \phantom{ \d_i\d_{i+1}\d_i\lrb{a_{k,i+1}} }{} \xlongequal[\eqref{eqn:ri+1}]{\lrb{\ref{eqn:ri}}} \d_{i+1}\d_i\d_{i+1}\lrb{a_{k,i+1}},
 \\
 \d_i\d_{i+1}\d_i\lrb{a_{i+1,j}}
\xlongequal[\eqref{eqn:ri+1}]{\lrb{\ref{eqn:ri}}} \bigl[\bigl[ {a_{i+1,i+2}}^{\d_{i}}, a_{i,i+1}\bigr]_q,{a_{i+2,j}}^{\d_{i+1}}\bigr]_q -a_{i,i}\bigl[a_{i,i+2}, {a_{i+2,j}}^{\d_{i+1}}\bigr]_q \\
 \phantom{ \d_i\d_{i+1}\d_i\lrb{a_{i+1,j}}\xlongequal{\qquad}\!\!}{} -a_{i+1,i+1}[ {a_{i+1,i+2}}^{\d_{i}}, a_{i+2,j}]_q +a_{i,i}a_{i+1,i+1}a_{i+3,j}+a_{i+2,i+2}a_{i,j}\\
 \phantom{ \d_i\d_{i+1}\d_i\lrb{a_{i+1,j}}}{} \xlongequal{\quad}
\bigl[\bigl[{a_{i+1,i+2}}^{\d_{i}}, a_{i,i+1}\bigr]_q -a_{i,i}a_{i,i+2}, {a_{i+2,j}}^{\d_{i+1}}\bigr]_q \\
 \phantom{ \d_i\d_{i+1}\d_i\lrb{a_{i+1,j}}\xlongequal{\quad}}{} -a_{i+1,i+1}\bigl[{a_{i+1,i+2}}^{\d_{i}}, a_{i+2,j}\bigr]_q +a_{i,i}a_{i+1,i+1}a_{i+3,j}+a_{i+2,i+2}a_{i,j}\\
 \phantom{ \d_i\d_{i+1}\d_i\lrb{a_{i+1,j}}}{}\xlongequal{\lrb{\ref{rela-2}}} -\bigl[a_{i+1,i+2}, {a_{i+2,j}}^{\d_{i+1}}\bigr]_q +a_{i+1,i+1}a_{i+2,i+2}{a_{i+2,j}}^{\d_{i+1}} \\
 \phantom{ \d_i\d_{i+1}\d_i\lrb{a_{i+1,j}}\xlongequal{\qquad}}{} -a_{i+1,i+1}\bigl[{a_{i+1,i+2}}^{\d_{i}}, a_{i+2,j}\bigr]_q +a_{i,i}a_{i+1,i+1}a_{i+3,j}+a_{i+2,i+2}a_{i,j}\\
 \phantom{ \d_i\d_{i+1}\d_i\lrb{a_{i+1,j}}}{} \xlongequal{\lrb{\ref{eqn:ri+1}}} [a_{i+1,i+2},\lr{a_{i,i+1}, a_{i+1,j}}_q]_q -a_{i,i}\lr{a_{i+1,i+2}, a_{i+2,j}}_q \\
 \phantom{ \d_i\d_{i+1}\d_i\lrb{a_{i+1,j}}\xlongequal{\qquad}\!\!}{} +a_{i+1,i+1} \bigl(a_{i+2,i+2}{a_{i+2,j}}^{\d_{i+1}}-a_{i+1,i+2}a_{i,j}-\bigl[{a_{i+1,i+2}}^{\d_{i}}, a_{i+2,j}\bigr]_q \bigr) \\
 \phantom{ \d_i\d_{i+1}\d_i\lrb{a_{i+1,j}}\xlongequal{\qquad}\!\!}{} +a_{i,i}a_{i+1,i+1}a_{i+3,j}+a_{i+2,i+2}a_{i,j}\\
 \phantom{ \d_i\d_{i+1}\d_i\lrb{a_{i+1,j}}}{} \xlongequal{\lrb{\ref{rela-3}}} [\lr{a_{i+1,i+2}, a_{i,i+1}}_q, a_{i+1,j}]_q -a_{i,i}\lr{a_{i+1,i+2}, a_{i+2,j}}_q \\
 \phantom{ \d_i\d_{i+1}\d_i\lrb{a_{i+1,j}}\xlongequal{\qquad}}{} - a_{i+1,i+1}\lr{a_{i,i+2}, a_{i+1,j}}_q +a_{i,i}a_{i+1,i+1}a_{i+3,j}+a_{i+2,i+2}a_{i,j}\\
 \phantom{ \d_i\d_{i+1}\d_i\lrb{a_{i+1,j}}}{} \xlongequal[\eqref{eqn:ri+1}]{\lrb{\ref{eqn:ri}}} \d_{i+1}\d_i\d_{i+1}\lrb{a_{i+1,j}},\\
 \d_i\d_{i+1}\d_i\lrb{a_{i+2,j}}
\xlongequal[\eqref{eqn:ri+1}]{\lrb{\ref{eqn:ri}}} -[{a_{i+1,i+2}}^{\d_{i}}, a_{i+2,j}]_q+a_{i,i}a_{i+3,j}+a_{i+2,i+2}{a_{i+2,j}}^{\d_{i+1}}\\
 \phantom{ \d_i\d_{i+1}\d_i\lrb{a_{i+2,j}}}{}\xlongequal{\lrb{\ref{rela-3}}} [\lr{a_{i,i+1}, a_{i+1,i+2}}_q, a_{i+2,j}]_q -a_{i+1,i+1}\lr{a_{i,i+2}, a_{i+2,j}}_q +a_{i,i}a_{i+3,j} \\
 \phantom{ \d_i\d_{i+1}\d_i\lrb{a_{i+2,j}}\xlongequal{\qquad}}{} -a_{i+2,i+2}\lr{a_{i,i+1}, a_{i+1,j}}_q+a_{i+1,i+1}a_{i+2,i+2}a_{i,j}\\
 \phantom{ \d_i\d_{i+1}\d_i\lrb{a_{i+2,j}}}{}\xlongequal{\lrb{\ref{rela-3}}} [a_{i,i+1},\lr{a_{i+1,i+2}, a_{i+2,j}}_q]_q -a_{i+1,i+1} \bigl(a_{i,i+1}a_{i+3,j} -a_{i,i}{a_{i+2,j}}^{\d_{i+1}} \\
 \phantom{ \d_i\d_{i+1}\d_i\lrb{a_{i+2,j}}\xlongequal{\qquad}}{} + \bigl[{a_{i,i+1}}^{\d_{i+1}}, a_{i+1,j} \bigr]_q \bigr)-a_{i+2,i+2}\lr{a_{i,i+1}, a_{i+1,j}}_q +a_{i,i}a_{i+3,j} \\
 \phantom{ \d_i\d_{i+1}\d_i\lrb{a_{i+2,j}}\xlongequal{\qquad}}{} +a_{i+1,i+1}a_{i+2,i+2}a_{i,j}\\
 \phantom{ \d_i\d_{i+1}\d_i\lrb{a_{i+2,j}}}{}\xlongequal{\lrb{\ref{eqn:ri+1}}} -\bigl[a_{i,i+1},{a_{i+2,j}}^{\d_{i+1}}\bigr]_q +a_{i,i}a_{i+1,i+1}{a_{i+2,j}}^{\d_{i+1}} +a_{i,i}a_{i+3,j} \\
 \phantom{ \d_i\d_{i+1}\d_i\lrb{a_{i+2,j}}\xlongequal{\qquad}\!\!}{} -a_{i+1,i+1}\bigl[ {a_{i,i+1}}^{\d_{i+1}}, a_{i+1,j}\bigr]_q+a_{i+1,i+1}a_{i+2,i+2}a_{i,j}\\
 \phantom{ \d_i\d_{i+1}\d_i\lrb{a_{i+2,j}}}{} \xlongequal[\eqref{eqn:ri+1}]{\lrb{\ref{rela-2}}} \bigl[\bigl[{a_{i,i+1}}^{\d_{i+1}}, a_{i+1,i+2}\bigr]_q-a_{i+2,i+2}a_{i,i+2}, {a_{i+2,j}}^{\d_{i+1}}\bigr]_q +a_{i,i}a_{i+3,j} \\
 \phantom{ \d_i\d_{i+1}\d_i\lrb{a_{i+2,j}}\xlongequal{\qquad}}{} -a_{i+1,i+1}\bigl[ {a_{i,i+1}}^{\d_{i+1}}, a_{i+1,j}\bigr]_q +a_{i+1,i+1}a_{i+2,i+2}a_{i,j}\\
 \phantom{ \d_i\d_{i+1}\d_i\lrb{a_{i+2,j}}}{} \xlongequal{\lrb{\ref{eqn:ri}}} \d_{i+1}\bigl(-\bigl[{a_{i+1,i+2}}^{\d_{i}}, a_{i+2,j}\bigr]_q+a_{i,i}a_{i+3,j}+a_{i+2,i+2} {a_{i+2,j}}^{\d_{i+1}}\bigr)\\
 \phantom{ \d_i\d_{i+1}\d_i\lrb{a_{i+2,j}}}{}\xlongequal[\eqref{eqn:ri+1}]{\lrb{\ref{eqn:ri}}} \d_{i+1}\d_i\d_{i+1}
\lrb{a_{i+2,j}}.
\end{gather*}

Hence, $\d_i$ satisfy the relation \eqref{eqn4-1}.
Now, let us prove that $\d_i$ satisfy the relation \eqref{eqn4-2}.

(a) If one of $i$ and $j$ is equal to $0$, let us say $i=0$, then $j\ge 2$.
If $ k \neq 1, j+1$, and $ l \neq j$, we have
$ \d_0\d_j\lrb{a_{k,l}}=a_{k,l}=\d_j\d_0\lrb{a_{k,l}}$.
If $ k = 1$ or $k=j+1$ or $l \neq j$, we have
\begin{gather*}
 \d_0\d_j\lrb{a_{1,1}}=\d_0\lrb{a_{1,1}}=a_{1,n}=\d_j\lrb{a_{1,n}}=\d_j\d_0\lrb{a_{1,1}},\\
 \d_0\d_j\lrb{a_{1,n}}=\d_0\lrb{a_{1,n}}=a_{1,1}=\d_j\lrb{a_{1,1}}=\d_j\d_0\lrb{a_{1,n}},\\
 \d_0\d_j\lrb{a_{1,m}}=\d_0\lrb{a_{1,m}}={a_{1,m}}^{\d_{0}}
=\d_j\bigl({a_{1,m}}^{\d_{0}}\bigr)=\d_j\d_0\lrb{a_{1,m}},\\
 \d_0\d_j\lrb{a_{j+1,j+1}}=\d_0\lrb{a_{j,j}}=a_{j,j}=\d_j\lrb{a_{j+1,j+1}}=\d_j\d_0 \lrb{a_{j+1,j+1}},\\
 \d_0\d_j\lrb{a_{j+1,l}}=\d_0\bigl({a_{j+1,l}}^{\d_{j}}\bigr)
={a_{j+1,l}}^{\d_{j}}=\d_j\lrb{a_{j+1,l}}=\d_j\d_0\lrb{a_{j+1,l}},\\
 \d_0\d_j\lrb{a_{j,j}}=\d_0\lrb{a_{j+1,j+1}}=a_{j+1,j+1}=\d_j\lrb{a_{j,j}} =\d_j\d_0\lrb{a_{j,j}},\\
 \d_0\d_j\lrb{a_{k,j}}=\d_0\bigl({a_{k,j}}^{\d_{j}}\bigr)
=\d_j\lrb{a_{k,j}}=\d_j\d_0\lrb{a_{k,j}}.
\end{gather*}

(b) If $0<i <j$, then if $ k \neq i+1, j+1$, and $l \neq i, j$, we have
$\d_i\d_j\lrb{a_{k,l}}=a_{k,l}=\d_j\d_i\lrb{a_{k,l}}$.
If $k=i+1$ or $k=j+1$ or $l=i$ or $l=j$, we have
\begin{gather*}
 \d_i\d_j\lrb{a_{i,i}}=\d_i\lrb{a_{i,i}}=a_{i+1,i+1}=\d_j\lrb{a_{i+1,i+1}} =\d_j\d_i\lrb{a_{i,i}},\\
 \d_i\d_j\lrb{a_{i+1,i+1}}=\d_i\lrb{a_{i+1,i+1}}=a_{i,i}=\d_j\lrb{a_{i,i}} =\d_j\d_i\lrb{a_{i+1,i+1}},\\
 \d_i\d_j\lrb{a_{j,j}}=\d_i\lrb{a_{j+1,j+1}}=a_{j+1,j+1}=\d_j\lrb{a_{j,j}} =\d_j\d_i\lrb{a_{j,j}},\\
 \d_i\d_j\lrb{a_{j+1,j+1}}=\d_i\lrb{a_{j,j}}=a_{j,j}=\d_j\lrb{a_{j+1,j+1}} =\d_j\d_i\lrb{a_{j+1,j+1}},\\
 \d_i\d_j\lrb{a_{k,i}}=\d_i\lrb{a_{k,i}}={a_{k,i}}^{\d_{i}}
=\d_j\bigl({a_{k,i}}^{\d_{i}}\bigr)=\d_j\d_i\lrb{a_{k,i}},\\
 \d_i\d_j\lrb{a_{k,j}}=\d_i\bigl({a_{k,j}}^{\d_{j}} \bigr)
={a_{k,j}}^{\d_{j}}=\d_j\lrb{a_{k,j}} =\d_j\d_i\lrb{a_{k,j}},\\
 \d_i\d_j\lrb{a_{i+1,l}}=\d_i\lrb{a_{i+1,l}}={a_{i+1,l}}^{\d_{i}}
=\d_j\bigl({a_{i+1,l}}^{\d_{i}}\bigr)=\d_j\d_i\lrb{a_{i+1,l}},\\
 \d_i\d_j\lrb{a_{j+1,l}}=\d_i\bigl({a_{j+1,l}}^{\d_{j}}\bigr)={a_{j+1,l}}^{\d_{j}}
=\d_j\lrb{a_{j+1,l}}=\d_j\d_i\lrb{a_{j+1,l}}.
\end{gather*}
Up to now, $\d_i$ also satisfy the relation \eqref{eqn4-2}.

In summary, the proof is finished.
\end{proof}

\section{Conclusion}

In the paper, we offer an equivalent perspective on defining the higher-rank Askey--Wilson algebras $\A(n)$ provided by the authors in \cite{CFPR(2023)}. We also write down a series of automorphisms in our settings, which coincide with those in \cite{CFPR(2023)}. The detailed proofs are done here that satisfy the braid group relations in $\A(n)$.

In our future work, we will explore the PBW basis of $\A(n)$ and investigate the theory of their representations.
It is interesting to find a more intuitive geometric interpretation of $\A(n)$.

\subsection*{Acknowledgements}
The authors are very grateful to the anonymous referees for their careful reviewing the man\-u\-script and many valuable comments, which significantly enhance the quality of the paper. This work is supported by the National Natural Science Foundation of China (Grant No.\ 12471038).

\pdfbookmark[1]{References}{ref}
\LastPageEnding


\begin{thebibliography}{99}
\footnotesize\itemsep=0pt

\bibitem{BK(2005)}
Baseilhac P., Koizumi K., A deformed analogue of {O}nsager's symmetry in
 the~{$XXZ$} open spin chain, \href{https://doi.org/10.1088/1742-5468/2005/10/p10005}{\textit{J.~Stat. Mech. Theory Exp.}}
 \textbf{2005} (2005), P10005, 15~pages, \href{https://arxiv.org/abs/hep-th/0507053}{arXiv:hep-th/0507053}.

\bibitem{CL(2022)}
Cooke J., Lacabanne A., Higher rank {A}skey--{W}ilson algebras as skein
 algebras, \href{https://arxiv.org/abs/2205.04414}{arXiv:2205.04414}.

\bibitem{CFGPRV(2021)}
Cramp\'e N., Frappat L., Gaboriaud J., d'Andecy L.P., Ragoucy E., Vinet L., The
 {A}skey--{W}ilson algebra and its avatars, \href{https://doi.org/10.1088/1751-8121/abd783}{\textit{J.~Phys.~A}} \textbf{54}
 (2021), 063001, 32~pages, \href{https://arxiv.org/abs/2009.14815}{arXiv:2009.14815}.

\bibitem{CFPR(2023)}
Cramp\'e N., Frappat L., Poulain~d'Andecy L., Ragoucy E., The higher-rank
 {A}skey--{W}ilson algebra and its braid group automorphisms, \href{https://doi.org/10.3842/SIGMA.2023.077}{\textit{SIGMA}}
 \textbf{19} (2023), 077, 36~pages, \href{https://arxiv.org/abs/2303.17677}{arXiv:2303.17677}.

\bibitem{CGVZ(2020)}
Cramp\'e N., Gaboriaud J., Vinet L., Zaimi M., Revisiting the {A}skey--{W}ilson
 algebra with the universal {$R$}-matrix of~{$\mathrm{U}_q(\mathfrak{sl}_2)$},
 \href{https://doi.org/10.1088/1751-8121/ab604e}{\textit{J.~Phys.~A}} \textbf{53} (2020), 05LT01, 10~pages,
 \href{https://arxiv.org/abs/1908.04806}{arXiv:1908.04806}.

\bibitem{CVZ(2021)}
Cramp\'e N., Vinet L., Zaimi M., Temperley--{L}ieb,
 {B}irman--{M}urakami--{W}enzl and {A}skey--{W}ilson algebras and other
 centralizers of~{$U_q(\mathfrak{sl}_2)$}, \href{https://doi.org/10.1007/s00023-021-01064-x}{\textit{Ann. Henri Poincar\'e}}
 \textbf{22} (2021), 3499--3528, \href{https://arxiv.org/abs/2008.04905}{arXiv:2008.04905}.

\bibitem{DDV(2020)}
De~Bie H., De~Clercq H., van~de Vijver W., The higher rank {$q$}-deformed
 {B}annai--{I}to and {A}skey--{W}ilson algebra, \href{https://doi.org/10.1007/s00220-019-03562-w}{\textit{Comm. Math. Phys.}}
 \textbf{374} (2020), 277--316, \href{https://arxiv.org/abs/1805.06642}{arXiv:1805.06642}.

\bibitem{DV(2020)}
De~Bie H., van~de Vijver W., A discrete realization of the higher rank {R}acah
 algebra, \href{https://doi.org/10.1007/s00365-019-09475-0}{\textit{Constr. Approx.}} \textbf{52} (2020), 1--29,
 \href{https://arxiv.org/abs/1808.10520}{arXiv:1808.10520}.

\bibitem{De(2019)}
De~Clercq H., Higher rank relations for the {A}skey--{W}ilson and
 {$q$}-{B}annai--{I}to algebra, \href{https://doi.org/10.3842/SIGMA.2019.099}{\textit{SIGMA}} \textbf{15} (2019), 099,
 32~pages, \href{https://arxiv.org/abs/1908.11654}{arXiv:1908.11654}.

\bibitem{GW(2023)}
Groenevelt W., Wagenaar C., An {A}skey--{W}ilson algebra of rank~2,
 \href{https://doi.org/10.3842/SIGMA.2023.008}{\textit{SIGMA}} \textbf{19} (2023), 008, 35~pages, \href{https://arxiv.org/abs/2206.03986}{arXiv:2206.03986}.

\bibitem{Huang(2015)}
Huang H.-W., Finite-dimensional irreducible modules of the universal
 {A}skey--{W}ilson algebra, \href{https://doi.org/10.1007/s00220-015-2467-9}{\textit{Comm. Math. Phys.}} \textbf{340} (2015),
 959--984, \href{https://arxiv.org/abs/1210.1740}{arXiv:1210.1740}.

\bibitem{Huang(2021)}
Huang H.-W., Finite-dimensional irreducible modules of the universal
 {A}skey--{W}ilson algebra at roots of unity, \href{https://doi.org/10.1016/j.jalgebra.2020.11.012}{\textit{J.~Algebra}} \textbf{569}
 (2021), 12--29, \href{https://arxiv.org/abs/1906.01776}{arXiv:1906.01776}.

\bibitem{KS(2001)}
Koelink E., Stokman J.V., The {A}skey--{W}ilson function transform,
 \href{https://doi.org/10.1155/S1073792801000575}{\textit{Internat. Math. Res. Notices}} \textbf{2001} (2001), 1203--1227,
 \href{https://arxiv.org/abs/math.CA/0004053}{arXiv:math.CA/0004053}.

\bibitem{Koor(2007)}
Koornwinder T.H., The relationship between {Z}hedanov's algebra {${\rm AW}(3)$}
 and the double affine {H}ecke algebra in the rank one case, \href{https://doi.org/10.3842/SIGMA.2007.063}{\textit{SIGMA}}
 \textbf{3} (2007), 063, 15~pages, \href{https://arxiv.org/abs/math.QA/0612730}{arXiv:math.QA/0612730}.

\bibitem{Koor(2008)}
Koornwinder T.H., Zhedanov's algebra {$\rm AW(3)$} and the double affine
 {H}ecke algebra in the rank one case.~{II}. {T}he spherical subalgebra,
 \href{https://doi.org/10.3842/SIGMA.2008.052}{\textit{SIGMA}} \textbf{4} (2008), 052, 17~pages, \href{https://arxiv.org/abs/0711.2320}{arXiv:0711.2320}.

\bibitem{Lav(1997)}
Lavrenov A., On {A}skey--{W}ilson algebra, \href{https://doi.org/10.1023/A:1022821531517}{\textit{Czechoslovak~J. Phys.}}
 \textbf{47} (1997), 1213--1219.

\bibitem{PW(2017)}
Post S., Walter A., A higher rank extension of the {A}skey--{W}ilson algebra,
 \href{https://arxiv.org/abs/1705.01860}{arXiv:1705.01860}.

\bibitem{Ter(2011)}
Terwilliger P., The universal {A}skey--{W}ilson algebra, \href{https://doi.org/10.3842/SIGMA.2011.069}{\textit{SIGMA}}
 \textbf{7} (2011), 069, 24~pages, \href{https://arxiv.org/abs/1104.2813}{arXiv:1104.2813}.

\bibitem{TV(2004)}
Terwilliger P., Vidunas R., Leonard pairs and the {A}skey--{W}ilson relations,
 \href{https://doi.org/10.1142/S0219498804000940}{\textit{J.~Algebra Appl.}} \textbf{3} (2004), 411--426, \href{https://arxiv.org/abs/math.QA/0305356}{arXiv:math.QA/0305356}.

\bibitem{Zhe(1991)}
Zhedanov A.S., ``{H}idden symmetry'' of {A}skey--{W}ilson polynomials,
 \href{https://doi.org/10.1007/BF01015906}{\textit{Theoret. and Math. Phys.}} \textbf{89} (1991), 1146--1157.

\end{thebibliography}
\end{document}